\documentclass[12pt]{amsart}
\usepackage{amsmath}
\usepackage{amsthm}
\usepackage{amstext}
\usepackage{amssymb}
\usepackage{amsfonts}
\usepackage{young}
\usepackage{multicol}
\usepackage{mathrsfs}
\usepackage[vcentermath]{youngtab}
\usepackage[mathscr]{euscript}
\usepackage{tikz}
\usetikzlibrary{matrix,arrows,decorations.pathmorphing} 
\usepackage{ytableau}

\newtheorem{theorem}{Theorem}[section]
\newtheorem*{theorem*}{Theorem}
\newtheorem{lemma}[theorem]{Lemma}

\newtheorem{proposition}[theorem]{Proposition}
\newtheorem{corollary}[theorem]{Corollary}
\newtheorem{definition}[theorem]{Definition}

\newtheorem{problem}[theorem]{Problem}

\theoremstyle{definition}
\newtheorem{example}[theorem]{Example}
\newtheorem{remark}[theorem]{Remark}

\def\r{{\mathfrak row}}

\def\QSym{{\rm QSym}}
\def\mQSym{{\mathfrak{m} \rm QSym}}

\DeclareMathAlphabet{\mathpzc}{OT1}{pzc}{m}{it}

\title{
Dual filtered graphs
}

\numberwithin{equation}{section}

\begin{document}

\author{Rebecca Patrias}
\address{\hspace{-.3in} Department of Mathematics, University of Minnesota,
127 Vincent Hall, 206 Church Street, Minneapolis, MN 55455, USA}
\email{patri080@umn.edu}\thanks{R.P. was supported by NSF grant DMS-1148634.}

\author{Pavlo Pylyavskyy}
\address{Department of Mathematics, University of Minnesota,
127 Vincent Hall, 206 Church Street, Minneapolis, MN 55455, USA}
\email{ppylyavs@umn.edu}\thanks{P.P. was supported by NSF grants DMS-1068169, DMS-1148634, DMS-1351590, and Sloan Fellowship.}

\date{\today
}

\thanks{
}


\begin{abstract}
We define a $K$-theoretic analogue of Fomin's dual graded graphs, which we call dual filtered graphs. The key formula in the definition is $DU-UD= D + I$. 
Our major examples are $K$-theoretic analogues of Young's lattice, of shifted Young's lattice, and of the Young-Fibonacci lattice. 
We suggest notions of tableaux, insertion algorithms, and growth rules whenever such objects are not already present in the literature. (See the table below.) 
We also provide a large number of other examples. 
Most of our examples arise via two constructions, which we call the Pieri construction and the M\"obius construction. 
The Pieri construction is closely related to the construction of dual graded graphs 
from a graded Hopf algebra, as described in \cite{BLL,HN,LS}. The M\"obius construction is more mysterious 
but also potentially more important, as it corresponds to natural insertion algorithms. 
\end{abstract}

\ \vspace{-.1in}

\maketitle

\begin{center}
\begin{tabular}{|l| p{4.8cm}| p{5cm}| c|}
\hline
$ $ & Tableaux & Insertion & Growth\\[.5ex]
\hline
Young & Standard Young tableaux \cite{Young} & RSK insertion \cite{Robinson,Schensted,Knuth}&\cite{Fomin} \\[.5ex]
\hline
Shifted Young & Standard shifted Young tableaux with and without circles \cite{Sagan}
& Shifted Robinson-Schensted insertion \cite{Sagan,W}
& \cite{FSchen}\\[.5ex]
\hline
Young-Fibonacci & Young-Fibonacci tableaux \cite{Fomin,F} & Young-Fibonacci insertion \cite{Fomin} & \cite{FSchen} \\[.5ex]
\hline
M\"obius Young & Increasing and set-valued tableaux \cite{ThY,B}& Hecke insertion \cite{BKSTY} & Section \ref{sec:HeckeGrowth}\\[.5ex]
\hline 
M\"obius shifted Young & Definitions \ref{def:increasingshifted} and \ref{def:shiftedset} & Section \ref{sec:shiftedinsertion} & Section \ref{sec:shiftedgrowth}\\[.5ex]
\hline
M\"obius Young-Fibonacci & Definitions \ref{def:KYFtableau} and \ref{def:KYFsetvalued} & Section \ref{sec:KYFinsertion} & Section \ref{sec:KYFgrowth} \\[.5ex]
\hline
\end{tabular}
\end{center}

\

\tableofcontents

\newpage

\section{Introduction}

Fomin's {\it {dual graded graphs}} \cite{F}, as well as their predecessors - Stanley's {\it {differential posets}} \cite{St},
 were invented as a tool to better understand the Robinson-Schensted insertion algorithm. Dual graded graphs are significant in the areas where the
Robinson-Schensted correspondence appears, for example in Schubert calculus or in the study of representations of towers of algebras \cite{BLL, BS}. 
A recent work of Lam and Shimozono \cite{LSh} associates a dual graded graph to any Kac-Moody algebra, bringing their study to a new level of generality.

\subsection{Weyl algebra and its deformations}
One way to view the theory is 
as a study of certain {\it {combinatorial representations}} of the {\it {first Weyl algebra}}. Let us briefly recall the definitions. The
{\it {Weyl algebra}}, or the first Weyl algebra, is an algebra over some field $K$ (usually $K = {\mathbb {R}}$) generated by two generators $U$ and $D$ with a 
single relation $DU-UD=1$. It was originally introduced by Hermann Weyl in his study of quantum mechanics. We refer the reader to \cite{Bjo}, for example,
for more background on the Weyl algebra.

A {\it {graded graph}} is a triple $G=(P,\rho,E)$, where $P$ is a discrete set of vertices, $\rho:P\to \mathbb{Z}$ is 
a rank function, and $E$ is a multiset of edges/arcs $(x,y)$, where $\rho(y)=\rho(x)+1$. In other words, each vertex 
is assigned a rank, and edges can only join vertices in successive ranks. For the set of vertices $P$, let 
$P_n$ denote the subset of vertices of rank $n$. For any field $K$ of characteristic zero, the formal linear combinations of vertices of $G$ form the vector space 
$KP$. 

Let $G_1=(P,\rho,E_1)$ and $G_2=(P,\rho,E_2)$ be a pair of graded graphs with a common vertex set and rank 
function. From this pair, define an \textit{oriented graded graph} $G=(P,\rho, E_1,E_2)$ by orienting the 
edges of $G_1$ in the direction of increasing rank and the edges of $G_2$ in the direction of decreasing rank. 
Let $a_i(x,y)$ denote the number of edges of $E_i$ joining $x$ and $y$ or the multiplicity of edge $(x,y)$ in $E_i$.
We define the \textit{up} and \textit{down operators} $U,D\in End(KP)$ associated with graph $G$ by 
$$Ux=\sum_{y} a_1(x,y)y$$ and $$Dy=\sum_x a_2(x,y) x.$$

Graded graphs $G_1$ and $G_2$ with a common vertex set and rank function are said to be \textit{dual} if 
$$DU-UD=I,$$ where $I$ is an identity operator acting on $KP$. Thus, we see that $KP$ is a representation of the Weyl algebra.
Furthermore, it is a representation of a very special kind, where $U$ and $D$ act in a particularly nice combinatorial way on a fixed basis.  

One would then expect that variations of the Weyl algebra would correspond to some variations of the theory of dual graded graphs. One such variation is the 
{\it {$q$-Weyl algebra}} defined by the relation $$DU-qUD=1.$$ The corresponding theory of {\it {quantum dual graded graphs}} was pursued by Lam in \cite{L}.

In this paper, we shall study the theory arising from an analogue $W$ of the Weyl algebra with defining relation $$DU-UD=D+1.$$ 
We shall refer to it as the Ore algebra \cite{O}, and we shall see that it is a very natural object. 
In particular, the corresponding theory of {\it {dual filtered graphs}} fits naturally with the existing body of work
on the $K$-theory of Grassmannians.

\begin{remark}\label{rem:rescale}
It is easy to see that by rescaling $U$ and $D$, we can pass from arbitrary $DU-UD= \alpha D + \beta$ to $DU-UD= D + I$. We thus suffer no loss in generality 
by writing down the relation in the latter form.
\end{remark}

\subsection{Differential vs difference operators}\label{sec:differential}

The following provides some intuitive explanation for the relation between representations of Weyl and Ore algebras.
The simplest representation of the Weyl algebra is that acting on a polynomial ring. Indeed, consider polynomial ring $R = \mathbb R[x]$, and for $f \in R$ let 
$$U(f) = xf, \;\;\; D(f) = \frac{\partial f}{\partial x}.$$
It is an easy exercise to verify that this indeed produces a representation of the Weyl algebra. In fact, this is how it is often defined.

Let us now redefine the down operator to be the {\it {difference operator}}:
$$U(f) = xf, \;\;\; D(f) = f(x+1) - f(x).$$
\begin{lemma}
 This gives a representation of the Ore algebra $W$. In other words, those operators satisfy 
$$DU-UD = D+I,$$
where $I$ is the identity map on $R$.
\end{lemma}

\begin{proof}
 We have 
$$(DU-UD)(f) = ((x+1)f(x+1)-xf(x)) - x(f(x+1)-f(x))$$ $$= f(x+1) = f(x) + (f(x+1)-f(x)) = (I+D)(f).$$
\end{proof}

It can also be noted that differential and difference operators can be related as follows: $$D = e^{\frac{\partial}{\partial x}}-1.$$ We omit the easy proof for brevity.   
As we shall see in Example \ref{ex:polys}, this representation of $W$ corresponds to a very basic dual filtered graph. 

\subsection{Pieri and M\"obius constructions}

We make use of two conceptual ways to build examples of dual filtered graphs. The first construction starts with an algebra $A$, a derivation $D$ on $A$, and an element $f \in A$ such that $D(f)=f+1$. We often build the desired derivation $D$ using a
{\it {bialgebra}} structure on $A$. This construction is very close to an existing one in the literature \cite{BLL,HN,LS}, where dual graded graphs are constructed 
from graded Hopf algebras. In fact, if we were to require $D(f)=1$, we would get dual graded graphs instead of dual filtered graphs in our construction. We refer to this 
method as the {\it {Pieri construction}} or sometimes as the {\it{Pieri deformation}}.

Instead, we can also start with an existing dual graded graph $G=(P, \rho, E_1, E_2)$ (see definition below), composed of graphs $G_1=(P,\rho, E_1)$ and $G_2=(P,\rho,e_2)$, and adjust $E_1$ and $E_2$ in the following manner. 
To obtain $G_1'$, we add $\#\{x|y\text{ covers }x\text{ in }G_1\}$ loops at each vertex $y \in P$ to $E_1$. As for $G_2'$, we create a new edge set $E'_2$ by forming 
$$a'_2(x,y) = |\mu(x,y)|$$ edges between vertices $x$ and $y$, where $\mu$ denotes the M\"obius function in $G_2=(P, \rho, E_2)$.
We refer to this construction as the {\it {M\"obius construction}} or {\it{M\"obius deformation}}. Note that this does not always produce a pair of dual filtered graphs, and it is mysterious to determine 
when it does. In some major examples, however, it is the result of this construction that relates to Robinson-Schensted-like algorithms, for example to {\it {Hecke insertion}}
of \cite{BKSTY}.

The following observation seems remarkable, and we call it the {\it {M\"obius via Pieri phenomenon}}. Let $A$ be a graded Hopf algebra, and let bialgebra $\tilde A$ be 
its {\it {$K$-theoretic deformation}} in some appropriate sense which we do not know how to formalize. Let $G$ be a natural dual graded graph associated with $A$.
What we observe is that applying the  M\"obius construction to $G$ yields {\it {the same result}} as a natural Pieri construction applied to $\tilde A$. In other words,
the following diagram commutes.  

\begin{center}
\begin{tikzpicture}
\node (A) at (-2,2) {$A$};
\node (tildeA) at (2,2) {$\tilde A$};
\node (G) at (-2,0) {$G$};
\node (tildeG) at (2,0) {$\tilde G$};
\path[->] (G) edge node [above] {\tiny{M\"obius construction}} (tildeG);
\path[->] (A) edge node [left] {$\substack{\text{Pieri}\\\text{construction}}$} (G);
\path[->] (tildeA) edge node [right] {$\substack{\text{Pieri}\\\text{construction}}$} (tildeG);
\path[->] (A) edge node [above] {\tiny{$K$ deformation}} (tildeA);
\end{tikzpicture}
\end{center}

This happens for Young's lattice, see Section \ref{ex:G}, and for the binary tree dual graded graph, see Section \ref{sec:bin}. We expect this phenomenon to also 
occur for the shifted Young's lattice. 

The crucial condition necessary for this phenomenon to be observed is for $A$ to be the associated graded algebra of $\tilde A$, but this is not always the case. For example, this is not 
the case for $K$-theoretic analogues of the Poirier-Reutenauer and Malvenuto-Reutenauer Hopf algebras, as described in Sections \ref{sec:PR}, \ref{sec:MR}. To put it simply,
the numbers of basis elements for the filtered components of $A$ and $\tilde A$ are distinct in those cases, thus there is no hope to obtain corresponding graphs one from the other via
M\"obius construction. Interestingly, those are exactly the examples we found where the M\"obius construction fails to produce a dual filtered graph.

\subsection{Synopsis of the paper}
In Section \ref{sec:dgg}, we recall the definition of dual graded graphs from \cite{F} as well as three major examples: Young's lattice, Young-Fibonacci lattice, and 
shifted Young's lattice. We then remind the reader of the definition of the Robinson-Schensted algorithm and how one can obtain it locally via the machinary of 
{\it {growth diagrams}}, also introduced by Fomin in \cite{F, FSchen}.

In Section \ref{sec:dfg}, we formulate our definition of dual filtered graphs. We then introduce the trivial, Pieri and M\"obius constructions. 

In Section \ref{sec:Y}, we build the Pieri and M\"obius deformations of Young's lattice. We recall the definition of Hecke insertion and observe it is a map into a pair of paths
in the M\"obius deformation of Young's lattice. We provide growth rules that realize Hecke insertion. We also show how the Pieri construction applied to the ring generated by 
{\it {Grothendieck polynomials}} yields the same result as the M\"obius construction applied to Young's lattice. Thus we demonstrate an instance of the M\"obius via Pieri 
phenomenon.  

In Section \ref{sec:sY},  we build Pieri and M\"obius deformations of the shifted Young's lattice. We introduce {\it {shifted Hecke insertion}} and remark that 
its result always coincides with that of $K$-theoretic jeu de taquin rectification as defined in \cite{ThYshift}. We also provide the corresponding growth rules. 

In Section \ref{sec:YF}, we build Pieri and M\"obius deformations of the Young-Fibonacci lattice. We define $K$-Young-Fibonacci tableaux and suggest the corresponding 
insertion algorithm and growth rules.

In Section \ref{sec:other}, we consider some other examples of dual filtered graphs. Of special interest are the Pieri constructions associated to the quasisymmetric 
functions, to the Poirier-Reutenauer and Malvenuto-Reutenauer Hopf algebras, as well as to their $K$-theoretic analogues, which we draw from \cite{LP, PP}. 
The Hopf algebra of quasisymmetric functions and its $K$-theoretic analogue provide another instance of the M\"obius via Pieri phenomenon. 

In Section \ref{sec:enum}, we use the calculus of up and down operators to prove some identities similar to the ones known for dual graded graphs. In particular, we 
formulate and prove a $K$-theoretic analogue of the Frobenius-Young identity.  

\subsection{Acknowledgements} We thank Vic Reiner, Alexander Garver, and Thomas Lam for helpful discussions.

\section{Dual graded graphs} \label{sec:dgg}
\subsection{Dual graded graphs}
This section follows \cite{F}, and we refer the reader to this source for further reading on dual graded graphs.

A graded graph is a triple $G=(P,\rho,E)$, where $P$ is a discrete set of vertices, $\rho:P\to \mathbb{Z}$ is 
a rank function, and $E$ is a multiset of edges/arcs $(x,y)$, where $\rho(y)=\rho(x)+1$. In other words, each vertex 
is assigned a rank, and edges can only join vertices in successive ranks. For the set of vertices $P$, let 
$P_n$ denote the subset of vertices of rank $n$. For any field $K$ of characteristic zero, the formal linear combinations of vertices of $G$ form the vector space 
$KP$. 

In future examples, the idea of \textit{inner corners} and \textit{outer corners} of certain configurations of boxes or cells will be necessary. We will call a cell an
inner corner if it can be removed from the configuration and the resulting configuration is still valid. A cell is an outer corner of a configuration if it can be 
added to the configuration, and the resulting configuration is still valid.

\begin{example} Young's lattice is an example of a graded graph where for any partition $\lambda$, $\rho(\lambda)=|\lambda|$,
 and there is an edge from $\lambda$ to $\mu$ if $\mu$ can be obtained from $\lambda$ by adding one box. Ranks zero through five are 
shown below. 
We see that $(3,31)\in E$ and $\rho(3)+1=3+1=\rho(31)=4$.

 \begin{center}
\begin{tikzpicture}[scale=1]
\node (empty) at (0,0) {$\varnothing$};
\node (1) at (0,1) {\ytableausetup{boxsize=.15cm}\begin{ytableau}$ $ \end{ytableau}};
\node (2) at (1,2) {\begin{ytableau}$ $ & $ $  \end{ytableau}};
\node (11) at (-1,2) {\begin{ytableau}$ $ \\ $ $  \end{ytableau}};
\node (111) at (-2,3) {\begin{ytableau}$ $ \\ $ $ \\ $ $ \end{ytableau}};
\node (21) at (0,3) {\begin{ytableau}$ $ & $ $ \\ $ $ \end{ytableau}};
\node (3) at (2,3) {\begin{ytableau}$ $ & $ $ & $ $ \end{ytableau}};
\node (1111) at (-3,4) {\begin{ytableau}$ $ \\ $ $ \\ $ $ \\ $ $\end{ytableau}};
\node (211) at (-1,4) {\begin{ytableau}$ $ & $ $ \\ $ $ \\ $ $\end{ytableau}};
\node (22) at (0,4) {\begin{ytableau}$ $ & $ $ \\ $ $ & $ $ \end{ytableau}};
\node (31) at (1,4) {\begin{ytableau}$ $ & $ $ & $ $ \\ $ $\end{ytableau}};
\node (4) at (3,4) {\begin{ytableau}$ $ & $ $ & $ $ & $ $\end{ytableau}};
\node (5) at (4,5) {\begin{ytableau}$ $ & $ $ & & & \end{ytableau}};
\node (41) at (2,5) {\begin{ytableau}$ $ & $ $ & & \\ $ $ \end{ytableau}};
\node (32) at (1,5) {\begin{ytableau}$ $ & $ $ & \\ $ $ & $ $ \end{ytableau}};
\node (311) at (0,5) {\begin{ytableau}$ $ & $ $ & \\ $ $ \\ $ $ \end{ytableau}};
\node (221) at (-1,5) {\begin{ytableau}$ $ & $ $ \\ $ $ & $ $ \\ $ $ \end{ytableau}};
\node (2111) at (-2,5) {\begin{ytableau}$ $ & $ $ \\ $ $ \\ $ $ \\ $ $ \end{ytableau}};
\node (11111) at (-4,5) {\begin{ytableau}$ $ \\ $ $ \\ $ $ \\ $ $ \\ $ $ \end{ytableau}};
\draw (empty) -- (1)
(1) -- (2)--(3)--(4)--(5)
(1)--(11)--(111)--(1111)--(11111)
(11)--(21)--(31)--(41)
(111)--(211)--(2111)
(1111)--(2111)
(4)--(41)
(22)--(221)
(211)--(221)
(31)--(311)
(3)--(31)--(32)
(21)--(22)--(32)
(2)--(21)--(211)--(311);
\end{tikzpicture}
\end{center}
\ytableausetup{boxsize=.3cm}

If we look at the partition (4,4,2,1), visually represented below, we see that the inner corners are in positions (2,4), (3,2), and (4,1), and the outer corners are in positions (1,5), (3,3), (4,2), and (5,1).

\begin{center}\begin{ytableau} $ $ & $ $ & $ $ & $ $ \\ $ $ & $ $ & $ $ & $ $ \\ $ $ & $ $ \\ $ $ \end{ytableau}\end{center}
\end{example}

Let $G_1=(P,\rho,E_1)$ and $G_2=(P,\rho,E_2)$ be a pair of graded graphs with a common vertex set and rank 
function. From this pair, define an \textit{oriented graded graph} $G=(P,\rho, E_1,E_2)$ by orienting the 
edges of $G_1$ in the direction of increasing rank and the edges of $G_2$ in the direction of decreasing rank. 
Let $a_i(x,y)$ denote the number of edges of $E_i$ joining $x$ and $y$ or the multiplicity of edge $(x,y)$ in $E_i$.
We define the \textit{up} and \textit{down operators} $U,D\in End(KP)$ associated with graph $G$ by 
$$Ux=\sum_{y} a_1(x,y)y$$ and $$Dy=\sum_x a_2(x,y) x.$$

For example, in Young's lattice shown above, $U(21)=31+22+211$ and $D(21)=2+11$.

The restrictions of $U$ and $D$ to ``homogeneous'' subspaces $KP_n$ are denoted by $U_n$ and $D_n$, respectively. 
Graded graphs $G_1$ and $G_2$ with a common vertex set and rank function are said to be \textit{r-dual} if 
$$D_{n+1}U_n=U_{n-1}D_n+rI_n$$
for some $r\in\mathbb{R}$ and simply \textit{dual} if 
$$D_{n+1}U_n=U_{n-1}D_n+I_n.$$ We focus on the latter.

\begin{example}
 Young's lattice is an example of a self-dual graded graph.
\end{example}

\begin{example} Another well-known example of a self-dual graded graph is the \textit{Young-Fibonacci lattice}, $\mathbb{YF}$. The first 
six ranks are shown below. The 
vertices of the graph are all finite words in the alphabet $\{1,2\}$, where the rank of a word is the sum of its entries. The covering relations 
are as follows:
a word $w'$ covers $w$ if and only if either $w'=1w$ or $w'=2v$ for some $v$ covered by $w$. For example, 
the word $121$ covers the word $21$ since $121$ is obtained by concatentating $1$ and $21$, and $121$ is covered by $221$ because 
$121$ covers $21$. The words in the alphabet $\{1,2\}$ are sometimes called snakes and can be represented as a collection of boxes whose heights
correspond to the entries of the word called snakeshapes. For example, the word $122112$ can be pictured as 
\begin{center}
 \ytableausetup{boxsize=.3cm}
 \begin{ytableau}
  \none & $ $ & $ $ & \none & \none & $ $ \\
  $ $ &   $ $ &  $ $ &  $ $ &  $ $ &  $ $ 
 \end{ytableau}\hspace{.05in}.
\end{center}
Thought of this way, the rank of a word is the number of boxes in its corresponding snakeshape.

  \begin{center}
  \ytableausetup{boxsize=.15cm}
\begin{tikzpicture}[scale=1]
\node (empty) at (0,0) {$\varnothing$};
\node (1) at (0,1) {\begin{ytableau} $ $ \end{ytableau}};
\node (11) at (1,2) {\begin{ytableau} $ $ & $ $ \end{ytableau}};
\node (2) at (-1,2) {\begin{ytableau} $ $ \\ $ $ \end{ytableau}};
\node (12) at (-2,3) {\begin{ytableau} \none & $ $ \\ $ $ & \end{ytableau}};
\node (21) at (0,3) {\begin{ytableau} $ $ & \none \\ $ $ & $ $ \end{ytableau}};
\node (111) at (2,3) {\begin{ytableau} $ $ & & \end{ytableau}};
\node (112) at (-3,4) {\begin{ytableau} \none & \none & $ $ \\ $ $ & & \end{ytableau}};
\node (22) at (-1,4) {\begin{ytableau} $ $ & \\ $ $ & \end{ytableau}};
\node (121) at (0,4) {\begin{ytableau} \none & $ $ & \none \\ $ $ & & \end{ytableau}};
\node (211) at (1,4) {\begin{ytableau} $ $ & \none & \none \\ $ $ & & \end{ytableau}};
\node (1111) at (3,4) {\begin{ytableau} $ $ & & & \end{ytableau}};
\node (11111) at (4,5) {\begin{ytableau} $ $ & & & & \end{ytableau}};
\node (2111) at (2.92,5) {\begin{ytableau} $ $ & \none & \none & \none \\ $ $ & & & \end{ytableau}};
\node (1211) at (1.78,5) {\begin{ytableau} \none & $ $ & \none & \none \\ $ $ & & &  \end{ytableau}};
\node (221) at (.64,5) {\begin{ytableau} $ $ & & \none \\ $ $ & & \end{ytableau}};
\node (1121) at (-.64,5) {\begin{ytableau} \none & \none & $ $ & \none \\ $ $ & & & \end{ytableau}};
\node (212) at (-2.92,5) {\begin{ytableau} $ $ & \none & $ $ \\ $ $ & & \end{ytableau}};
\node(122) at(-1.78,5) {\begin{ytableau} \none & $ $ &  \\ $ $ & & \end{ytableau}};
\node (1112) at (-4,5) {\begin{ytableau} \none & \none & \none & $ $ \\ $ $ & & &  \end{ytableau}};
\draw (empty) -- (1) --(11)--(111)--(1111)--(11111)
(1)--(2)--(12)--(112)--(1112)
(11)--(21)--(22)--(212)
(2)--(21)--(211)--(2111)
(111)--(211)--(221)
(21)--(121)--(1121)
(12)--(22)--(212)
(112)--(212)
(22)--(122)
(22)--(221)
(121)--(221)
(211)--(1211)
(1111)--(2111);
\end{tikzpicture}
\end{center}
\end{example}
\ytableausetup{boxsize=.3cm}

Young's lattice and the Young-Fibonacci lattice discussed above are the most interesting examples of self-dual graded graphs. We close this section 
by describing the graph of shifted shapes, $\mathbb{SY}$, and its dual, which together form a dual graded graph.

\begin{example}\label{ex:shifted} Given a strict partition $\lambda$ (i.e. $\lambda = (\lambda_1>\lambda_2>\ldots>\lambda_k)$), the corresponding shifted shape 
$\lambda$ is an arrangement of cells in $k$ rows where row $i$ contains $\lambda_i$ cells and is indented $i-1$ spaces. For example, the 
strict partition $(5,4,3,1)$ corresponds to the shifted shape shown below.
\begin{center}
 \begin{ytableau}
  $ $ &   $ $ &   $ $ &   $ $ &   $ $ \\
  \none &   $ $ &  $ $ &  $ $ &  $ $ \\
  \none & \none &   $ $ &  $ $ &  $ $ \\
  \none & \none & \none &   $ $ 
 \end{ytableau}
\end{center}
We say that cells in row $i$ and column $i$ are \textit{diagonal} cells and all other cells are \textit{off-diagonal}.

The graph of shifted shapes, $\mathbb{SY}$, has shifted shapes as vertices, the rank function counts the number of cells in a shape, and shifted shape $\lambda$ 
covers $\mu$ if $\lambda$ is obtained from $\mu$ by adding one cell. The first six ranks of the graph $\mathbb{SY}$ are shown below on the left. 
The graph shown on the right is its dual. In the dual graph, there is one edge between $\lambda$ and $\mu$ if $\lambda$ is obtained 
from $\mu$ by adding a diagonal cell, and there are two edges between $\lambda$ and $\mu$ if $\lambda$ is obtained from $\mu$ by adding 
an off-diagonal cell. To form the dual graded graph, the edges of $\mathbb{SY}$ are oriented upward and those of its dual 
are oriented downward. 

\begin{minipage}{.5\textwidth}
\begin{center}
\begin{tikzpicture}
\node (empty) at (0,0) {$\varnothing$};
\node (1) at (0,1) {\ytableausetup{boxsize=.15cm}\begin{ytableau}$ $ \end{ytableau}};
\node (2) at (0,2) {\begin{ytableau}$ $ & $ $  \end{ytableau}};
\node (3) at (1,3) {\begin{ytableau}$ $ & $ $ & $ $ \end{ytableau}};
\node (21) at (-1,3) {\begin{ytableau}$ $ & $ $ \\ \none & $ $ \end{ytableau}};
\node (31) at (0,4) {\begin{ytableau}$ $ & $ $ & $ $\\ \none & $ $ \end{ytableau}};
\node (4) at (2,4) {\begin{ytableau}$ $ & $ $ & $ $ & $ $ \end{ytableau}};
\node (32) at (-1,5) {\begin{ytableau}$ $ & $ $ & $ $\\ \none & $ $ & $ $\end{ytableau}};
\node (41) at (1,5) {\begin{ytableau}$ $ & $ $ & $ $ & $ $  \\ \none & $ $ \end{ytableau}};
\node (5) at (3,5) {\begin{ytableau}$ $ & $ $ & $ $ & $ $ & $ $ \end{ytableau}};
  \draw (empty)--(1)--(2)--(3)--(4)--(5)
(2)--(21)--(31)--(32)
(3)--(31)--(41)
(4)--(41);
\end{tikzpicture}
\end{center}
\end{minipage}
\begin{minipage}{.5\textwidth}
\begin{center}
\begin{tikzpicture}
\node (empty) at (0,0) {$\varnothing$};
\node (1) at (0,1) {\ytableausetup{boxsize=.15cm}\begin{ytableau}$ $ \end{ytableau}};
\node (2) at (0,2) {\begin{ytableau}$ $ & $ $  \end{ytableau}};
\node (3) at (1,3) {\begin{ytableau}$ $ & $ $ & $ $ \end{ytableau}};
\node (21) at (-1,3) {\begin{ytableau}$ $ & $ $ \\ \none & $ $ \end{ytableau}};
\node (31) at (0,4) {\begin{ytableau}$ $ & $ $ & $ $\\ \none & $ $ \end{ytableau}};
\node (4) at (2,4) {\begin{ytableau}$ $ & $ $ & $ $ & $ $ \end{ytableau}};
\node (32) at (-1,5) {\begin{ytableau}$ $ & $ $ & $ $\\ \none & $ $ & $ $\end{ytableau}};
\node (41) at (1,5) {\begin{ytableau}$ $ & $ $ & $ $ & $ $  \\ \none & $ $ \end{ytableau}};
\node (5) at (3,5) {\begin{ytableau}$ $ & $ $ & $ $ & $ $ & $ $ \end{ytableau}};
  \draw (empty) -- (1) (2) -- (21) (3) -- (31) (4) -- (41);
\draw[double, thick] (1)--(2) (2)--(3) (3)--(4) (4)--(5)
(21)--(31) (31)--(32) (31)--(41);
\end{tikzpicture}
\end{center}
\end{minipage}
\ytableausetup{boxsize=.4cm}
 
\end{example}

\subsection{Growth rules as generalized RSK}

There is a well-known combinatorial correspondence between words in the alphabet of positive integers and pairs consisting of a 
semistandard Young tableau and a standard Young tableau of the same shape called the Robinson-Schensted-Knuth (RSK) correspondence. 
We briefly review the correspondence here and refer the reader to \cite{EC2} for more information.

Given a word $w$, the RSK correspondence maps $w$ to a pair of tableaux via a row insertion algorithm consisting of inserting a positive 
integer into a tableau. The algorithm for inserting positive integer $k$ into a row of a semistandard tableau is as follows. If $k$ is 
greater than or equal to all entries in the row, add a box labeled $k$ to the end of the row. Otherwise, find the first $y$ in the row 
with $y>k$. Replace $y$ with $k$ in this box, and proceed to insert $y$ into the next row. To insert $k$ into semistandard tableau $P$, 
we start by inserting $k$ into the first row of $P$. To create the \textit{insertion tableau} of a word $w=w_1w_2\cdots w_r$, we 
first insert $w_1$ into the empty tableau, insert $w_2$ into the result of the previous insertion, insert $w_3$ into the result of the 
previous insertion, and so on until we have inserted all letters of $w$. We denote the resulting insertion tableau by $P(w)$. The insertion 
tableau will always be a semistandard tableau.

To obtain a standard Young tableau from $w$, we define the \textit{recording tableau}, $Q(w)$, of $w$ by labeling the box of 
$P(w_1\cdots w_s)/P(w_1\cdots w_{s-1})$ by $s$. For example, $w=14252$ has insertion and recording tableau

\begin{center}
$P(w)=$
 \begin{ytableau}
  1 & 2 & 2 \\
  4 & 5
 \end{ytableau}\hspace{1in}
 $Q(w)=$
 \begin{ytableau}
  1 & 2 & 4 \\
  3 & 5
 \end{ytableau}.
\end{center}

The RSK correspondence described above is a bijection between words consisting of positive integers and pairs $(P,Q)$, where $P$ is a 
semistandard Young tableau and $Q$ is a standard Young tableau of the same shape.

If $w=w_1w_2\cdots w_k$ is a permutation, we can also obtain $(P(w),Q(w))$ from $w$ by using growth diagrams. First, we 
create a $k \times k$ array with an $X$ in the $w_i^{th}$
square from the bottom of column $i$. For example, if $w=14253$, we have 

\begin{center}
\begin{ytableau}
 $ $ & $ $ & $ $ & X & $ $ \\
 $ $ & X & $ $ & $ $ & $ $ \\
 $ $ & $ $ & $ $ & $ $ & X \\
 $ $ & $ $ & X & $ $ & $ $ \\
 X & $ $ & $ $ & $ $ & $ $ 
\end{ytableau}\hspace{.05in}.
\end{center}

We will label the corners of each square with a partition. We begin by labeling all corners
along the bottom row and left side of the diagram with the empty shape, $\varnothing$.

To complete the labeling of the corners, suppose the corners $\mu$, $\lambda$, and $\nu$ are labeled, 
where $\mu$, $\lambda$, and $\nu$ are as in the picture below. We label $\gamma$ according to the following rules.

 \begin{center}
\begin{tikzpicture}[scale=1]
\node (A) at (-1,-1) {$\lambda$};
\node (B) at (1,-1) {$\nu$};
\node (C) at (-1,1) {$\mu$};
\node (D) at (1,1) {$\gamma$};
\draw (A) -- (B)
(C) -- (D)
(B) -- (D)
(A) -- (C);
\end{tikzpicture}
\end{center}

\begin{itemize}
 \item[(L1)] If the square does not contain an $X$ and $\lambda=\mu=\nu$, then $\gamma=\lambda$.
 \item[(L2)] If the square does not contain an $X$ and $\lambda\subset\mu=\nu$, then $\mu/\lambda$ is one box in some row $i$. In this 
 case, $\gamma$ is obtained from $\mu$ by adding one box in row $i=1$.
 \item[(L3)] If the square does not contain an $X$ and $\mu\neq\nu$, define $\gamma=\mu\cup\nu$.
 \item[(L4)] If the square contains an $X$, then $\lambda=\nu=\mu$ and $\gamma$ is obtained from $\mu$ by adding one box in the first row.
\end{itemize}

Following these rules, there is a unique way to label the corners of the diagram. The resulting array is called the \textit{growth diagram}
of $w$, and rules (L1)-(L4) are called \textit{growth rules}. For the remainder of this paper, we will use the word ``square'' when referring to a square in the 
growth diagram and the word ``box'' or ``cell'' when referring to a box in a partition, shifted partition, or snakeshape.

The chains of partitions read across the top of the diagram and up the right side of the diagram determine $Q(w)$ and $P(w)$, respectively.
To obtain a tableau of shape $\lambda$ from a chain of partitions $\varnothing \subset \lambda^1 \subset \lambda^2 \subset \ldots \subset \lambda^k=\lambda$, 
label the box of $\lambda^i/\lambda^{i-1}$ by $i$.

\begin{example}
The growth diagram for the word $14253$ is shown below.

\begin{center}
\begin{tikzpicture}[scale=1]
\node (00) at (0,0) {$\varnothing$};
\node (10) at (1,0) {$\varnothing$};
\node (20) at (2,0) {$\varnothing$};
\node (30) at (3,0) {$\varnothing$};
\node (40) at (4,0) {$\varnothing$};
\node (50) at (5,0) {$\varnothing$};
\node (01) at (0,1) {$\varnothing$};
\node (.5.5) at (.5,.5) {$X$};
\node (11) at (1,1) {$1$};
\node (21) at (2,1) {$1$};
\node (31) at (3,1) {$1$};
\node (41) at (4,1) {$1$};
\node (51) at (5,1) {$1$};
\node (02) at (0,2) {$\varnothing$};
\node (12) at (1,2) {$1$};
\node (22) at (2,2) {$1$};
\node (32) at (3,2) {$2$};
\node (2.51.5) at (2.5,1.5) {$X$};
\node (42) at (4,2) {$2$};
\node (52) at (5,2) {$2$};
\node (03) at (0,3) {$\varnothing$};
\node (13) at (1,3) {$1$};
\node (23) at (2,3) {$1$};
\node (33) at (3,3) {$2$};
\node (43) at (4,3) {$2$};
\node (53) at (5,3) {$3$};
\node (4.52.5) at (4.5,2.5) {$X$};
\node (04) at (0,4) {$\varnothing$};
\node (14) at (1,4) {$1$};
\node (24) at (2,4) {$2$};
\node (34) at (3,4) {$21$};
\node (44) at (4,4) {$21$};
\node (54) at (5,4) {$31$};
\node (1.53.5) at (1.5,3.5) {$X$};
\node (05) at (0,5) {$\varnothing$};
\node (15) at (1,5) {$1$};
\node (25) at (2,5) {$2$};
\node (35) at (3,5) {$21$};
\node (45) at (4,5) {$31$};
\node (55) at (5,5) {$32$};
\node (3.54.5) at (3.5,4.5) {$X$};

\draw (00)--(10)--(20)--(30)--(40)--(50)
(01)--(11)--(21)--(31)--(41)--(51)
(02)--(12)--(22)--(32)--(42)--(52)
(03)--(13)--(23)--(33)--(43)--(53)
(04)--(14)--(24)--(34)--(44)--(54)
(05)--(15)--(25)--(35)--(45)--(55)
(00)--(01)--(02)--(03)--(04)--(05)
(10)--(11)--(12)--(13)--(14)--(15)
(20)--(21)--(22)--(23)--(24)--(25)
(30)--(31)--(32)--(33)--(34)--(35)
(40)--(41)--(42)--(43)--(44)--(45)
(50)--(51)--(52)--(53)--(54)--(55);
\end{tikzpicture}
\end{center}

From the chains of partitions on the rightmost edge and uppermost edge of the diagram, we can read the insertion tableau and recording tableau, respectively.

\begin{center}
$P(14253)=$
 \begin{ytableau}
 1 & 2 & 3 \\
 4 & 5
 \end{ytableau}\hspace{1in}
 $Q(14253)=$
 \begin{ytableau}
 1 & 2 & 4\\
 3 & 5
 \end{ytableau}
\end{center}
\end{example}

We next draw a connection between the growth rules described and an explicit cancellation in Young's lattice that shows that $DU-UD=I$.
We shall start by giving explicit cancellation rules to pair all down-up paths from $\mu$ to $\nu$ with up-down paths from $\mu$ to $\nu$, 
which will leave one up-down path unpaired in $\mu=\nu$. (Note that if 
$\rho(\mu)\neq \rho (\nu)$, then there are no up-down or down-up paths from $\mu$ to $\nu$.)

\begin{itemize}
 \item[(C1)] If $\mu\neq \nu$, there is only one down-up path from $\mu$ to $\nu$, which passes through $\mu\cap\nu$. Pair this 
 with the only up-down path from $\mu$ to $\nu$, which passes through $\mu\cup\nu$.
 \item[(C2)] If $\mu=\nu$, then any down-up path can be represented by the inner corner of $\mu$ that is deleted in the downward step, and 
 any up-down path can be represented by the outer corner of $\mu$ that is added in the upward step. 
 Pair each down-up path with the up-down path that adds the first outer corner of $\mu$ strictly below the inner corner of the down-up
 path. 
\end{itemize}

Now, the only up-down path that has not been paired with a down-up path is the path corresponding to adding the outer corner of 
$\mu$ in the first row. This gives $(DU-UD)\mu=\mu$.

It is easy to see that this cancellation yields exactly the growth rules described above. Cancellation rule (C1) determines 
growth rule (L3), cancellation rule (C2) determines growth rule (L2), and the up-down path left unmatched determines growth rule
(L4). If we started with a different explicit cancellation, we could modify (L2), (L3), and (L4) accordingly 
to obtain new growth rules. Note that growth rule (L1) is simply a way to transfer information and will remain the 
same for any explicit cancellation. 

\begin{example}
 Starting with $\mu=(3,1)$, cancellation rule (C2) says that the down-up path from $\mu$ to itself passing through $(2,1)$ is paired 
 with the up-down path passing through $(3,2)$. The corresponding growth rule in this situation, (L4), is illustrated on the right.

\begin{center}
\begin{tikzpicture}[scale=1]
\node (A) at (0,-1) {21};
\node (B) at (0,0) {31};
\node (D) at (0,1) {32};
\draw (A) -- (B) -- (D);
\end{tikzpicture}\hspace{1.5in}
\begin{tikzpicture}[scale=1]
\node (A) at (-1,-1) {21};
\node (B) at (1,-1) {31};
\node (C) at (-1,1) {31};
\node (D) at (1,1) {32};
\draw (A) -- (B)
(C) -- (D)
(B) -- (D)
(A) -- (C);
\end{tikzpicture}
\end{center}
\end{example}

\section{Dual filtered graphs} \label{sec:dfg}

\subsection{The definition}
A weak filtered graph is a triple $G=(P,\rho,E)$, where $P$ is a discrete set of vertices, $\rho:P\to \mathbb{Z}$ is 
a rank function, and $E$ is a multiset of edges/arcs $(x,y)$, where $\rho(y) \geq \rho(x)$. 
A strict filtered graph is a triple $G=(P,\rho,E)$, where $P$ is a discrete set of vertices, $\rho:P\to \mathbb{Z}$ is 
a rank function, and $E$ is a multiset of edges/arcs $(x,y)$, where $\rho(y) > \rho(x)$. 
Let $P_n$ as before denote the subset of vertices of rank $n$. 

Let $G_1=(P,\rho,E_1)$ and $G_2=(P,\rho,E_2)$ be a pair of filtered graphs with a common vertex set, $G_1$ weak while $G_2$ strict. 
From this pair, define an \textit{oriented filtered graph} $G=(P, E_1,E_2)$ by orienting the 
edges of $G_1$ in the positive filtration direction and the edges of $G_2$ in the negative filtration direction. 
Recall that $a_i(x,y)$ denotes the number of edges of $E_i$ joining $x$ and $y$ or the multiplicity of edge $(x,y)$ in $E_i$.
Using the \textit{up} and \textit{down operators} associated with graph $G$ as previously defined, we say that $G_1$ and $G_2$ 
are \textit{dual filtered graphs} if for any $x\in G$, $$(DU-UD)x=\alpha x+\beta Dx$$ for some $\alpha,\beta\in\mathbb{R}$. 
We focus on examples where $DU-UD=D+I$, see Remark \ref{rem:rescale}.

In the next sections, we describe three constructions of dual filtered graphs.

\subsection{Trivial construction}
Every dual graded graph has an easy deformation that makes it a dual filtered graph. To construct it, simply add 
$\rho(x)$ ``upward'' loops to each element $x$ of the dual graded graph, where $\rho$ is the rank function. We will call this the 
\textit{trivial construction}.

\begin{theorem}
 For any dual graded graph $G$, the trivial construction gives a dual filtered graph $G'$.
\end{theorem}

\begin{proof}
 It suffices to show that if $q$ covers $p$ in  $G$, then $[p](DU-UD)(q)=[p](D)(q)$. If we restrict to all paths that do not contain a loop, 
the coefficient of 
$p$ in $(DU-UD)(q)$ is $0$ since without loops, we have a pair of dual graded graphs. Hence we may restrict our attention to 
up-down paths and down-up paths that contain a loop. 
The up-down paths beginning at $q$ and ending at $p$ are exactly the paths consisting of a loop at $q$ followed by an 
edge from $q$ to $p$. There are $\rho(q)e(q\to p)$ such paths.
The down-up paths containing a loop are the paths consisting of an edge from $q$ to $p$ followed by a loop at $p$. There are $e(q\to p)\rho(p)$ such paths. Thus 
$$[p](DU-UD)(q)=e(q \to p)(\rho(q)-\rho(p))=e(q \to p) = [p]D(q).$$
\end{proof}

\begin{example} The trivial construction for the first four ranks of Young's lattice is shown below. The edges in the graph on the left are oriented upward, and 
the edges in the graph on the right are oriented downward.

\begin{minipage}{.5\textwidth}
 \begin{center}
\begin{tikzpicture}[scale=.9]
\node (empty) at (0,0) {$\varnothing$};
\node (1) at (0,1) {\ytableausetup{boxsize=.15cm}\begin{ytableau}$ $ \end{ytableau}};
\node (2) at (1,2) {\begin{ytableau}$ $ & $ $  \end{ytableau}};
\node (11) at (-1,2) {\begin{ytableau}$ $ \\ $ $  \end{ytableau}};
\node (111) at (-2,3) {\begin{ytableau}$ $ \\ $ $ \\ $ $ \end{ytableau}};
\node (21) at (0,3) {\begin{ytableau}$ $ & $ $ \\ $ $ \end{ytableau}};
\node (3) at (2,3) {\begin{ytableau}$ $ & $ $ & $ $ \end{ytableau}};
\draw (empty) -- (1)
(1) -- (2)--(3)
(1)--(11)--(111)
(11)--(21)
(2)--(21);
\draw(1) to [out=60,in=100,distance=.5cm] (1);
\draw(2) to [out=60,in=100,distance=.5cm] (2);
\draw(2) to [out=60,in=100,distance=.7cm] (2);
\draw(11) to [out=60,in=100,distance=.5cm] (11);
\draw(11) to [out=60,in=100,distance=.7cm] (11);
\draw(111) to [out=60,in=100,distance=.5cm] (111);
\draw(111) to [out=60,in=100,distance=.7cm] (111);
\draw(111) to [out=60,in=100,distance=.9cm] (111);
\draw(21) to [out=60,in=100,distance=.5cm] (21);
\draw(21) to [out=60,in=100,distance=.7cm] (21);
\draw(21) to [out=60,in=100,distance=.9cm] (21);
\draw(3) to [out=60,in=100,distance=.5cm] (3);
\draw(3) to [out=60,in=100,distance=.7cm] (3);
\draw(3) to [out=60,in=100,distance=.9cm] (3);
\end{tikzpicture}
\end{center}
\end{minipage}
\begin{minipage}{.5\textwidth}
 \begin{center}
\begin{tikzpicture}[scale=.9]
\node (empty) at (0,0) {$\varnothing$};
\node (1) at (0,1) {\ytableausetup{boxsize=.15cm}\begin{ytableau}$ $ \end{ytableau}};
\node (2) at (1,2) {\begin{ytableau}$ $ & $ $  \end{ytableau}};
\node (11) at (-1,2) {\begin{ytableau}$ $ \\ $ $  \end{ytableau}};
\node (111) at (-2,3) {\begin{ytableau}$ $ \\ $ $ \\ $ $ \end{ytableau}};
\node (21) at (0,3) {\begin{ytableau}$ $ & $ $ \\ $ $ \end{ytableau}};
\node (3) at (2,3) {\begin{ytableau}$ $ & $ $ & $ $ \end{ytableau}};
\draw (empty) -- (1)
(1) -- (2)--(3)
(1)--(11)--(111)
(11)--(21)
(2)--(21);
\end{tikzpicture}
\end{center}
\end{minipage}
\end{example}
\ytableausetup{boxsize=.4cm}

\subsection{Pieri construction}

The following construction is close to the one for dual graded graphs described in the works of Bergeron-Lam-Li \cite{BLL}, Nzeutchap \cite{HN},
 and Lam-Shimozono \cite{LS}. 

Let $$A = \oplus_{i \geq 0} A_i$$ be a filtered, commutative algebra  over a field $\mathbb F$ with $A_0 = \mathbb F$ and a linear basis 
$a_{\lambda} \in A_{|\lambda|}$. Here, the $\lambda$'s belong to some infinite indexing set $P$, each graded component of which $$P_i = \{\lambda \mid |\lambda|=i\}$$
is finite. The property of being filtered means that $a_\lambda a_\mu=\sum c_\nu a_\nu$, where $|\lambda|+|\mu|\leq|\nu|.$ 

Let $f = f_1 +f_2 + \ldots \in \hat A$ be an element of the completion $\hat A$ of 
$A$ such that $f_i \in A_i$. Assume $D \in End(A)$ is a derivation on $A$ satisfying $D(f) = f+1$. 

To form a graph, we define $E_1$ by defining $a_1(\lambda,\mu)$ to be the coefficient of $a_{\lambda} \text{ in } D (a_{\mu}).$
We also define $E_2$ by defining $a_2(\lambda,\mu)$ to be the coefficient of $a_{\mu} \text{ in } a_{\lambda} f.$

\begin{theorem} \label{thm:pieri}
 The resulting graph is a dual filtered graph.
\end{theorem}

\begin{proof}
 It follows from $$D(f a_{\lambda}) = D(f) a_{\lambda} + f D(a_{\lambda}),$$ or 
$$D(f a_{\lambda}) - f D(a_{\lambda}) = a_{\lambda} + D(a_{\lambda}).$$
\end{proof}

Assume $A_1$ is one-dimensional, and let $g$ be its generator. 
We shall often find a derivation $D$ from a {\it {bialgebra}} structure on $A$. Indeed, assume $\Delta$
is a coproduct on $A$ such that $\Delta(a) = 1 \otimes a + a \otimes 1 + \dotsc$, where the rest of the terms lie in 
$(A_1 \oplus A_2 \oplus \ldots ) \otimes (A_1 \oplus A_2 \oplus \ldots )$. Assume also 
$\xi$ is a map $A \otimes A \longrightarrow A$ such that $\xi(p \otimes g) = p$ and for any element $q \in A_i$, $i \not = 1$, 
we have $\xi(p \otimes q) = 0$.

\begin{lemma}\label{lem:derivation}
 The map $D(a) = \xi(\Delta(a))$ is a derivation.
\end{lemma}

\begin{proof}
 We have $\xi(\Delta(ab)) = \xi(\Delta(a) \Delta(b)))$. Pick a linear basis for $A$ that contains $1$ and $g$. Express $\Delta(a)$ and $\Delta(b)$ so that the right 
side of tensors consists of the basis elements. The only terms that are not killed by $\xi$ are the ones of the form $p \otimes g$. The only way we can get such terms is from
$(p_1 \otimes 1) \cdot (p_2 \otimes g)$ or  $(p_1 \otimes g) \cdot (p_2 \otimes 1)$. Those are exactly the terms that contribute to
$$\xi(\Delta(a))\cdot b + a \cdot \xi(\Delta(b)).$$ 
\end{proof}

\subsection{M\"obius construction}

Let $G=(P, \rho, E_1, E_2)$ be a dual graded graph. From $G$, form a pair of filtered graphs as follows. Let $G_1' = (P, \rho, E_1')$ have the set of edges $E_1'$
consisting of the same edges as $E_1$ plus $\#\{x|y\text{ covers }x\text{ in }G_1\}$ loops at each vertex $y \in P$. Let $G_2' = (P, \rho, E_2')$ 
have the set of edges $E_2'$ consisting 
of $$a'_2(x,y) = |\mu(x,y)|$$ edges between vertices $x$ and $y$, where $\mu$ denotes the M\"obius function in $(P, \rho, E_2)$. Compose $G_1'$ and $G_2'$
into an oriented filtered graph $\tilde G$.

As we shall see, in all three of our major examples, the M\"obius deformation forms a dual filtered graph. 
Unlike with the Pieri construction, we do not have the algebraic machinary explaining why 
the construction yields dual filtered graphs. Instead, we provide the proofs on a case-by-case basis. 
Nevertheless, the M\"obius construction is the most important one, as it is this construction that relates to insertion algorithms and thus to $K$-theory of Grassmannians and 
Lagrangian Grassmannians.

It is not the case that every dual graded graph's M\"obius deformation makes it a dual filtered graph. For example,
it is easy to see that the M\"obius deformation of the graphs in Example \ref{ex:mMR} and Example \ref{ex:SYTtree} are not dual filtered graphs.

\begin{problem}
 Determine for which dual graded graphs $G$ the M\"obius construction $\tilde G$ yields a dual filtered graph.
\end{problem}

\section{Major examples: Young's lattice} \label{sec:Y}

\subsection{Pieri deformations of Young's lattice}

Let $A = \Lambda$ be the ring of symmetric functions. Let $s_{\lambda}$, $p_{\lambda}$, and $h_{\lambda}$ be its bases of Schur functions, power sum symmetric functions, and 
complete homogeneous symmetric functions. Let $$f = h_1 + h_2 + \ldots. $$ Define up and down edges of a filtered graph $G = (P, \rho, E_1, E_2)$ by letting
$a_1(\mu, \nu)$ be the coefficient of $s_{\nu} \text{ in } p_1 s_{\mu}, \text{  and  } a_2(\mu, \nu)$ be the coefficient of $s_{\nu} \text{ in } f s_{\mu}.$

Recall that given two partitions, $\mu$ and $\nu$, $\mu/\nu$ forms a \textit{horizontal strip} if no two boxes of $\mu/\nu$ lie in the same column. For example, 
$(4,2,1)/(2,1)$ forms a horizontal strip while $(4,3,1)/(2,1)$ does not.
\begin{center}
 \begin{ytableau}
  \none & \none & $ $ & $ $ \\
  \none & $ $ \\
  $ $
 \end{ytableau}\hspace{1in}
 \begin{ytableau}
  \none & \none & $ $ & $ $ \\
  \none & $ $ & $ $\\
  $ $
 \end{ytableau}
\end{center}

\begin{lemma}
 The up edges $E_1$ coincide with those of Young's lattice, while the down edges $E_2$ connect shapes that differ by a horizontal strip.
\end{lemma}

\begin{proof}
 The statement follows from the Pieri rule and the fact that $p_1=h_1$, see \cite{EC2}.
\end{proof}

In the figure below, the first six ranks are shown. The edges in the graph on the left are upward-oriented, and the edges on the graph on the right are downward-oriented.

\begin{minipage}{.5\textwidth}
 \begin{center}
\begin{tikzpicture}[scale=1]
\node (empty) at (0,0) {$\varnothing$};
\node (1) at (0,1) {\ytableausetup{boxsize=.15cm}\begin{ytableau}$ $ \end{ytableau}};
\node (2) at (1,2) {\begin{ytableau}$ $ & $ $  \end{ytableau}};
\node (11) at (-1,2) {\begin{ytableau}$ $ \\ $ $  \end{ytableau}};
\node (111) at (-2,3) {\begin{ytableau}$ $ \\ $ $ \\ $ $ \end{ytableau}};
\node (21) at (0,3) {\begin{ytableau}$ $ & $ $ \\ $ $ \end{ytableau}};
\node (3) at (2,3) {\begin{ytableau}$ $ & $ $ & $ $ \end{ytableau}};
\node (1111) at (-3,4) {\begin{ytableau}$ $ \\ $ $ \\ $ $ \\ $ $\end{ytableau}};
\node (211) at (-1,4) {\begin{ytableau}$ $ & $ $ \\ $ $ \\ $ $\end{ytableau}};
\node (22) at (0,4) {\begin{ytableau}$ $ & $ $ \\ $ $ & $ $ \end{ytableau}};
\node (31) at (1,4) {\begin{ytableau}$ $ & $ $ & $ $ \\ $ $\end{ytableau}};
\node (4) at (3,4) {\begin{ytableau}$ $ & $ $ & $ $ & $ $\end{ytableau}};
\draw (empty) -- (1)
(1) -- (2)--(3)--(4)
(1)--(11)--(111)--(1111)
(11)--(21)--(31)
(111)--(211)
(3)--(31)
(21)--(22)
(2)--(21)--(211);
\end{tikzpicture}
\end{center}
\end{minipage}
\begin{minipage}{.5\textwidth}
 \begin{center}
\begin{tikzpicture}[scale=1]
\node (empty) at (0,0) {$\varnothing$};
\node (1) at (0,1) {\ytableausetup{boxsize=.15cm}\begin{ytableau}$ $ \end{ytableau}};
\node (2) at (1,2) {\begin{ytableau}$ $ & $ $  \end{ytableau}};
\node (11) at (-1,2) {\begin{ytableau}$ $ \\ $ $  \end{ytableau}};
\node (111) at (-2,3) {\begin{ytableau}$ $ \\ $ $ \\ $ $ \end{ytableau}};
\node (21) at (0,3) {\begin{ytableau}$ $ & $ $ \\ $ $ \end{ytableau}};
\node (3) at (2,3) {\begin{ytableau}$ $ & $ $ & $ $ \end{ytableau}};
\node (1111) at (-3,4) {\begin{ytableau}$ $ \\ $ $ \\ $ $ \\ $ $\end{ytableau}};
\node (211) at (-1,4) {\begin{ytableau}$ $ & $ $ \\ $ $ \\ $ $\end{ytableau}};
\node (22) at (0,4) {\begin{ytableau}$ $ & $ $ \\ $ $ & $ $ \end{ytableau}};
\node (31) at (1,4) {\begin{ytableau}$ $ & $ $ & $ $ \\ $ $\end{ytableau}};
\node (4) at (3,4) {\begin{ytableau}$ $ & $ $ & $ $ & $ $\end{ytableau}};
\draw (empty) -- (1)
(1) -- (2)--(3)--(4)
(1)--(11)--(111)--(1111)
(11)--(21)--(31)
(111)--(211)
(3)--(31)
(21)--(22)
(2)--(21)--(211)
(3) to [bend left=20] (empty)
(4) to [bend left=25] (empty)
(2) to [bend left=15] (empty)
(3) to [bend left =15] (1)
(4) to [bend left =15](2)
(4) to [bend left =20] (1)
(21)--(1)
(211)--(11)
(22) to [bend left =10] (2)
(31) to [bend left =15] (11)
(31) to [bend left =10] (1)
(31)--(2);
\end{tikzpicture}
\end{center}
\end{minipage}

\begin{theorem} \label{thm:young}
The Pieri deformation of Young's lattice is a dual filtered graph with $$DU-UD=D+I.$$ 
\end{theorem} 

Recall the following facts regarding the ring of symmetric functions. 

\begin{theorem} \cite{EC2, Mac}
\begin{itemize}
 \item[(a)] $\Lambda$ is a free polynomial ring in $p_1, p_2, \ldots$.
 \item[(b)] $\Lambda$ can be endowed with a standard bilinear inner product determined by $\langle s_{\lambda}, s_{\mu} \rangle = \delta_{\lambda, \mu}$.
\end{itemize}
\end{theorem}

Because of the first property, we can differentiate elements of $\Lambda$ with respect to $p_1$ by expressing them first as a polynomial in the $p_i$'s. 
We shall need the following two properties of $f$ and $p_1$.

\begin{lemma} \label{lem:young}
 \begin{itemize}
  \item[(a)] For any $h,g \in \Lambda$ we have $$\langle g, p_1 h \rangle = \langle \frac{d}{d p_1} g, h \rangle.$$
  \item[(b)]  We have $$\frac{d}{d p_1} f = f+1.$$
 \end{itemize}
\end{lemma}

\begin{proof}
By bilinearity, it suffices to prove the first claim for $g,h$ in some basis of $\Lambda$. Let us choose 
the power sum basis. Then we want to argue that $$\langle p_{\lambda}, p_1 p_{\mu} \rangle = \langle \frac{d}{d p_1} p_{\lambda}, p_{\mu} \rangle.$$
This follows from the fact that $p_{\lambda}$'s form an orthogonal basis and \cite[Proposition 7.9.3]{EC2}.

 For the second claim, using \cite[Proposition 7.7.4]{EC2} we have 
$$1 + f = e^{p_1 + \frac{p_2}{2} + \ldots},$$ which of course implies $\frac{d}{d p_1} (1+f) = 1+f$.
\end{proof}

Now we are ready for the proof of Theorem \ref{thm:young}.

\begin{proof}
 Applying Lemma \ref{lem:young} we see that $A = \Lambda$, $a_{\lambda} = s_{\lambda}$, $D = \frac{d}{d p_1}$ and $f = h_1 + h_2 + \ldots$ satisfy the conditions of 
Theorem \ref{thm:pieri}. The claim follows.
\end{proof}

\begin{remark}
We can make a similar argument for the analogous construction where the edges in $E_1$ are again those of Young's lattice and the edges in $E_2$
connect shapes that differ by a vertical strip. In this case we use $f = e_1 + e_2 + \ldots,$ where the $e_i$ are the elementary symmetric functions. 
\end{remark}

\subsection{M\"obius deformation of Young's lattice}\label{sec:MobiusYoung}

Given two partitions, $\lambda$ and $\nu$, $\lambda/\nu$ forms a \textit{rook strip} if no two boxes of $\lambda/\nu$ lie 
in the same row and no two boxes of $\lambda/\nu$ lie in the same column. In other words, $\lambda/\nu$ is a rook strip if it is a 
disconnected union of boxes. We state the following well-known result about the M\"obius function 
on Young's lattice.

\begin{proposition}
We have 
$$
|\mu(\lambda, \nu)| = 
\begin{cases}
1 & \text{if $\lambda/\nu$ is a rook strip;}\\
0 & \text{otherwise.}
\end{cases}
$$
\end{proposition}

\begin{proof}
The statement follows from the fact that Young's lattice is a distributive lattice, and the interval $[\lambda,\mu]$ is isomorphic to a Boolean algebra if and only if $\mu/\lambda$ is a rook strip.
See \cite[Example 3.9.6]{EC1}.
\end{proof}

The M\"obius deformation of Young's lattice is shown below with upward-oriented edges shown on the left and downward-oriented edges shown on the right.
Notice that loops at a partition $\lambda$ may be indexed by inner corners of $\lambda$.

\begin{minipage}{.5\textwidth}
 \begin{center}
\begin{tikzpicture}[scale=1]
\node (empty) at (0,0) {$\varnothing$};
\node (1) at (0,1) {\ytableausetup{boxsize=.15cm}\begin{ytableau}$ $ \end{ytableau}};
\node (2) at (1,2) {\begin{ytableau}$ $ & $ $  \end{ytableau}};
\node (11) at (-1,2) {\begin{ytableau}$ $ \\ $ $  \end{ytableau}};
\node (111) at (-2,3) {\begin{ytableau}$ $ \\ $ $ \\ $ $ \end{ytableau}};
\node (21) at (0,3) {\begin{ytableau}$ $ & $ $ \\ $ $ \end{ytableau}};
\node (3) at (2,3) {\begin{ytableau}$ $ & $ $ & $ $ \end{ytableau}};
\node (1111) at (-3,4) {\begin{ytableau}$ $ \\ $ $ \\ $ $ \\ $ $\end{ytableau}};
\node (211) at (-1,4) {\begin{ytableau}$ $ & $ $ \\ $ $ \\ $ $\end{ytableau}};
\node (22) at (0,4) {\begin{ytableau}$ $ & $ $ \\ $ $ & $ $ \end{ytableau}};
\node (31) at (1,4) {\begin{ytableau}$ $ & $ $ & $ $ \\ $ $\end{ytableau}};
\node (4) at (3,4) {\begin{ytableau}$ $ & $ $ & $ $ & $ $\end{ytableau}};
\draw (empty) -- (1)
(1) -- (2)--(3)--(4)
(1)--(11)--(111)--(1111)
(11)--(21)--(31)
(111)--(211)
(3)--(31)
(21)--(22)
(2)--(21)--(211)
(1) to [out=60,in=100,distance=.5cm] (1)
(2) to [out=60,in=100,distance=.5cm] (2) 
(11) to [out=60,in=100,distance=.5cm] (11)
(3) to [out=60,in=100,distance=.5cm] (3) 
(4) to [out=60,in=100,distance=.5cm] (4) 
(111) to [out=60,in=100,distance=.5cm] (111) 
(1111) to [out=60,in=100,distance=.5cm] (1111)
(21) to [out=-60,in=-100,distance=.5cm] (21) 
(21) to [out=-60,in=-100,distance=.7cm] (21) 
(211) to [out=60,in=100,distance=.5cm] (211) 
(211) to [out=60,in=100,distance=.7cm] (211) 
(31) to [out=60,in=100,distance=.5cm] (31) 
(31) to [out=60,in=100,distance=.7cm] (31) 
(22) to [out=60,in=100,distance=.5cm] (22) ;
\end{tikzpicture}
\end{center}
\end{minipage}
\begin{minipage}{.5\textwidth}
 \begin{center}
\begin{tikzpicture}[scale=1]
\node (empty) at (0,0) {$\varnothing$};
\node (1) at (0,1) {\ytableausetup{boxsize=.15cm}\begin{ytableau}$ $ \end{ytableau}};
\node (2) at (1,2) {\begin{ytableau}$ $ & $ $  \end{ytableau}};
\node (11) at (-1,2) {\begin{ytableau}$ $ \\ $ $  \end{ytableau}};
\node (111) at (-2,3) {\begin{ytableau}$ $ \\ $ $ \\ $ $ \end{ytableau}};
\node (21) at (0,3) {\begin{ytableau}$ $ & $ $ \\ $ $ \end{ytableau}};
\node (3) at (2,3) {\begin{ytableau}$ $ & $ $ & $ $ \end{ytableau}};
\node (1111) at (-3,4) {\begin{ytableau}$ $ \\ $ $ \\ $ $ \\ $ $\end{ytableau}};
\node (211) at (-1,4) {\begin{ytableau}$ $ & $ $ \\ $ $ \\ $ $\end{ytableau}};
\node (22) at (0,4) {\begin{ytableau}$ $ & $ $ \\ $ $ & $ $ \end{ytableau}};
\node (31) at (1,4) {\begin{ytableau}$ $ & $ $ & $ $ \\ $ $\end{ytableau}};
\node (4) at (3,4) {\begin{ytableau}$ $ & $ $ & $ $ & $ $\end{ytableau}};
\draw (empty) -- (1)
(1) -- (2)--(3)--(4)
(1)--(11)--(111)--(1111)
(11)--(21)--(31)
(111)--(211)
(3)--(31)
(21)--(22)
(2)--(21)--(211)
(21)--(1)
(31)--(2)
(211)--(11);
\end{tikzpicture}
\end{center}
\end{minipage}
\ytableausetup{boxsize=.5cm}

\begin{theorem}
 The M\"obius deformation of Young's lattice forms a dual filtered graph with $$DU-UD=D+I.$$
\end{theorem}

\begin{proof}

It suffices to show that $[\lambda](DU-UD)\mu=[\lambda]D\mu$ when $\mu$ covers $\lambda$ since we started with a 
dual graded graph.

Suppose shape $\lambda$ is covered by shape $\mu$. Each inner corner of $\mu$ contributes one to the coefficient
of $\lambda$ in $DU\mu$ via taking the loop corresponding to that inner corner of $\mu$ as an upward move
and then going down to $\lambda$.

Next, consider shape $\mu$ and mark the outer corners of $\lambda$ contained in $\mu$. The marked boxes will be inner corners 
of $\mu$. The shapes with an upward arrow from $\mu$ and a downward arrow to $\lambda$ are exactly those obtained by
adding one box to an outer corner of $\mu$ which is not adjacent to one of the marked boxes. There are 
$\#\{\text{outer corners of $\lambda$}\}-|\mu/\lambda|$ possible places to add such a box. Thus the coefficient of $\lambda$
in $DU\mu$ is $$\#\{\text{inner corners of $\mu$}\}+\#\{\text{outer corners of $\lambda$}\}-|\mu/\lambda|.$$

For down-up paths that involve three shapes, consider the inner corners of $\mu$ that are also inner corners of $\lambda$.
Every shape obtained by removing one of these inner corners determines one down-up path from $\mu$ to $\lambda$, and there
are $\#\{\text{inner corners of $\mu$}\}-|\mu/\lambda|$ such shapes. 
The remaining down-up paths are given by going from $\mu$ to $\lambda$ and then taking a loop at $\lambda$. Hence
the coefficient of $\lambda$ in $UD\mu$ is $$\#\{\text{inner corners of $\mu$}\}-|\mu/\lambda|+\#\{\text{inner corners of }\lambda\}.$$
Since all shapes have one more outer corner than inner corner, the coefficient of $\lambda$ in $(DU-UD)\mu$
is 1.

\end{proof}

\subsection{Hecke insertion}
We next describe an insertion and reverse insertion procedure that correspond to the M\"obius deformation of Young's lattice in the same way that RSK corresponds 
to Young's lattice, which is called \textit{Hecke insertion}. This procedure was introduced in \cite{BKSTY}. In other words, if we insert a word $w$ of length $n$ using Hecke insertion, 
the construction of the insertion tableau will be represented as a path of length $n$ downward that ends at $\varnothing$ in the 
M\"obius deformation of Young's lattice, and the construction of the recording tableau is represented as a path of length $n$ upward starting at $\varnothing$.

An \textit{increasing tableau} is a filling of a Young diagram with positive integers such that the entries in rows are strictly
increasing from left to right and the entries in columns are strictly increasing
from top to bottom. 

\begin{example}
 The tableau shown on the left is an increasing tableau. The tableau on the right is not an increasing tableau
because the entries in the first row are not strictly increasing.
\begin{center}
\begin{ytableau}
 1 & 2 & 4 & 5 \\
2 & 3 & 5 & 7 \\
6 & 7 \\
8
\end{ytableau}\hspace{1in}
\begin{ytableau}
 1 & 2 & 2 & 4 \\
3 & 4 \\
5
\end{ytableau}
\end{center}
\end{example}

We follow \cite{BKSTY} to give a description of Hecke (row) insertion of a positive integer $x$ 
into an increasing tableau $Y$ resulting in an increasing tableau $Z$.
The shape of $Z$ is obtained from the shape of $Y$ by adding at most one box. If a box is added in position $(i,j)$, then
we set $c=(i,j)$. 
In the case where no box is added, then $c=(i,j)$, where $(i,j)$ is a special corner indicating where the insertion
process terminated. We will use a parameter $\alpha\in\{0,1\}$ to keep track
of whether or not a box is added to $Y$ after inserting $x$ by setting
$\alpha=0$ if $c\in Y$ and $\alpha=1$ if $c\notin Y$. 
We use the notation $Z=(Y {\overset{H}{\longleftarrow}} x)$ to denote the resulting tableau, and we denote the 
outcome of the insertion by $(Z,c,\alpha)$.

We now describe how to insert $x$ into increasing tableau $Y$ by describing
how to insert $x$ into $R$, a row of $Y$. This insertion may 
modify the row and may produce an output integer, 
which we will insert into the next row. To begin the insertion process, insert $x$ into the first row of $Y$.  
The process stops when there is no output integer. 
The rules for insertion of $x$ into $R$ are as follows:
\begin{itemize}
\item[(H1)] If $x$ is weakly larger than all integers in $R$ and adjoining $x$ to the end of row $R$ 
results in an increasing tableau, then $Z$ is the resulting tableau and $c$ is the new
corner where $x$ was added.

\item[(H2)] If $x$ is weakly larger than all integers in $R$ and adjoining $x$ to the end of row $R$ 
does not result in an increasing tableau, then $Z=Y$, and $c$ is the 
box at the bottom of the column of $Z$ containing the rightmost box of the row $R$. 
\end{itemize}
For the next two rules, assume $R$ contains boxes strictly larger than $x$, and let $y$ be the smallest such box.
\begin{itemize}
\item[(H3)] If replacing $y$ with $x$ results in an increasing tableau, then replace $y$ with $x$. 
In this case, $y$ is the output integer to be inserted into the next row
\item[(H4)] If replacing $y$ with $x$ does not result in an increasing tableau, then do not change row $R$. 
In this case, $y$ is the output integer to be inserted into the next row.
\end{itemize}

\begin{example}\text{ }\\
\begin{center}
\begin{ytableau}
  $1$ & $2$ & $3$ & $5$ \\
$2$ & $3$ & $4$ & $6$ \\
$6$ \\
$7$ 
 \end{ytableau}\text{ }${\overset{H}{\longleftarrow}} 3$\text{ } $=$ \text{ }
\begin{ytableau}
  $1$ & $2$ & $3$ & $5$ \\
$2$ & $3$ & $4$ & $6$ \\
$6$ \\
$7$
 \end{ytableau} \end{center}
We use rule (H4) in the first row to obtain output integer $5$. Notice that the $5$ cannot replace the $6$ 
in the second row since it would be directly below the $5$ in the first row. 
Thus we use (H4) again and get output integer $6$. Since we cannot add this $6$ to the end of the third row, 
we use (H2) and get $c=(4,1)$. Notice that the shape did not
change in this insertion, so $\alpha=0$.
\end{example}
\begin{example}  \label{ex:ins2} \text{ }\\
\begin{center}
\begin{ytableau} 
 $2$ & $4$ & $6$ \\
$3$ & $6$ & $8$ \\
$7$
\end{ytableau}\text{ }${\overset{H}{\longleftarrow}} 5$\text{ } $=$ \text{ }
\begin{ytableau} 
 $2$ & $4$ & $5$ \\
$3$ & $6$ & $8$ \\
$7$ & $8$
\end{ytableau}
\end{center}
The integer $5$ bumps the $6$ from the first row using (H3). The $6$ is inserted into the second row, 
which already contains a $6$. Using (H4), the second row remains unchanged
and we insert $8$ into the third row. Since $8$ is larger than everything in the third row, we use (H1) 
to adjoin it to the end of the row. Thus $\alpha=1$.
\end{example}

Using this insertion algorithm, we define the \textit{Hecke insertion tableau} of a word $w=w_1w_2\cdots w_n$ to be 
$$P_H(w)=(\ldots((\varnothing\overset{H}{\longleftarrow} w_1)\overset{H}{\longleftarrow} w_2)\ldots) \overset{H}{\longleftarrow} w_n.$$

In \cite{BKSTY}, Buch, Kresch, Shimozono, Tamvakis, and Yong give the following algorithm for reverse Hecke 
insertion starting with the triple $(Z,c,\alpha)$ as described above and ending with a pair $(Y,x)$ consisting of an 
increasing tableau and a postive integer. 

\begin{itemize}
\item[(rH1)] If $y$ is the cell in square $c$ of $Z$ and $\alpha=1$, 
then remove $y$ and reverse insert $y$ into the row above.

\item[(rH2)] If $\alpha=0$, do not remove $y$, 
but still reverse insert it into the row above.
\end{itemize}
In the row above, let $x$ be the largest integer 
such that $x<y$. 

\begin{itemize}
\item[(rH3)] If replacing $x$ with $y$ results in an 
increasing tableau, then we replace $x$ with $y$ and reverse insert $x$ into the row above.

\item[(rH4)] If replacing
$x$ with $y$ does not result in an increasing tableau, leave the row unchanged and reverse insert $x$ into the row
above.

\item[(rH5)] If $R$ is the first row of the tableau, the final output consists of $x$ and the modified tableau.
\end{itemize}

\begin{theorem} \cite[Theorem 4]{BKSTY}\label{thm:insbijection}
 Hecke insertion $(Y,x)\mapsto (Z,c,\alpha)$ and reverse Hecke insertion $(Z,c,\alpha)\mapsto (Y,x)$ define
mutually inverse bijections between the set of pairs consisting of an increasing tableau and a positive integer and the
set of triples consisting of an increasing tableau, a corner cell of the increasing tableau, and $\alpha\in\{0,1\}$.
\end{theorem}

We next describe the tableaux that will act as recording tableau for Hecke insertion.

A \textit{set-valued tableau} $T$ of shape $\lambda$ is a filling of the boxes
with finite, non-empty subsets of positive integers so that
\begin{enumerate}
 \item the smallest number in each box is greater than or equal to the largest
number in the box directly to the left of it (if that box is present), and
\item the smallest number in each box is strictly greater than the largest number 
in the box directly above it (if that box is present).
\end{enumerate}

We call of set-valued tableau \textit{standard} if it contains exactly the integers $[n]$ for some $n$.

Given a word $w=w_1w_2\ldots w_l$, we can associate a pair of tableaux $(P_H(w), Q_H(w))$, where $P_H(w)$ is the Hecke insertion tableau described previously 
and $Q_H(w)$ is a standard set-valued tableau called the \textit{Hecke recording tableau} obtained as follows. Start with $Q_H(\varnothing)=\varnothing$. At each step of the insertion of $w$, let $Q_H(w_1\ldots w_k)$ be 
obtained from $Q_H(w_1\ldots w_{k-1})$ by labeling the special corner, $c$, in the insertion of $w_k$ into $P_H(w_1\ldots w_{k-1})$ with the positive integer $k$. Then 
$Q_H(w)=Q_H(w_1 w_2\ldots w_l)$ is the resulting standard set-valued tableau.

\begin{example}
 Let $w$ be $15133$. We obtain $(P_H(w),Q_H(w))$ with the following sequence, where in column $k$, $Q_H(w_1 \ldots w_k)$ is shown below $P_H(w_1 \ldots w_k)$. 
\begin{center}
\begin{ytableau}
 1
\end{ytableau}\hspace{.3in}
\begin{ytableau}
 1 & 5
\end{ytableau}\hspace{.3in}
\begin{ytableau}
 1 & 5 \\
5
\end{ytableau}\hspace{.3in}
\begin{ytableau}
 1 & 3 \\
5
\end{ytableau}\hspace{.3in}
\begin{ytableau}
 1 & 3 \\
5
\end{ytableau} 
$=P_H(w)$
\end{center}
\vspace{.2in}
\begin{center}
\begin{ytableau}
 1
\end{ytableau}\hspace{.3in}
\begin{ytableau}
 1 & 2
\end{ytableau}\hspace{.3in}
\begin{ytableau}
 1 & 2 \\
3
\end{ytableau} \hspace{.3in}
\begin{ytableau}
 1 & 2 \\
34
\end{ytableau}\hspace{.3in}
\begin{ytableau}
 1 & 25 \\
34
\end{ytableau}
$=Q_H(w)$
\end{center}
\end{example}

Call a word $w$ {\it {initial}} if the letters appearing in it are exactly the numbers in $[k]$ for some positive integer $k$. 

\begin{example}
 The word $13422$ is initial since the letters appearing in it are the numbers from $1$ to $4$. On the other hand, the word $1422$ is not initial 
because the letters appearing in it are $1$, $2$, and $4$ and do not form a set $[k]$ for any $k$. 
\end{example}

As with RSK, Hecke insertion gives a useful bijection.

\begin{theorem} \cite{PP}
The map sending $w=w_1w_2\cdots w_n$ to $(P_H(w),Q_H(w))$ is a bijection between words and ordered pairs of tableaux of the same shape $(P,Q)$, where $P$ is an increasing tableau 
and $Q$ is a set-valued tableau with entries $\{1,2,\ldots,n\}$. It is also a bijection if there is an extra condition of being initial imposed both on $w$ and $P$.
\end{theorem}

\begin{remark}
Hecke insertion is closely related to the $K$-theoretic jeu de taquin algorithm introduced in \cite{ThY}. This relationship is explored in \cite{BuchSam,ThY2}
\end{remark}

\subsection{Hecke growth and decay}\label{sec:HeckeGrowth}
Given any word, $h=h_1h_2\ldots h_k$ containing
$n \leq k$ distinct numbers, we can create an $n \times k$ array with an $X$ in the $h_i^{th}$
square from the bottom of column $i$. Note that there can be multiple $X$'s in the same row but is at most
one $X$ per column. For example, if $h=121331$, we have 

\begin{center}
\begin{ytableau}
 $ $ & $ $ & $ $ & X & X & $ $ \\
 $ $ & X & $ $ & $ $ & $ $ & $ $ \\
 X & $ $ & X & $ $ & $ $ & X 
\end{ytableau}
\end{center}

We will label each of the four corners of each square with a partition and some of the horizontal edges of the squares with
a specific inner corner of the partition on the left vertex of the edge, which will indicate the box where the insertion terminates. We do this by recording the row 
of the inner corner at which the insertion terminates. We begin by labeling all corners
along the bottom row and left side of the diagram with the partition $\varnothing$.

\begin{center}
\begin{tikzpicture}[scale=1]
\node (00) at (0,0) {$\varnothing$};
\node (10) at (1,0) {$\varnothing$};
\node (20) at (2,0) {$\varnothing$};
\node (30) at (3,0) {$\varnothing$};
\node (40) at (4,0) {$\varnothing$};
\node (50) at (5,0) {$\varnothing$};
\node (01) at (0,1) {$\varnothing$};
\node (.5.5) at (.5,.5) {$X$};
\node (11) at (1,1) {$ $};
\node (21) at (2,1) {$ $};
\node (31) at (3,1) {$ $};
\node (41) at (4,1) {$ $};
\node (51) at (5,1) {$ $};
\node (02) at (0,2) {$\varnothing$};
\node (12) at (1,2) {$ $};
\node (22) at (2,2) {$ $};
\node (32) at (3,2) {$ $};
\node (1.51.5) at (1.5,1.5) {$X$};
\node (42) at (4,2) {$ $};
\node (52) at (5,2) {$ $};
\node (03) at (0,3) {$\varnothing$};
\node (13) at (1,3) {$ $};
\node (23) at (2,3) {$ $};
\node (33) at (3,3) {$ $};
\node (43) at (4,3) {$ $};
\node (53) at (5,3) {$ $};
\node (2.5.5) at (2.5,.5) {$X$};
\node (3.52.5) at (3.5,2.5) {$X$};
\node (4.53.5) at (4.5,2.5) {$X$};
\node (5,51.5) at (5.5,.5) {$X$};
\node (60) at (6,0) {$\varnothing$};
\node (61) at (6,1) {$ $};
\node (62) at (6,2) {$ $};
\node (63) at (6,3) {$ $};

\draw (00)--(10)--(20)--(30)--(40)--(50)--(60)
(01)--(11)--(21)--(31)--(41)--(51)--(61)
(02)--(12)--(22)--(32)--(42)--(52)--(62)
(03)--(13)--(23)--(33)--(43)--(53)--(63)

(00)--(01)--(02)--(03)
(10)--(11)--(12)--(13)
(20)--(21)--(22)--(23)
(30)--(31)--(32)--(33)
(40)--(41)--(42)--(43)
(50)--(51)--(52)--(53)
(60)--(61)--(62)--(63);
\end{tikzpicture}
\end{center}

To complete the labeling of the corners, suppose the corners $\mu$, $\lambda$, and $\nu$ are labeled, 
where $\mu$, $\lambda$, and $\nu$ are as in the picture below. We label $\gamma$ according to the following rules.

 \begin{center}
\begin{tikzpicture}[scale=1]
\node (A) at (-1,-1) {$\lambda$};
\node (B) at (1,-1) {$\nu$};
\node (C) at (-1,1) {$\mu$};
\node (D) at (1,1) {$\gamma$};
\draw (A) -- (B)
(C) -- (D)
(B) -- (D)
(A) -- (C);
\end{tikzpicture}
\end{center}

\textbf{If the square contains an X:}
\begin{itemize}
 \item[(1)] If $\mu_1=\nu_1$, then $\gamma/\mu$ consists of one box in row 1.
\item[(2)] If $\mu_1 \neq \nu_1$, then $\gamma=\mu$ and the edge between then is labeled by 
the row of the highest inner corner of $\mu$. 
\end{itemize}

\textbf{If the square does not contain an X and if either $\mu=\lambda$ or $\nu=\lambda$ with no label between $\lambda$ and $\nu$}:
\begin{itemize}
 \item[(3)] If $\mu=\lambda$, then set $\gamma = \nu$. If $\nu=\lambda$, then $\gamma = \mu$.  
\end{itemize}

\textbf{If $\nu \nsubseteq \mu$ and the square does not contain an X:}
\begin{itemize}
 \item[(4)] In the case where $\nu \nsubseteq \mu$, $\gamma = \nu \cup \mu$.  
\end{itemize}

\textbf{If $\nu \subseteq \mu$ and the square does not contain an X:}
\begin{itemize}
 \item[(5)] If $\nu/\lambda$ is one box in row $i$ and $\mu/\nu$ has no boxes in row $i+1$, 
then $\gamma/\mu$ is one box in row $i+1$.

\item[(6)] If $\nu/\lambda$ is one box in row $i$ and $\mu/\nu$ has a box in row $i+1$, then $\gamma=\mu$ and the edge between then is labeled $i+1$.

\item[(7)] If $\nu=\lambda$, the edge between them is labeld $i$, and there are no boxes of $\mu/\nu$
immediately to the right or immediately below this inner corner of $\nu$ in row $i$, then $\gamma=\mu$ with the edge between them labeled $i$.

\item[(8)] If $\nu=\lambda$, the edge between them is labeled $i$, and there is a box of $\mu/\nu$ directly below
this inner corner of $\nu$ in row $i$, then $\gamma=\mu$ with the edge between them labeled $i+1$.

\item[(9)] If $\nu=\lambda$, the edge between them is labeled $i$, and there is a box of $\mu/\nu$ immediately 
to the right of this inner corner of $\nu$ in row $i$ but no box of $\mu/\nu$ in row $i+1$, then $\gamma/\mu$ is one box
in row $i+1$.

\item[(10)] If $\nu=\lambda$, the edge between them is labeled $i$, and there is a box of $\mu/\nu$ immediately 
to the right of this inner corner of $\nu$ in row $i$ and a box of $\mu/\nu$ in row $i+1$, then $\gamma=\mu$ with the edge between them 
labeled $i+1$. 

\end{itemize}

We call the resulting array the \textit{Hecke growth diagram} of $h$. In our previous example with
$h=121342$, we would have the diagram below.

\begin{center}
\begin{tikzpicture}[scale=1]
\node (00) at (0,0) {$\varnothing$};
\node (10) at (1,0) {$\varnothing$};
\node (20) at (2,0) {$\varnothing$};
\node (30) at (3,0) {$\varnothing$};
\node (40) at (4,0) {$\varnothing$};
\node (50) at (5,0) {$\varnothing$};
\node (01) at (0,1) {$\varnothing$};
\node (.5.5) at (.5,.5) {$X$};
\node (11) at (1,1) {$1$};
\node (21) at (2,1) {$1$};
\node (31) at (3,1) {$1$};
\node (41) at (4,1) {$1$};
\node (51) at (5,1) {$1$};
\node (02) at (0,2) {$\varnothing$};
\node (12) at (1,2) {$1$};
\node (22) at (2,2) {$2$};
\node (32) at (3,2) {$21$};
\node (1.51.5) at (1.5,1.5) {$X$};
\node (42) at (4,2) {$21$};
\node (52) at (5,2) {$21$};
\node (03) at (0,3) {$\varnothing$};
\node (13) at (1,3) {$1$};
\node (23) at (2,3) {$2$};
\node (33) at (3,3) {$21$};
\node (43) at (4,3) {$31$};
\node (53) at (5,3) {$31$};
\node (2.5.5) at (2.5,.5) {$X$};
\node (3.52.5) at (3.5,2.5) {$X$};
\node (4.53.5) at (4.5,2.5) {$X$};
\node (5,51.5) at (5.5,.5) {$X$};
\node (60) at (6,0) {$\varnothing$};
\node (61) at (6,1) {$1$};
\node (62) at (6,2) {$21$};
\node (63) at (6,3) {$31$};

\node (2.51.1) at (2.5,1.15) {\tiny 1};
\node (5.51.1) at (5.5,1.15) {\tiny 1};
\node (5.52.1) at (5.5,2.15) {\tiny 2};
\node (4.53.1) at (4.5,3.15) {\tiny 1};
\node (5.53.1) at (5.5,3.15) {\tiny 2};

\draw (00)--(10)--(20)--(30)--(40)--(50)--(60)
(01)--(11)--(21)--(31)--(41)--(51)--(61)
(02)--(12)--(22)--(32)--(42)--(52)--(62)
(03)--(13)--(23)--(33)--(43)--(53)--(63)

(00)--(01)--(02)--(03)
(10)--(11)--(12)--(13)
(20)--(21)--(22)--(23)
(30)--(31)--(32)--(33)
(40)--(41)--(42)--(43)
(50)--(51)--(52)--(53)
(60)--(61)--(62)--(63);
\end{tikzpicture}
\end{center}

Let $\mu_0 = \varnothing \subseteq \mu_1 \subseteq \ldots \mu_k$ be the seqence of partitions across the top of the Hecke
growth diagram, and let $\nu_0=\varnothing \subseteq \nu_1 \subseteq \ldots \nu_n$ be the sequence of partitions 
on the right side of the Hecke growth diagram. These sequences correspond to increasing tableaux
$Q(h)$ and $P(h)$, respectively. If the edge between $\mu_i$ and $\mu_{i+1}$ is labeled $j$, then $\mu_i=\mu_{i+1}$, and the label $i+1$ of $Q(h_1\cdots h_{i+1})$ 
is placed in the box at the end of row $j$ of $Q(h_1\cdots h_i)$. In the example above, we have 

\begin{center}
 $P(h)=$
 \begin{ytableau}
  1 & 2 & 3  \\
  2 
 \end{ytableau}\hspace{1in}
 $Q(h)=$
 \begin{ytableau}
  1 & 2 & 45 \\
  36
 \end{ytableau}
\end{center}

\begin{theorem}
For any word $h$, the increasing tableau $P(h)$ and set-valued tableau $Q(h)$ obtained from the sequence of partitions across the right side of the Hecke growth diagram for $h$
and across the top of the Hecke growth diagram for $h$, respectively, are $P_H(h)$ and $Q_H(h)$, the Hecke insertion and recording tableau for $h$.
\end{theorem}

For the following proof, again note that the loops at shape $\lambda$ may be indexed by inner corners of $\lambda$.

\begin{proof}
 Suppose that the square described in rules 1 through 10 is in row $t$ and column $s$. We will argue the result by induction. 

\begin{lemma} Oriented as in the square above, $|\nu/\lambda|\leq 1$.\end{lemma}
\begin{proof}
By the induction hypothesis, $\lambda$ is the recording tableau
after inserting the numbers in columns $1$ through $s-1$, and $\nu$ is the insertion tableau after inserting the 
number in column $s$. Since we are only inserting one number, $|\nu/\lambda|$ is at most 1.
\end{proof}

\begin{lemma} Oriented as in the square above, $\mu/\lambda$ is a rook strip, that is, no two boxes 
in $\mu/\lambda$ are in the same row or column.\end{lemma}
\begin{proof} By induction, $\mu/\lambda$ represents the positions of boxes filled with 
$t$ at some point in the insertion. Since insertion results in an increasing tableau, $\mu/\lambda$ must be a rook
strip.
\end{proof}

\begin{itemize}
 \item[(1)] First note that in any square with an X, $\lambda=\nu$ since there is exactly one square in each column
 and that an X in the square we are considering corresponds to inserting a $t$ into the tableau of shape $\lambda$.
Since, by the induction hypothesis, shapes along columns represent the insertion tableau, $\mu_1=\nu_1=\lambda_1$ 
means that before adding the $t$, there are no $t$'s in the first row of the insertion tableau.
Since $t$ is weakly the largest number inserted to this point, inserting the $t$ will result in a 
$t$ being added to the first row of the insertion tableau. Thus $\gamma=(\mu_1+1,\mu_2,\mu_3,\ldots)$.

\item[(2)] If $\mu_1\neq\nu_1=\lambda_1$, then after inserting the numbers in columns $1$ through $s-1$, there is
a $t$ at the end of the first row of the insertion tableau. It follows that inserting another $t$ will result in 
this $t$ merging with the $t$ at the end of the first row, and the special corner in this case becomes 
the bottom box of the last column, or in other words, the highest inner corner of $\mu$.

\item[(3)] If $\mu=\lambda$, then there is no X in row $t$ in columns $1$ to $s$. Thus there is no $t$ to add, and so
the insertion tableau $\nu$ should not change. Similarly for if $\nu=\lambda$.

\item[(4)] Suppose that $\nu/\lambda$ is one box in row i and $\mu/\lambda$ is a rook strip containing boxes in rows
$j_1,j_2,\ldots ,j_k$. Then $\nu \nsubseteq \mu$ implies that $i\notin \{j_1,j_2,\ldots ,j_k\}$. Since $\nu/\lambda$ 
is one box, the last action in the insertion sequence of $r$ into the insertion tableau of shape $\lambda$ 
is a box being added in row $i$. Since there is not a $t$ in row $i$, the bumping sequence when inserting $r$ 
into the insertion tableau of shape $\mu$ will not disturb the $t$'s.

\item[(5)] Suppose $\nu$ is $\lambda$ plus one box in position $(i,j)$. Then since $\nu \subseteq \mu$, $\mu/\lambda$ contains the box at $(i,j)$.
It follows that inserting $r$ into the insertion tableau of shape $\mu$ will result in bumping the $t$ in position $(i,j)$ since the bumping path 
of $r$ inserting in the insertion tableau of shape $\lambda$ ends by adding a box at $(i,j)$. Since there is no box in row $i+1$ of $\mu/\lambda$,
everything in row $i+1$ of the insertion tableau of shape $\mu$ is strictly less than $t$. It follows that the $t$ bumped from $(i,j)$ can be added 
to the end of row $i+1$.

\item[(6)] This is almost identical to the proof of Rule 4. Since there is now a box in row $i+1$ of $\mu/\lambda$, there is a $t$ at the end of row
$i+1$ of the insertion tableau of shape $\mu$. Inserting $r$ will bump the $t$ in row $i$ as above, but now the $t$ will merge with the $t$ in row 
$i+1$. Thus the special corner is now at the inner corner of $\mu$ in row $i+1$.

\item[(7)] There are two ways that inserting $r$ into the insertion tableau of shape $\mu$ will involve the $t$'s.
Assume $\nu$ is obtained from $\lambda$ by taking a loop in row $i$ and column $j$. 
Assume rows $k$ through $i$ all have length $j$.

Case 1: A box containing $y$ was bumped from row $k-1$, merged with itself in row $k$, and there was a $t$ to the right
of the $y$. Note that $k\neq i$. Then the $y$ duplicates and bumps the $t$ in row $k$, but this $t$ cannot be added to row
$k+1$ as it would be directly below the original $t$. Thus, by the rules of Hecke insertion, the special
corner is the box at the bottom of column $j$, box $(i,j)$. See the left-hand figure below where the dot indicates
the box $(i,j)$.

Case 2: Consider the situation in the right-hand figure below where the dot indicates
the box $(i,j)$. Assume that during the insertion of $r$, some
number merged with itself to the left of $z$ in row $k-1$ and $z$ is duplicated bumped to the next row.
Following the rules of Hecke insertion, $z$ bumps the $t$ in row $k$, but as it cannot replace the $t$, the $t$
remains at the end of row $k$. The $t$ that was bumped cannot be added to end of row $k+1$, so no new boxes
are created in the insertion and the special corner is the box at the bottom of column $j$, box $(i,j)$.

\begin{center}
\includegraphics[scale=0.7]{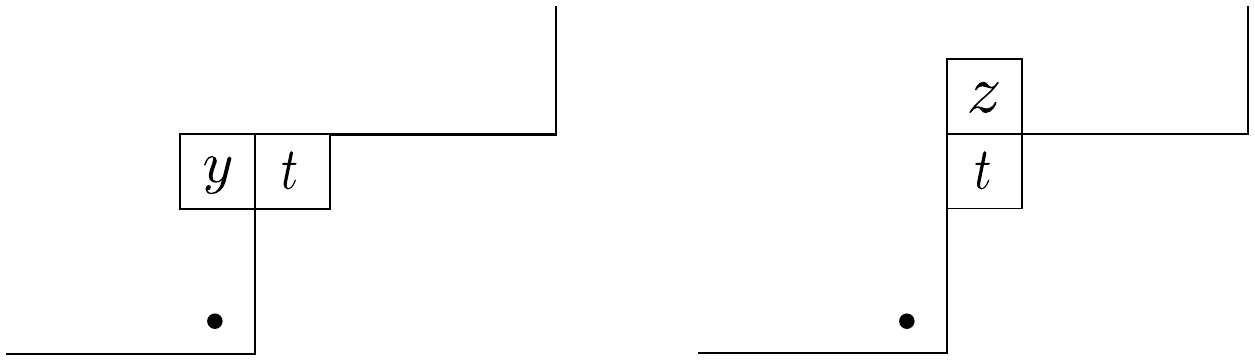}
\end{center}

\item[(8)] There are three ways that inserting $r$ into the insertion tableau of shape $\mu$ will involve the $t$'s.
Assume $\nu$ is obtained from $\lambda$ by taking a loop in row $i$ and column $j$. 
Assume rows $k$ through $i$ all have length $j$.

Case 1: If there is a $t$ in square $(i,j+1)$, we are in the situation described by Rule 9.

Cases 2 and 3 are described in the proof of Rule 7. The difference is that the bottom of column $j$ after
we insert $r$ is now the box $(i+1,j)$, so the special corner in each case will be $(i+1,j)$, as desired.

\begin{center}
\includegraphics[scale=0.7]{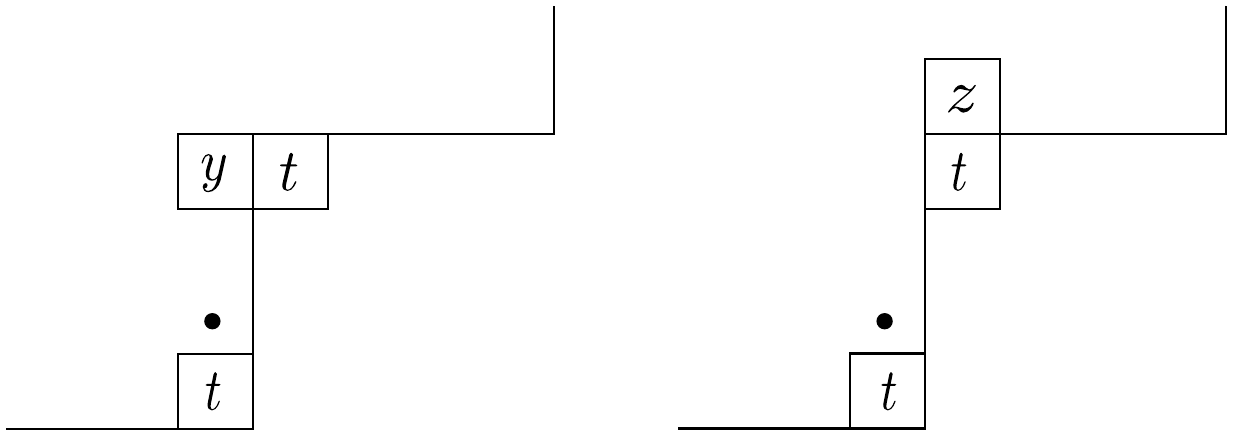}
\end{center}

\item[(9)] Note that a loop in row $i$ with a box of $\mu/\nu$ to its right implies that $\nu_{i-1}>\nu_i>\nu_{i+1}$.
There are two ways for the special corner of the insertion of $r$ into the insertion tableau of shape $\lambda$ could
have been $(i,j)$.

Case 1: Some $y$ was bumped from row $i-1$ and merged with itself in box $(i,j)$. Since there is a $t$ to the right
of box $(i,j)$ in $\mu$, this will result in duplicating and bumping $t$. Since there is nothing in row $i+1$
of $\mu/\nu$, everything in row $i+1$ of $\mu$ is strictly less than $t$. Hence the $t$ bumped from row $i$
can be added to the end of row $i+1$.

Case 2: As in the figure below, $z$ was duplicated and bumped from box $(i-1,j+1)$. This $z$ cannot replace the $t$ in 
box $(i,j+1)$ since it would be directly below the orginal $z$, but it bumps the $t$ to row $i+1$. 
Since there is nothing in row $i+1$
of $\mu/\nu$, everything in row $i+1$ of $\mu$ is strictly less than $t$. Hence the $t$ bumped from row $i$
can be added to the end of row $i+1$.

\begin{center}
\includegraphics[scale=0.7]{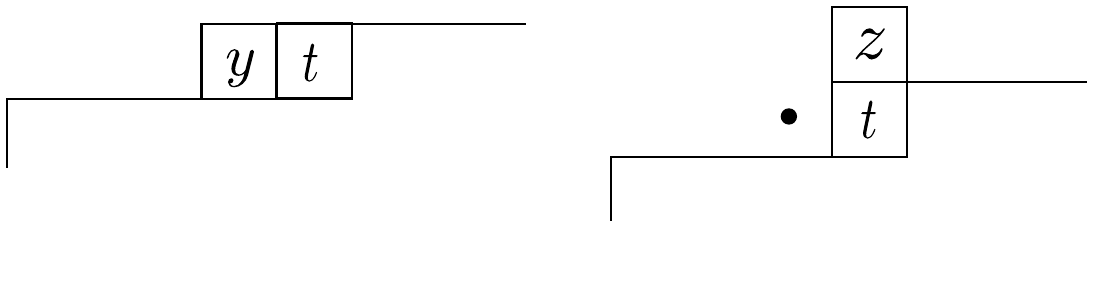}
\end{center}

\item[(10)] This is similar to the situation described in the proof of Rule 9. The difference is that 
a box in row $i+1$ of $\mu/\nu$ means there is already a $t$ at the end of row $i+1$ of $\mu$, so each insertion
will end with the bumped $t$ merging with itself at the end of row $i+1$.

\begin{center}
\includegraphics[scale=0.7]{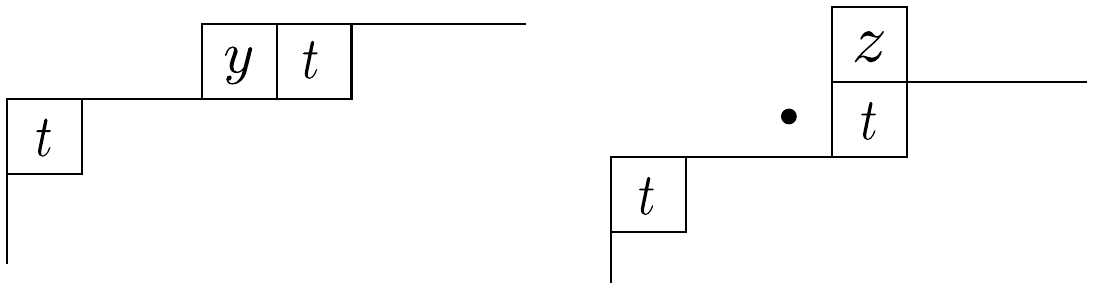}
\end{center}
\end{itemize}
\end{proof}

\subsection{M\"obius via Pieri} \label{ex:G}
The following shows that we have an instance of the M\"obius via Pieri phenomenon in the case of Young's lattice. Namely, the M\"obius deformation of Young's lattice is 
obtained by the Pieri construction from an appropriate $K$-theoretic deformation of the underlying Hopf algebra. 

Let $A$ be the subring of the completion $\hat\Lambda$ of the ring of symmetric functions with basis $\{G_\lambda\}$, the stable 
Grothendieck polynomials. For details, see \cite{B, LP}. Define the \textit{signless stable Grothendieck polynomials} $\tilde G_\lambda$ by omitting
the sign. Note that these coincide with the $\tilde K_\lambda$ defined in \cite{LP}. The structure constants will coincide with those of the 
stable Grothendieck polynomials up to sign. We recall these structure constants here.

Let $S_\mu$ be the superstandard tableau of shape $\mu$, that is, 
the first row of $S$ is filled with $1,2,\ldots,\mu_1$, the second row with $\mu_1+1,m_1+2,\ldots,\mu_1+\mu_2$, etc. 
This space has a bialgebra product structure given by 
$$\tilde G_\lambda \tilde G_\mu=\displaystyle\sum c_{\lambda,\mu}^\nu \tilde G_\nu,$$ where $c_{\lambda,\mu}^\nu$ is the number of 
increasing tableaux $R$ of skew shape $\nu/\lambda$ such that $P_H(\r(R))=S_\mu$. The coproduct structure is given by
$$\Delta(\tilde G_\nu)=\displaystyle\sum_{\lambda,\mu} d_{\lambda,\mu}^\nu \tilde G_\lambda \otimes \tilde G_\mu,$$ 
where $d_{\lambda,\mu}^\nu$ is the number of increasing tableaux $R$ of skew shape $\lambda \oplus \mu$ such that $P_H(\r(R))=S_\nu$, 
see \cite{BuchSam, PP,ThY3}. Here $\lambda\oplus\mu$ denotes the skew shape obtained by joining the rightmost top corner of $\lambda$ to the leftmost
bottom corner of $\mu$. For example, $(2,1)\oplus (2,2)=(4,4,2,1)/(2,2)$ is shown below.

\begin{center}\begin{ytableau} \none & \none & $ $ & $ $ \\ \none & \none & $ $ & $ $ \\ $ $ & $ $ \\ $ $ \end{ytableau}\end{center}

We consider the following Pieri construction in $A$. 
Let $g=\tilde G_1$, and define an operator $D$ by $D(\tilde G_\nu)=\xi(\Delta(\tilde G_\nu))$. We define a graph $G$, where elements are partitions, 
$a_1(\mu,\lambda)$ is the 
coefficient of $\tilde G_\mu$ in $D(\tilde G_\lambda)$, and
$a_2(\mu,\lambda)$ is the coefficient of $\tilde G_\lambda$ in $\tilde G_\mu \tilde G_1$.

\begin{proposition} \label{prop:mobpieriyoung}
The resulting graph coincides with the one obtained via M\"obius construction in Section \ref{sec:MobiusYoung}.
\end{proposition}

\begin{minipage}{.5\textwidth}
 \begin{center}
\begin{tikzpicture}[scale=1]
\node (empty) at (0,0) {$\varnothing$};
\node (1) at (0,1) {\ytableausetup{boxsize=.15cm}\begin{ytableau}$ $ \end{ytableau}};
\node (2) at (1,2) {\begin{ytableau}$ $ & $ $  \end{ytableau}};
\node (11) at (-1,2) {\begin{ytableau}$ $ \\ $ $  \end{ytableau}};
\node (111) at (-2,3) {\begin{ytableau}$ $ \\ $ $ \\ $ $ \end{ytableau}};
\node (21) at (0,3) {\begin{ytableau}$ $ & $ $ \\ $ $ \end{ytableau}};
\node (3) at (2,3) {\begin{ytableau}$ $ & $ $ & $ $ \end{ytableau}};
\node (1111) at (-3,4) {\begin{ytableau}$ $ \\ $ $ \\ $ $ \\ $ $\end{ytableau}};
\node (211) at (-1,4) {\begin{ytableau}$ $ & $ $ \\ $ $ \\ $ $\end{ytableau}};
\node (22) at (0,4) {\begin{ytableau}$ $ & $ $ \\ $ $ & $ $ \end{ytableau}};
\node (31) at (1,4) {\begin{ytableau}$ $ & $ $ & $ $ \\ $ $\end{ytableau}};
\node (4) at (3,4) {\begin{ytableau}$ $ & $ $ & $ $ & $ $\end{ytableau}};
\draw (empty) -- (1)
(1) -- (2)--(3)--(4)
(1)--(11)--(111)--(1111)
(11)--(21)--(31)
(111)--(211)
(3)--(31)
(21)--(22)
(2)--(21)--(211)
(1) to [out=60,in=100,distance=.5cm] (1)
(2) to [out=60,in=100,distance=.5cm] (2) 
(11) to [out=60,in=100,distance=.5cm] (11)
(3) to [out=60,in=100,distance=.5cm] (3) 
(4) to [out=60,in=100,distance=.5cm] (4) 
(111) to [out=60,in=100,distance=.5cm] (111) 
(1111) to [out=60,in=100,distance=.5cm] (1111)
(21) to [out=-60,in=-100,distance=.5cm] (21) 
(21) to [out=-60,in=-100,distance=.7cm] (21) 
(211) to [out=60,in=100,distance=.5cm] (211) 
(211) to [out=60,in=100,distance=.7cm] (211) 
(31) to [out=60,in=100,distance=.5cm] (31) 
(31) to [out=60,in=100,distance=.7cm] (31) 
(22) to [out=60,in=100,distance=.5cm] (22) ;
\end{tikzpicture}
\end{center}
\end{minipage}
\begin{minipage}{.5\textwidth}
 \begin{center}
\begin{tikzpicture}[scale=1]
\node (empty) at (0,0) {$\varnothing$};
\node (1) at (0,1) {\ytableausetup{boxsize=.15cm}\begin{ytableau}$ $ \end{ytableau}};
\node (2) at (1,2) {\begin{ytableau}$ $ & $ $  \end{ytableau}};
\node (11) at (-1,2) {\begin{ytableau}$ $ \\ $ $  \end{ytableau}};
\node (111) at (-2,3) {\begin{ytableau}$ $ \\ $ $ \\ $ $ \end{ytableau}};
\node (21) at (0,3) {\begin{ytableau}$ $ & $ $ \\ $ $ \end{ytableau}};
\node (3) at (2,3) {\begin{ytableau}$ $ & $ $ & $ $ \end{ytableau}};
\node (1111) at (-3,4) {\begin{ytableau}$ $ \\ $ $ \\ $ $ \\ $ $\end{ytableau}};
\node (211) at (-1,4) {\begin{ytableau}$ $ & $ $ \\ $ $ \\ $ $\end{ytableau}};
\node (22) at (0,4) {\begin{ytableau}$ $ & $ $ \\ $ $ & $ $ \end{ytableau}};
\node (31) at (1,4) {\begin{ytableau}$ $ & $ $ & $ $ \\ $ $\end{ytableau}};
\node (4) at (3,4) {\begin{ytableau}$ $ & $ $ & $ $ & $ $\end{ytableau}};
\draw (empty) -- (1)
(1) -- (2)--(3)--(4)
(1)--(11)--(111)--(1111)
(11)--(21)--(31)
(111)--(211)
(3)--(31)
(21)--(22)
(2)--(21)--(211)
(21)--(1)
(31)--(2)
(211)--(11);
\end{tikzpicture}
\end{center}
\end{minipage}
\ytableausetup{boxsize=.5cm}

\begin{proof}
To verify this claim, we first show why the coefficient of $\tilde G_\nu$ in $\tilde G_\lambda \tilde G_1$ is 1 if $\nu/\lambda$ is a rook 
strip and is 0 otherwise. This follows from that fact that the only words that Hecke insert into $S_1$ are words consisting only of the letter 
$1$. Thus, $\r(R)=11\cdots 1$ for any $R$ contributing to $c_{\lambda,1}^\nu$. Increasing tableaux of shape $\nu/\lambda$ with row word $11\cdots 1$ 
are precisely rook strips. 

Secondly, it must be true that $d_{\lambda,1}^\lambda$ is the number of inner corners of $\lambda$ and if $\nu\neq\lambda$, $d_{\lambda,1}^\nu$ is 1 if 
$\nu$ is obtained from $\lambda$ by adding one box and 0 otherwise. Suppose $\nu\neq\lambda$ and consider all 
skew tableaux $R$ of shape $\lambda \oplus 1$ such that $P_H(\r(R))=S_\nu$. To obtain such $\lambda$, simply reverse insert an entry of 
$S_\nu$ with $\alpha=1$ and use the output integer to fill the single box in $\lambda\oplus 1$. The resulting shapes are the shapes 
obtained from $\nu$ by deleting one inner corner of $\mu$. Similarly, if $\nu=\lambda$, we reverse insert each inner corner of $S_\lambda$ using $\alpha=0$
and then let $R$ be $S_\lambda\oplus x$, where $x$ is the result of the reverse insertion.
\end{proof}

\section{Major examples: shifted Young's lattice} \label{sec:sY}

\subsection{Pieri deformation of shifted Young's lattice}

Let $A = \Lambda'$ be the subring of the ring of symmetric functions generated by the odd power sums $p_1, p_3, \ldots$. 
Let $Q_{\lambda}$ and $P_{\lambda}$ be the bases of Schur $Q$ and Schur $P$ functions for this ring. Here $\lambda$ varies over the set 
$P$ of partitions with distinct parts. We refer the reader to \cite{Mac} for background. 
Let $q_i$ be defined by $$q_i = \sum_{0 \leq j \leq i} e_j h_{i-j},$$ and let $f = q_1 + q_2 + \ldots.$

Define up and down edges of a filtered graph $G = (P, \rho, E_1, E_2)$ by letting 
$a_2(\mu, \nu)$ be the coefficient of $Q_\mu$ in $fQ_\nu$ and $a_1(\mu,\nu)$ be the coefficient of $P_\mu$ in $p_1P_\nu$.

We will fill a shifted partition $\lambda$ with ordered alphabet $1'<1<2'<2<3'<3\ldots$ according to the usual rules: 
\begin{itemize}
 \item The filling must be weakly increasing in rows and columns. 
 \item There can be at most one instance of $k'$ in any row. 
 \item There can be at most one instance of $k$ in any column.
 \item There can be no primed entries on the main diagonal. 
\end{itemize}

For example, the shifted partition below is filled according to the rules stated above.
\begin{center}
 \begin{ytableau}
  1 & 2' & 2 & 2 & 4' \\
  \none & 2 & 3' & 5' & 5 \\
  \none & \none & 4 & 5' \\
  \none & \none & \none & 6
 \end{ytableau}

\end{center}

Given two shifted shapes, $\mu$ and $\nu$, we say that $\mu/\nu$ forms a \textit{border strip} if it contains no 2-by-2 square. 

\begin{lemma}
The Pieri deformation of the shifted Young's lattice is formed by adding downward-oriented edges from $\mu$ to $\nu$ whenever $\mu/\nu$ forms a border strip. 
The number of such edges added is the same as the number of ways to fill the boxes of the border strip with $k$ and $k'$ according to the usual rules. 
\end{lemma}

\begin{proof}
 Follows from formula (8.15) in \cite{Mac}, which is an analogue of the Pieri rule for Schur $P$ functions, and the fact that $q_1 = 2 p_1$. 
\end{proof}

The first six ranks of the Pieri deformation are shown below.

\begin{minipage}{.5\textwidth}
\begin{center}
\begin{tikzpicture}
       \node (empty) at (0,0) {$\varnothing$};
  \node (1) at (0,1) {\ytableausetup{boxsize=.15cm}\begin{ytableau}$ $ \end{ytableau}};
  \node (2) at (0,2) {\begin{ytableau}$ $ & $ $  \end{ytableau}};
  \node (3) at (1,3) {\begin{ytableau}$ $ & $ $ & $ $ \end{ytableau}};
  \node (21) at (-1,3) {\begin{ytableau}$ $ & $ $ \\ \none & $ $ \end{ytableau}};
  \node (31) at (0,4) {\begin{ytableau}$ $ & $ $ & $ $\\ \none & $ $ \end{ytableau}};
  \node (4) at (2,4) {\begin{ytableau}$ $ & $ $ & $ $ & $ $ \end{ytableau}};
  \node (32) at (-1,5) {\begin{ytableau}$ $ & $ $ & $ $\\ \none & $ $ & $ $\end{ytableau}};
  \node (41) at (1,5) {\begin{ytableau}$ $ & $ $ & $ $ & $ $  \\ \none & $ $ \end{ytableau}};
  \node (5) at (3,5) {\begin{ytableau}$ $ & $ $ & $ $ & $ $ & $ $ \end{ytableau}};
  \draw (empty)--(1)--(2)--(3)--(4)--(5)
(2)--(21)--(31)--(32)
(3)--(31)--(41)
(4)--(41);
\end{tikzpicture}
\end{center}
\end{minipage}
\begin{minipage}{.5\textwidth}
\begin{center}
\begin{tikzpicture}
       \node (empty) at (0,0) {$\varnothing$};
  \node (1) at (0,1) {\ytableausetup{boxsize=.15cm}\begin{ytableau}$ $ \end{ytableau}};
  \node (2) at (0,2) {\begin{ytableau}$ $ & $ $  \end{ytableau}};
  \node (3) at (1,3) {\begin{ytableau}$ $ & $ $ & $ $ \end{ytableau}};
  \node (21) at (-1,3) {\begin{ytableau}$ $ & $ $ \\ \none & $ $ \end{ytableau}};
  \node (31) at (0,4) {\begin{ytableau}$ $ & $ $ & $ $\\ \none & $ $ \end{ytableau}};
  \node (4) at (2,4) {\begin{ytableau}$ $ & $ $ & $ $ & $ $ \end{ytableau}};
  \node (32) at (-1,5) {\begin{ytableau}$ $ & $ $ & $ $\\ \none & $ $ & $ $\end{ytableau}};
  \node (41) at (1,5) {\begin{ytableau}$ $ & $ $ & $ $ & $ $  \\ \none & $ $ \end{ytableau}};
  \node (5) at (3,5) {\begin{ytableau}$ $ & $ $ & $ $ & $ $ & $ $ \end{ytableau}};
  \draw (empty) -- (1) to [bend right =5] (2)
(2) to [bend left=5] (3)
(3) to [bend left =5] (2)
(2) -- (21) to [bend left =5] (31)
(21) to [bend right =5] (31)
(1) to [bend left=5] (2)
(3) to [bend left=5] (4)
 (3) to [bend right=5] (4)
(3) -- (31) to [bend left=5] (41)
(31) to [bend right =5] (41)
(31) to [bend left =5] (32)
(31) to [bend right=5] (32)
(4) to [bend right=5] (5)
(4) to [bend left=5] (5)
(2) to [bend left=15] (empty)
(3) to [bend left =15] (1)
(3) to [bend left =25] (1)
(3) to [bend left =15] (empty)
(4) to [bend left =15] (2)
(4) to [bend left =25] (2)
(4) to [bend left =15] (1)
(4) to [bend left =25] (1)
(4) to [bend left = 15] (empty)
(5) to [bend left =15] (3)
(5) to [bend left =25] (3)
(5) to [bend left =15] (2)
(5) to [bend left =25] (2)
(5) to [bend left =15] (1)
(5) to [bend left =25] (1)
(5) to [bend left=15] (empty)
(31) to [bend left=5] (2)
(31) to [bend right =5] (2)
(31) to [bend right = 15] (1)
(21) to [bend right = 15] (1)
(32) to [bend right =20] (3)
(41) to [bend left = 5] (3)
(41) to [bend right =5] (3)
(41) to [bend right = 15] (21)
(41) to [bend right =20] (21)
(41) to [bend left = 5] (2)
(41) to [bend right = 5] (2)
(32) -- (2)
(4) -- (41);
\end{tikzpicture}
\end{center}
\end{minipage}
\ytableausetup{boxsize=.5cm}

Notice that there are two edges from $41$ to $2$ corresponding to the following fillings of the border strip $(4,1)/(2)$.
\begin{center}
 \begin{ytableau}
  \none & \none & $k$' & $k$ \\
  \none & $k$
 \end{ytableau}\hspace{1in}
 \begin{ytableau}
    \none & \none & $k$ & $k$ \\
  \none & $k$
 \end{ytableau}
\end{center}

\begin{theorem} \label{thm:syoung}
 The Pieri deformation of the dual graded graph of shifted shapes satisfies $$DU-UD=D+I.$$
\end{theorem}

We shall need the following two properties of $\Lambda'$.

\begin{theorem} \cite{Mac}
\begin{itemize}
 \item $\Lambda'$ is a free polynomial ring in odd power sums $p_1, p_3, \ldots$.
 \item $\Lambda'$ inherits the standard bilinear inner product from $\Lambda$, which satisfies $\langle Q_{\lambda}, P_{\mu} \rangle = 2^{l(\mu)} \delta_{\lambda, \mu}$.
\end{itemize}
\end{theorem}

Because of the first property, we can again differentiate elements of $\Lambda'$ with respect to $p_1$, by expressing them first as a polynomial in the $p_i$'s. 
We shall need the following property of $f$.

\begin{lemma} \label{lem:syoung}
 We have $$\frac{d}{d p_1} f = 2f+2.$$
\end{lemma}

\begin{proof}
From  \cite[Ex. 6 (a), III, 8]{Mac} we have 
$$1 + f = e^{2p_1 + \frac{2p_3}{3} + \ldots},$$ which of course implies $\frac{d}{d p_1} (1+f) = 2+2f$.
\end{proof}

Now we are ready for the  proof of Theorem \ref{thm:syoung}.

\begin{proof}
 Applying Lemma \ref{lem:young} we see that $A = \Lambda'$, $a_{\lambda} = Q_{\lambda}$, $D = \frac{1}{2} \frac{d}{d p_1}$ and $f = q_1 + q_2 + \ldots$ satisfy the conditions of 
Theorem \ref{thm:pieri}. The claim follows.
\end{proof}

\begin{remark}
 If we instead take $a_1(\mu, \nu)$ to be 2 times the coefficient of $Q_{\nu} \text{ in } p_1 Q_{\mu}$, and $a_2(\mu, \nu)$ to be 
 $\frac{1}{2}$ times the coefficient of $P_{\mu} \text{ in } f P_{\nu},$ we get a dual filtered graph with $$DU-UD=2D+I.$$ This choice corresponds to first swapping
 $E_1$ and $E_2$ in the original dual graded graph of $\mathbb{SY}$ and then adding downward edges to $E_2$ as described above.
 
\begin{minipage}{.5\textwidth}
\begin{center}
\begin{tikzpicture}
       \node (empty) at (0,0) {$\varnothing$};
  \node (1) at (0,1) {\ytableausetup{boxsize=.15cm}\begin{ytableau}$ $ \end{ytableau}};
  \node (2) at (0,2) {\begin{ytableau}$ $ & $ $  \end{ytableau}};
  \node (3) at (1,3) {\begin{ytableau}$ $ & $ $ & $ $ \end{ytableau}};
  \node (21) at (-1,3) {\begin{ytableau}$ $ & $ $ \\ \none & $ $ \end{ytableau}};
  \node (31) at (0,4) {\begin{ytableau}$ $ & $ $ & $ $\\ \none & $ $ \end{ytableau}};
  \node (4) at (2,4) {\begin{ytableau}$ $ & $ $ & $ $ & $ $ \end{ytableau}};
  \node (32) at (-1,5) {\begin{ytableau}$ $ & $ $ & $ $\\ \none & $ $ & $ $\end{ytableau}};
  \node (41) at (1,5) {\begin{ytableau}$ $ & $ $ & $ $ & $ $  \\ \none & $ $ \end{ytableau}};
  \node (5) at (3,5) {\begin{ytableau}$ $ & $ $ & $ $ & $ $ & $ $ \end{ytableau}};
  \draw (empty)--(1) to [bend right =5] (2)
(2) to [bend left=5] (3)
(3) to [bend left =5] (2)
(2) -- (21) to [bend left =5] (31)
(21) to [bend right =5] (31)
(1) to [bend left=5] (2)
(3) to [bend left=5] (4)
 (3) to [bend right=5] (4)
(3) -- (31) to [bend left=5] (41)
(31) to [bend right =5] (41)
(31) to [bend left =5] (32)
(31) to [bend right=5] (32)
(4) to [bend right=5] (5)
(4) -- (41)
(4) to [bend left=5] (5);
\end{tikzpicture}
\end{center}
\end{minipage}
\begin{minipage}{.5\textwidth}
\begin{center}
\begin{tikzpicture}
       \node (empty) at (0,0) {$\varnothing$};
  \node (1) at (0,1) {\ytableausetup{boxsize=.15cm}\begin{ytableau}$ $ \end{ytableau}};
  \node (2) at (0,2) {\begin{ytableau}$ $ & $ $  \end{ytableau}};
  \node (3) at (1,3) {\begin{ytableau}$ $ & $ $ & $ $ \end{ytableau}};
  \node (21) at (-1,3) {\begin{ytableau}$ $ & $ $ \\ \none & $ $ \end{ytableau}};
  \node (31) at (0,4) {\begin{ytableau}$ $ & $ $ & $ $\\ \none & $ $ \end{ytableau}};
  \node (4) at (2,4) {\begin{ytableau}$ $ & $ $ & $ $ & $ $ \end{ytableau}};
  \node (32) at (-1,5) {\begin{ytableau}$ $ & $ $ & $ $\\ \none & $ $ & $ $\end{ytableau}};
  \node (41) at (1,5) {\begin{ytableau}$ $ & $ $ & $ $ & $ $  \\ \none & $ $ \end{ytableau}};
  \node (5) at (3,5) {\begin{ytableau}$ $ & $ $ & $ $ & $ $ & $ $ \end{ytableau}};
  \draw (empty)--(1)--(2)--(3)--(4)--(5)
(2)--(21)--(31)--(32)
(3)--(31)--(41)
(4)--(41)
(2) to [bend left=15] (empty)
(3) to [bend left =15] (1)
(3) to [bend left =25] (1)
(3) to [bend left =15] (empty)
(4) to [bend left =15] (2)
(4) to [bend left =25] (2)
(4) to [bend left =15] (1)
(4) to [bend left =25] (1)
(4) to [bend left = 15] (empty)
(5) to [bend left =15] (3)
(5) to [bend left =25] (3)
(5) to [bend left =15] (2)
(5) to [bend left =25] (2)
(5) to [bend left =15] (1)
(5) to [bend left =25] (1)
(5) to [bend left=15] (empty)
(31) to [bend left=5] (2)
(31) to [bend right =5] (2)
(31) to [bend right = 15] (1)
(21) to [bend right = 15] (1)
(32) to [bend right =20] (3)
(41) to [bend left = 5] (3)
(41) to [bend right =5] (3)
(41) to [bend right = 15] (21)
(41) to [bend right =20] (21)
(41) to [bend left = 5] (2)
(41) to [bend right = 5] (2)
(32) -- (2)
(4) -- (41);
\end{tikzpicture}
\end{center}
\end{minipage}
\end{remark}

\subsection{M\"obius deformation of shifted Young's lattice}
The Mobius deformation of the shifted Young's lattice is shown below with upward-oriented edges shown on the left and downward-oriented edges on the right. 
Note that we swapped the $E_1$ and $E_2$ of Example \ref{ex:shifted}.

\begin{minipage}{.5\textwidth}
\begin{center}
\begin{tikzpicture}
       \node (empty) at (0,0) {$\varnothing$};
  \node (1) at (0,1) {\ytableausetup{boxsize=.15cm}\begin{ytableau}$ $ \end{ytableau}};
  \node (2) at (0,2) {\begin{ytableau}$ $ & $ $  \end{ytableau}};
  \node (3) at (1,3) {\begin{ytableau}$ $ & $ $ & $ $ \end{ytableau}};
  \node (21) at (-1,3) {\begin{ytableau}$ $ & $ $ \\ \none & $ $ \end{ytableau}};
  \node (31) at (0,4) {\begin{ytableau}$ $ & $ $ & $ $\\ \none & $ $ \end{ytableau}};
  \node (4) at (2,4) {\begin{ytableau}$ $ & $ $ & $ $ & $ $ \end{ytableau}};
  \node (32) at (-1,5) {\begin{ytableau}$ $ & $ $ & $ $\\ \none & $ $ & $ $\end{ytableau}};
  \node (41) at (1,5) {\begin{ytableau}$ $ & $ $ & $ $ & $ $  \\ \none & $ $ \end{ytableau}};
  \node (5) at (3,5) {\begin{ytableau}$ $ & $ $ & $ $ & $ $ & $ $ \end{ytableau}};
  \draw (empty) -- (1) to [bend right =5] (2)
(2) to [bend left=5] (3)
(3) to [bend left =5] (2)
(2) -- (21) to [bend left =5] (31)
(21) to [bend right =5] (31)
(1) to [bend left=5] (2)
(3) to [bend left=5] (4)
 (3) to [bend right=5] (4)
(3) -- (31) to [bend left=5] (41)
(31) to [bend right =5] (41)
(31) to [bend left =5] (32)
(31) to [bend right=5] (32)
(4) to [bend right=5] (5)
(4) to [bend left=5] (5)
(4) -- (41)
(1) to [out=160,in=100,distance=.5cm] (1)
(2) to [out=60, in=100, distance=.5cm] (2)
(2) to [out=60, in=100, distance=.7cm] (2)
(21) to [out=60, in=100, distance=.5cm] (21)
(3) to [out=60, in=100, distance=.5cm] (3)
(3) to [out=60, in=100, distance=.7cm] (3)
(4) to [out=60, in=100, distance=.5cm] (4)
(4) to [out=60, in=100, distance=.7cm] (4)
(5) to [out=60, in=100, distance=.5cm] (5)
(5) to [out=60, in=100, distance=.7cm] (5)
(31) to [out=60, in=100, distance=.5cm] (31)
(31) to [out=60, in=100, distance=.7cm] (31)
(31) to [out=60, in=100, distance=.9cm] (31)
(41) to [out=60, in=100, distance=.5cm] (41)
(41) to [out=60, in=100, distance=.7cm] (41)
(41) to [out=60, in=100, distance=.9cm] (41)
(32) to [out=60, in=100, distance=.5cm] (32)
(32) to [out=60, in=100, distance=.7cm] (32);
\end{tikzpicture}
\end{center}
\end{minipage}
\begin{minipage}{.5\textwidth}
\begin{center}
\begin{tikzpicture}
       \node (empty) at (0,0) {$\varnothing$};
  \node (1) at (0,1) {\ytableausetup{boxsize=.15cm}\begin{ytableau}$ $ \end{ytableau}};
  \node (2) at (0,2) {\begin{ytableau}$ $ & $ $  \end{ytableau}};
  \node (3) at (1,3) {\begin{ytableau}$ $ & $ $ & $ $ \end{ytableau}};
  \node (21) at (-1,3) {\begin{ytableau}$ $ & $ $ \\ \none & $ $ \end{ytableau}};
  \node (31) at (0,4) {\begin{ytableau}$ $ & $ $ & $ $\\ \none & $ $ \end{ytableau}};
  \node (4) at (2,4) {\begin{ytableau}$ $ & $ $ & $ $ & $ $ \end{ytableau}};
  \node (32) at (-1,5) {\begin{ytableau}$ $ & $ $ & $ $\\ \none & $ $ & $ $\end{ytableau}};
  \node (41) at (1,5) {\begin{ytableau}$ $ & $ $ & $ $ & $ $  \\ \none & $ $ \end{ytableau}};
  \node (5) at (3,5) {\begin{ytableau}$ $ & $ $ & $ $ & $ $ & $ $ \end{ytableau}};
  \draw (empty)--(1)--(2)--(3)--(4)--(5)
(31)--(2)
(41)--(3)
(2)--(21)--(31)--(32)
(3)--(31)--(41)
(4)--(41);
\end{tikzpicture}
\end{center}
\end{minipage}
\ytableausetup{boxsize=.5cm}

\begin{lemma}
 For the lattice of shifted shapes, the Mobius function is given by the following formula:
$$\mu(p,q)=\left\{ \begin{array}{ll}
                    (-1)^{|p|-|q|} &  \text{if }q/p\text{ is a disjoint union of boxes} \\
		    0 &  \text{otherwise}
                   \end{array}
\right. 
$$
\end{lemma}
\begin{proof}
 It is known that $\mathbb{SY}$ is a distributive lattice \cite[Example 2.2.8]{F}, and so it 
follows that $\mu(p,q)$ is nonzero if and only if the interval $[p,q]$ is a Boolean lattice, 
in which case $\mu(p,q)$ is as described above, \cite[Example 3.9.6]{EC2}. And $[p,q]$ is a Boolean 
lattice exactly when $q/p$ is a disjoint union of boxes, i.e. no two boxes of $q/p$ share an edge.
\end{proof}

\begin{theorem}
 The M\"obius deformation of the shifted Young's lattice forms a dual filtered graph with $$DU-UD=D+I.$$ 
\end{theorem}

\begin{proof}

Given that the non-deformed pair form a pair of dual graded graphs, 
it suffices to show that $$[\lambda](DU-UD)(\mu)=[\lambda](D)(\mu)$$ when $\mu$ covers $\lambda$.

For a shifted shape $x$, let $I_{diag}(x)$ denote the size of the set of inner corners of $x$ that are on the diagonal, 
$I_{odiag}(x)$ denote the size of the set of inner corners of $x$ that are not on the diagonal, 
$O_{diag}(x)$ denote the size of the set of outer corners of $x$ that are 
on the diagonal, and $O_{odiag}(x)$ denote the size of the set of outer corners of $x$ that are not on the diagonal.

\begin{lemma}\label{lem:shiftededges}
 For any shifted shape $\lambda$, $$O_{diag}(\lambda) + 2O_{odiag}(\lambda)-I_{diag}(\lambda)-2I_{odiag}(\lambda)=1.$$
\end{lemma}
\begin{proof}
 Consider some element $\lambda=(\lambda_1,\lambda_2,\ldots,\lambda_k)$.
Suppose $\lambda_k=1$. We can add a box at some subset of rows $\{i_1=1,i_2,\ldots,i_t\}$ and fill each with either $k$ or $k'$. 
Thus $O_{diag}(\lambda) + 2O_{odiag}(\lambda)=2t$. We can delete a box in rows $i_2-1,i_3-1,\ldots,i_t-1,k\}$, where each can be filled with $k$ or 
$k'$ except the box in row $k$, which is on the diagonal. Thus $I_{diag}(\lambda)-2I_{odiag}(\lambda)=2(t-1)+1$.

Now suppose $\lambda_k\geq 2$. We can add a box at some subset of rows $\{i_1=1,i_2,\ldots,i_t\,k+1\}$ and fill each with either $k$ or $k'$ except for the box in row $k+1$, which is on 
the diagonal. This gives $O_{diag}(\lambda) + 2O_{odiag}(\lambda)=2t+1$. We can delete boxes in rows $\{i_2-1,\ldots,i_t-1\,k\}$, where each can be filled with $k$ or $k'$, 
giving $I_{diag}(\lambda)-2I_{odiag}(\lambda)=2t$.
\end{proof}

First, consider all up-down paths from $\mu$ to $\lambda$ that begin with a loop at $\mu$. There are 
$$2I_{odiag}(\mu)+I_{diag}(\mu)$$ such paths since there are the same number of loops at $\mu$. Up-down 
paths that do not start with a loop can be counted by the number of outer corners of $\lambda$ that are not 
inner corners of $q$ counted with multiplicity two if they are off the main diagonal 
since each of these corners may be added to $q$ for the step up and then removed 
in addition to the boxes of $\mu/\lambda$ for the step down. There are 
$$2(O_{odiag}(\lambda)-I_{odiag}(\mu/\lambda))+(O_{diag}(\lambda)-I_{diag}(\mu/\lambda))$$ such corners. Thus there are 
$$2I_{odiag}(\mu)+I_{diag}(\mu)+2(O_{odiag}(\lambda)-I_{odiag}(\mu/\lambda))+(O_{diag}(\lambda)-I_{diag}(\mu/\lambda))$$ up-down paths from $\mu$ to $\lambda$.

Next, consider all down-up paths from $\mu$ to $\lambda$ that end with a loop at $\lambda$. There are exactly 
$$2I_{odiag}(\lambda)+I_{diag}(\lambda)$$ such paths. Down-up paths that do not end with a loop can be counted by 
the inner corners of $q$ that are also inner corners of $\lambda$ counted twice if they are off the main diagonal 
since removing this type of inner corner in addition
to the boxes of $\lambda/\mu$ will result in a shifted shape covered by $\lambda$. There are 
$$2(I_{odiag}(\mu)-I_{odiag}(\mu/\lambda))+(I_{diag}(\mu)-I_{diag}(\mu/\lambda))$$ such corners. Thus there are 
$$2I_{odiag}(\lambda)+I_{diag}(\lambda)+2(I_{odiag}(\mu)-I_{odiag}(\mu/\lambda))+(I_{diag}(\mu)-I_{diag}(\mu/\lambda))$$ down-up paths from $\mu$ to $\lambda$.

Hence 
\begin{eqnarray}
[\lambda](DU-UD)(\mu) &=& 2I_{odiag}(\mu)+I_{diag}(\mu)+2(O_{odiag}(\lambda)-I_{odiag}(\mu/\lambda))+(O_{diag}(\lambda)-I_{diag}(\mu/\lambda)) \nonumber \\
&-&(2I_{odiag}(\lambda)+ I_{diag}(\lambda)+2(I_{odiag}(\mu)-I_{odiag}(\mu/\lambda))+(I_{diag}(\mu)-I_{diag}(\mu/\lambda)))) \nonumber \\
&=& 2O_{0diag}(\lambda) + O_{diag}(\lambda)-2I_{odiag}(\lambda)-I_{diag}(\lambda) \nonumber \\
&=& 1 \nonumber
\end{eqnarray}
by Lemma \ref{lem:shiftededges}. 
\end{proof}

\subsection{Shifted Hecke insertion}\label{sec:shiftedinsertion}
\begin{definition}\label{def:increasingshifted}
An \textit{increasing shifted tableau} is a filling of a shifted shape with non-negative integers such that entries are strictly increasing across rows and columns.
\end{definition}

The following are examples of increasing shifted tableaux.
\begin{center}\begin{ytableau}1 & 2 & 4 \\ \none & 4\end{ytableau}\hspace{1in}\begin{ytableau} 2 & 3 & 5 & 7 & 8 \\ \none & 5 & 8 \\ \none & \none & 9\end{ytableau}\end{center}

We now describe an insertion procedure for increasing shifted shapes that is similar to Hecke insertion, which we call \textit{shifted Hecke insertion}. This procedure is a natural analogue of that of Sagan and Worley, \cite{Sagan, W}. As before, any insertion tableau will correspond to a walk 
downward ending at $\varnothing$ in the dual filtered graph of shifted shapes from the M\"obius deformation and any recording tableau corresponds to a walk upward starting at $\varnothing$.

We start by describing how to insert $x$ into an increasing shifted tableau $Y$ to obtain increasing shifted tableau $Z$. We begin by 
inserting $x=y_1$ into the first row of $Y$. This insertion may modify the first row of $Y$ and may produce an output integer $y_2$. 
Following the rules below, we then insert $y_2$ into the second row of $Y$ and continue in this manner until one of two things happens. Either at 
some stage the insertion will not create an output integer, in which case the insertion of $x$ into $Y$ terminates, or $y_k$ will replace some 
diagonal element of $Y$. In the latter case, we continue the insertion process along the columns: $y_{k+1}$ is inserted into the column to the right of 
its original position and so on until some step in this process does not create an output integer. 

For each insertion, we designate a specific box of the resulting tableau, $Z$, as being the box where the insertion terminated. We will later use 
this notion to define recording tableaux. 

The rules for inserting any positive integer $x$ 
into a row or column are as follows:

\begin{itemize}
\item[(S1)] If $x$ is weakly larger than all integers in the row (resp. column) and adjoining $x$ to the end of the row (resp. column)
results in an increasing tableau, then $Z$ is the resulting tableau. We say that the insertion terminated at this new box.
\end{itemize}

\begin{example}
 Inserting $4$ into the first row of the tableau on the left gives the tableau on the right. This insertion 
terminated in position $(1,3)$.

\begin{center}
\begin{ytableau}
 1 & 2 \\
\none & 4
\end{ytableau}\hspace{1in}
\begin{ytableau}
 1 & 2 & 4\\
\none & 4
\end{ytableau}
\end{center}

Inserting $5$ into the third column of the resulting tableau gives the tableau shown below. This insertion 
terminated in position $(2,3)$. 

\begin{center}
\begin{ytableau}
  1 & 2 & 4\\
\none & 4 & 5
\end{ytableau}
\end{center}
\end{example}

\begin{itemize}
\item[(S2)] If $x$ is weakly larger than all integers in the row (resp. column) and adjoining $x$ to the end of the row (resp. column)
does not result in an increasing tableau, then $Z=Y$. If $x$ was row inserted into a nonempty row, we say that the insertion terminated at the box at the bottom of the column 
containing the rightmost box of this row. If $x$ was row inserted into an empty row, we say the insertion terminated at the 
rightmost box of the previous row. If $x$ was column inserted, we say the insertion terminated at the rightmost box of the row containing the bottom 
box of the column $x$ could not be added to.
\end{itemize}

\begin{example} Inserting $4$ into the first row of the tableau on the left does not change the row and does not
give an output integer, and so the insertion does not change the tableau. The insertion terminated in position $(2,3)$
 \begin{center}
  \begin{ytableau}
   1 & 2 & 4 \\
\none & 3 & 5
  \end{ytableau}
 \end{center}
Inserting $4$ into the third column of the tableau below does not change the column and does not produce an 
output integer, and so the insertion does not change the tableau. The insertion terminated in position $(1,4)$.
\begin{center}
 \begin{ytableau}
  1 & 2 & 4 & 5 \\
\none & 3
 \end{ytableau}
\end{center}

Inserting 2 into the (empty) second row of the tableau below does not change the row. The insertion terminated in position (1,2)
\begin{center}
 \begin{ytableau}
  1 & 2
 \end{ytableau}

\end{center}

\end{example}

For the last two rules, suppose the row (resp. column) contains a box with label strictly larger than $x$, and let $y$ be the smallest such box.

\begin{itemize}
\item[(S3)] If replacing $y$ with $x$ results in an increasing tableau, then replace $y$ with $x$. 
In this case, $y$ is the output integer. If $x$ was inserted into a column 
or if $y$ was on the main diagonal, proceed to insert all future 
output integers into the next column to the right. If $x$ was inserted into a row and $y$ was not on the main diagonal, then insert $y$ into the row 
below. 
\end{itemize}

\begin{example} Inserting $6$ into the second row of the tableau on the left results in the tableau 
on the right with output integer $7$ to be inserted into the third row.
 \begin{center}
  \begin{ytableau}
   1 & 2 & 3 & 4 \\
\none & 4 & 7 & 8 \\
\none & \none & 8
  \end{ytableau}\hspace{1in}
  \begin{ytableau}
   1 & 2 & 3 & 4 \\
\none & 4 & 6 & 8 \\
\none & \none & 8
  \end{ytableau}
 \end{center}
Inserting $6$ into the third column of the tableau on the left also results in the tableau on the right with 
output integer $7$, but this time, $7$ is to be inserted into the fourth column.

Inserting $3$ into the second row of the tableau on the left results in the tableau shown below with output 
integer $4$ to be inserted into the third column.
\begin{center}
 \begin{ytableau}
   1 & 2 & 3 & 4 \\
\none & 3 & 7 & 8 \\
\none & \none & 8
  \end{ytableau}
 \end{center}
\end{example}

\begin{itemize}
\item[(S4)] If replacing $y$ with $x$ does not result in an increasing tableau, then do not change the 
row (resp. column). 
In this case, $y$ is the output integer. If $x$ was inserted into a column 
or if $y$ was on the main diagonal, proceed to insert all future 
output integers into the next column to the right. If $x$ was inserted into a row, then insert $y$ into the row 
below. 
\end{itemize}

\begin{example} Inserting $2$ into the first row of the tableau on the left does not change the tableau and produces 
output integer $5$ to be inserted into the second row.
 \begin{center}
  \begin{ytableau}
   1 & 2 & 5 & 8\\
\none & 3 & 6
  \end{ytableau}
 \end{center}
Inserting $5$ into the third column does not change the tableau and gives output integer $6$ to be inserted 
into the fourth column.

Inserting $2$ into the second row does not change the tableau. In this case, the output integer is 3 and is to 
be inserted into the third column.
\end{example}

Using this insertion algorithm, we define the \textit{shifted Hecke insertion tableau} of a word $w=w_1w_2\cdots w_n$ to be 
$$P_S(w)=(\ldots((\varnothing\leftarrow w_1)\leftarrow w_2)\ldots) \leftarrow w_n.$$

\begin{example}
 We show the sequence of shifted tableaux obtained while computing $P_S(4211232)$. We start by inserting $4$ into the first row.
\begin{center}
\begin{ytableau}
 4
\end{ytableau}\hspace{.1in}
\begin{ytableau}
 2 & 4
\end{ytableau}\hspace{.1in}
\begin{ytableau}
 1 & 2 & 4
\end{ytableau}\hspace{.1in}
\begin{ytableau}
 1 & 2 & 4
\end{ytableau}\hspace{.1in}
\begin{ytableau}
 1 & 2 & 4 \\
\none & 4
\end{ytableau}\hspace{.1in}
\begin{ytableau}
 1 & 2 & 3 \\
\none & 4 
\end{ytableau}\hspace{.1in}
\begin{ytableau}
 1 & 2 & 3 \\
\none & 3 & 4
\end{ytableau}\hspace{.1in}
\end{center}
\end{example}

\begin{remark}
 The result of such insertion agrees with that of $K$-theoretic jeu de taquin rectification described in \cite{ThYshift}. See \cite{REU2015}.
\end{remark}

In this setting, recording tableaux will be set-valued shifted tableaux.

\begin{definition}\label{def:shiftedset}
A \textit{set-valued shifted tableau} $T$ of shifted shape $\lambda$ is a filling of the boxes with finite, nonempty subsets of primed and unprimed positive integers so that 
\begin{enumerate}
 \item the smallest number in each box is greater than or equal to the largest number in the box directly to the left of it (if that box is present),
  \item the smallest number in each box is greater than or equal to the largest number in the box directly to the above it (if that box is present),
 \item any positive integer appears at most once, either primed or unprimed, but not both, and
 \item there are no primed entries on the main diagonal.
\end{enumerate}
\end{definition}

A set-valued shifted tableau is called \textit{standard} if the set of labels consists of $1,2,\ldots,n$, each appearing either primed or unprimed exactly once, for some $n$.

A recording tableau for a word $w=w_1w_2\ldots w_n$ is a standard shifted set-valued tableau and is obtained as follows. Begin with $Q_S(\varnothing)=\varnothing$. If the insertion of $w_k$ into $P_S(w_1\cdots w_{k-1})$ resulted in adding a new 
box to $P_S(w_1\cdots w_{k-1})$, add this same box with label $k$ if the box was added via row insertion and $k'$ if the box was added via column insertion to $Q_S(w_1\cdots w_{k-1})$ to obtain $Q_S(w_1\cdots w_k)$. 
If the insertion of $k$ into $P_S(w_1\cdots w_{k-1})$ did not change the shape 
of $P_S(w_1\cdots w_{k-1})$, obtain $Q_S(w_1\cdots w_k)$ from $Q_S(w_1\cdots w_{k-1})$ by adding the label $k$ to the box where the insertion terminated if the last move was a row insertion into a nonempty row and $k'$ if 
the last move was a column insertion. If the last move was row insertion into an empty row, label the box where the insertion terminated $k'$.

\begin{example} The top row of tableaux shows the sequence of tableaux obtained from inserting $w=4211232$ as in the previous example, and the bottom row shows the corresponding steps to form $Q_S(w)$.
 \begin{center}
 \ytableausetup{boxsize=.6cm}
 \begin{ytableau}
 4
\end{ytableau}\hspace{.1in}
\begin{ytableau}
 2 & 4
\end{ytableau}\hspace{.1in}
\begin{ytableau}
 1 & 2 & 4
\end{ytableau}\hspace{.1in}
\begin{ytableau}
 1 & 2 & 4
\end{ytableau}\hspace{.1in}
\begin{ytableau}
 1 & 2 & 4 \\
\none & 4
\end{ytableau}\hspace{.1in}
\begin{ytableau}
 1 & 2 & 3 \\
\none & 4 
\end{ytableau}\hspace{.1in}
\begin{ytableau}
 1 & 2 & 3 \\
\none & 3 & 4
\end{ytableau}\hspace{.1in}
$=P_S(w)$
\end{center}

\begin{center}
\begin{ytableau}
 1
\end{ytableau}\hspace{.1in}
\begin{ytableau}
 1 & 2'
\end{ytableau}\hspace{.1in}
\begin{ytableau}
 1 & 2' & 3'
\end{ytableau}\hspace{.1in}
\begin{ytableau}
 1 & 2' & 3'4'
\end{ytableau}\hspace{.1in}
\begin{ytableau}
 1 & 2' & 3'4' \\
\none & 5
\end{ytableau}\hspace{.1in}
\begin{ytableau}
 1 & 2' & 3'4' \\
\none & 56
\end{ytableau}\hspace{.1in}
\begin{ytableau}
 1 & 2' & 3'4' \\
\none & 56 & 7'
\end{ytableau}\hspace{.1in}
$=Q_S(w)$
\end{center}
\end{example}

We next define a reverse insertion procedure so that given a pair $(P_S(w),Q_S(w))$, we can recover $w$.

First locate the box containing the largest label of $Q_S(w)$, call the label $n$ and the position of the box $(i_n,j_n)$, and find the corresponding box in $P_S(w)$. Say 
the integer in positioin $(i_n,j_n)$ of $P_S(w)$ is $y_n$. We then perform reverse insertion on cell $(i_n,j_n)$ of $P_S(w)$ by following the rules below.
\begin{itemize}

\item[(rS1)] If $n$ is the only label in cell $(i_n,j_n)$ of $Q(w)$, remove box $(i_n,j_n)$ from $P_S(w)$ and reverse insert $y_n$ into the row above if $n$ is unprimed and into the column 
to the left if $n$ is primed.

\item[(rS2)] If $n$ is not the only label in cell $(i_n,j_n)$ of $Q_S(w)$, do not remove box $(i_n,j_n)$ from $P_S(w)$, but still reverse insert $y_n$ into the row above if $n$ is unprimed and into the column 
to the left if $n$ is primed.
\end{itemize}

In the row above if $y_n$ is reverse inserted into a row or the column to the left if it is reverse inserted into a column, let $x$ be the largest label with $x<y_n$.

\begin{itemize}
\item[(rS3)] If replacing $x$ with $y_n$ results in an 
increasing shifted tableau, replace $x$ with $y_n$. If $y_n$ was reverse column inserted and $x$ was not on the main diagonal, reverse insert $x$ into the column to the left. 
Otherwise, reverse insert $x$ into the row above.

\item[(rS4)] If replacing $x$ with $y_n$ does not result in an increasing shifted tableau, leave the row 
or column unchanged. If $y_n$ was reverse column inserted and $x$ was not on the main diagonal, reverse insert $x$ into the column to the left. 
Otherwise, reverse insert $x$ into the row above.
\end{itemize}

If we are in the first row of the tableau and the last step was a reverse row insertion or we are in the first column and the last step was a reverse column insertion, 
then $x$ and the modified tableau, which we will call $P_{S,1}(w)$, are the final output value. Define $x_1=x$.

Repeat this process using the pair $(P_{S,1}(w),Q_{S,n-1}(w))$, where $Q_{S,n-1}(w)$ is the result of removing the entry $n$ or $n'$ from $Q_S(w)$. 
Define $x_2$ to be the output value and $P_{S,2}(w)$ to be the modified tableau. 
Continue this process with pairs $(P_{S,2}(w),Q_{S,n-2}(w)),\ldots,(P_{S,n-1}(w),Q_{S,1}(w))$. Then we claim $w=x_1x_2\cdots x_n$.

\begin{theorem}
There is a bijection between pairs consisting of an increasing shifted tableau and a standard set-valued shifted tableau of the same shape, $(P,Q)$, 
and words, where the word $w$ corresponds to the pair $(P_S(w),Q_S(w))$.
\end{theorem}
\begin{proof}
We show that given a pair of tableaux of the same shape, $(P,Q)$, where $P$ is an increasing shifted tableau on some $[n]$ and $Q$ is standard shifted set-valued tableau on
$1<2'<2<3'<\ldots<n'<n$,
we can recover $w$ so that $P=P_S(w)$ and 
$Q=Q_S(w)$. Assume first that $(P,Q)$ has been obtained by inserting some integer $h$ into a pair 
$(Y,Z)$. We must show that the reverse insertion procedure defined above recovers $h$ and $(Y,Z)$ from $(P,Q)$.

It is clear that if the insertion of $h$ ended by adding a new box to $Y$ and $Z$ using (S1), then 
reverse insertion steps (rS1), (rS2), (rS3), and (rS4) undo insertion steps (S1), (S2), (S3), and (S4), respectively.
Therefore, it suffices to show that if the insertion of $h$ does not result in adding a new box to $Y$ and $Z$ 
using (S2), then the reverse insertion procedure can recover $h$ from $(Y,Z)$. 

Suppose the last step of the insertion of $h$ into $Y$ is inserting some $x$ into nonempty row $i$ using (S2)
and the insertion terminates in row ${i+j}$. We first use $(rS2)$ and reverse insert the entry 
at the end of row ${i+j}$ into row ${i+j-1}$, which will not change row ${i+j-1}$ because the entry being 
reverse inserted is strictly larger than all entries in row $i+j-1$ but cannot 
replace the entry at the end of row $i+j-1$ since this would put it directly above the occurence of the 
entry in row $i+j$. Next, we reverse insert the entry at the end of row 
${i+j-1}$ into  ${i+j-2}$, and in the same manner, it will not change row ${i+j-2}$. This process continues 
until we reverse insert the entry at the end of row 
${i+1}$ into $i$ and obtain output integer $x$. From this point onward, the reverse insertion rules
will exactly undo the corresponding insertion rules to recover $h$. We can use the transpose of 
this argument if $x$ is column inserted into column $j$ using (S2) and the insertion terminates in column $j+i$.

For example, inserting 4 into the first row of the tableau below does not change the tableau, and the insertion terminates at $(3,4)$. Starting reverse insertion at $(3,4)$, we reverse insert 6 into the second row. 
It cannot replace the 5, so the third row doesn't change, and 5 is reverse inserted into the first row. The 5 cannot replace the 4, so the first row is unchanged, and we end with the original tableau and the integer 4.

\begin{center}
 \begin{ytableau}
  1 & 2 & 3 & 4 \\
\none & 3 & 4 & 5 \\
\none & \none & 5 & 6 
 \end{ytableau}
\end{center}

If $x$ is inserted into empty row $i$ using (S2). Then our original tableau looked like the one shown below, where the bottom row shown is row $i-1$ and the dot indicates
where the insertion terminated. 
In this case, $x$ must have been bumped by a row insertion of $a$ into row $i-1$. We show reverse insertion starting from the box where the insertion terminated recovers the step of inserting $a$ into row $i-1$.

\begin{center}
\begin{ytableau}
$ $ & $ $ & $ $ & $ $ &$ $ & $ $ & $ $ & $ $  \\
 \none & a & x & $ $ & $ $ & $ $ & \cdot 
\end{ytableau}
\end{center}

Using (rS4), we reverse insert the entry at the end of row $i-1$ into the column to the left. The column will not change since the entry being reverse inserted is strictly larger than all entries in the column 
and cannot replace the entry at the end of the column. We then reverse insert the second rightmost entry of row $i-1$ into the column to its left. Continuing in this manner, we eventually reverse column insert $x$ into the main 
diagonal, which results in reverse row inserting $a$ into row $i-2$. From this point on, the reverse insertion 
rules will exactly undo the corresponding insertion rules to recover $h$.
 
For example, suppose some insertion ends with inserting 2 into the empty second row of the tableau below. The insertion terminates in position $(1,4)$, and we see that the 2 must have been bumped by a row insertion of 1 into the first row.
We start reverse column insertion in position $(1,4)$ and reverse column insert 4 into the third column. The integer 4 cannot replace the 3 in the third column since it would be directly to the left of the 4 in the fourth column, 
so the third column is left unchanged. We reverse column insert 3 into the second column. Again, the column remains unchanged, and we reverse insert 2 into the first column. The 2 cannot replace the 1, so the column is unchanged. Since the 1 
was on the main diagonal, we next reverse row insert 1 into the row above.
\begin{center}
 \begin{ytableau}
  1 & 2 & 3 & 4 
 \end{ytableau}
\end{center}
\end{proof}

\subsection{Shifted Hecke growth and decay}\label{sec:shiftedgrowth}

As before, given any word $w=w_1w_2\cdots w_k$, containing
$n \leq k$ distinct numbers, we can create an $n \times k$ array with an $X$ in the $w_i^{th}$
square from the bottom of column $i$. Note that there can be multiple $X$'s in the same row but is at most
one $X$ per column. 

We will label the corners of each square with a shifted partition and label some of the horizontal edges of the squares with
a specific inner corner where the insertion terminated and/or a `c' to designate column insertion. To denote a specific inner corner, we will give either its row 
or column. For example, the edge label 2 denotes the inner corner at the end of the second row of the shifted diagram, and an edge labeled 2c denotes the 
inner corner at the end of the second column of the shifted diagram. We begin by labeling all corners
along the bottom row and left side of the diagram with the empty shape, $\varnothing$, as before.

To complete the labeling of the corners, suppose the corners $\mu$, $\lambda$, and $\nu$ are labeled, 
where $\mu$, $\lambda$, and $\nu$ are as in the picture below. We label $\gamma$ according to the following rules.

 \begin{center}
\begin{tikzpicture}[scale=1]
\node (A) at (-1,-1) {$\lambda$};
\node (B) at (1,-1) {$\nu$};
\node (C) at (-1,1) {$\mu$};
\node (D) at (1,1) {$\gamma$};
\draw (A) -- (B)
(C) -- (D)
(B) -- (D)
(A) -- (C);
\end{tikzpicture}
\end{center}

\textbf{If the square contains an X:}

\begin{itemize}
 \item[(1)] If $\lambda_1=\mu_1$, then $\gamma/\mu$ consists of one box in the first row. 
\item[(2)] If $\lambda_1+1=\mu_1$, then $\gamma=\mu$ and the edge between them is labeled with a $1$ to signify the inner corner in the first 
row of $\gamma$.
\end{itemize}

 \begin{center}
\begin{tikzpicture}[scale=1]
\node (A) at (-1,-1) {1};
\node (B) at (1,-1) {1};
\node (C) at (-1,1) {1};
\node (D) at (1,1) {2};
\node (E) at (0,0) {$X$};
\draw (A) -- (B)
(C) -- (D)
(B) -- (D)
(A) -- (C);
\end{tikzpicture}\hspace{1in}
\begin{tikzpicture}[scale=1]
\node (A) at (-1,-1) {1};
\node (B) at (1,-1) {1};
\node (C) at (-1,1) {2};
\node (D) at (1,1) {2};
\node (E) at (0,0) {$X$};
\node (F) at (0,1.25) {1};
\draw (A) -- (B)
(C) -- (D)
(B) -- (D)
(A) -- (C);
\end{tikzpicture}
\end{center}

\textbf{If the square does not contain an X and $\mu=\lambda$ or if $\nu=\lambda$ with no edge label on the bottom edge}

\begin{itemize}
 \item[(3)] If $\mu=\lambda$, then set $\gamma = \nu$ and label the top edge 
label as the bottom edge if one exists. If $\nu=\lambda$, then $\gamma = \mu$.  
\end{itemize}

 \begin{center}
\begin{tikzpicture}[scale=1]
\node (A) at (-1,-1) {1};
\node (B) at (1,-1) {2};
\node (C) at (-1,1) {1};
\node (D) at (1,1) {2};
\node (E) at (0,1.25) {c};
\node (F) at (0,-.75){c};
\draw (A) -- (B)
(C) -- (D)
(B) -- (D)
(A) -- (C);
\end{tikzpicture}\hspace{1in}
\begin{tikzpicture}[scale=1]
\node (A) at (-1,-1) {1};
\node (B) at (1,-1) {1};
\node (C) at (-1,1) {2};
\node (D) at (1,1) {2};
\draw (A) -- (B)
(C) -- (D)
(B) -- (D)
(A) -- (C);
\end{tikzpicture}
\end{center}

\textbf{If $\nu \nsubseteq \mu$ and the square does not contain an X:}

\begin{itemize}
 \item[(4)] In the case where $\nu \nsubseteq \mu$, $\gamma = \nu \cup \mu$, and the top edge label is the same as the bottom edge label.
\end{itemize}

\begin{center}
\begin{tikzpicture}[scale=1]
\node (A) at (-1,-1) {2};
\node (B) at (1,-1) {3};
\node (C) at (-1,1) {21};
\node (D) at (1,1) {31};
\draw (A) -- (B)
(C) -- (D)
(B) -- (D)
(A) -- (C);
\end{tikzpicture}
\end{center}

\textbf{If $\nu \subseteq \mu$ and the square does not contain an X:}

Suppose the square of $\nu/\lambda$ is one box in position $(i,j)$. If $(i,j)$ is not on the diagonal and the edge between $\lambda$ and $\nu$ is not labeled $c$:

\begin{itemize}
 \item[(5)] If there is no box of $\mu/\lambda$ in 
row $i+1$ of $\mu$, then $\gamma$ is obtained from $\mu$ by adding a box to the end of row $i+1$.

 \item[(6)] If there is a box of $\mu/\lambda$ in 
row $i+1$ of $\mu$, then $\gamma=\mu$ and the top edge is labeled $i+1$.
\end{itemize}

 \begin{center}
\begin{tikzpicture}[scale=1]
\node (A) at (-1,-1) {1};
\node (B) at (1,-1) {2};
\node (C) at (-1,1) {2};
\node (D) at (1,1) {21};
\draw (A) -- (B)
(C) -- (D)
(B) -- (D)
(A) -- (C);
\end{tikzpicture}\hspace{1in}
\begin{tikzpicture}[scale=1]
\node (A) at (-1,-1) {2};
\node (B) at (1,-1) {3};
\node (C) at (-1,1) {31};
\node (D) at (1,1) {31};
\node(E) at (0,1.25){2};
\draw (A) -- (B)
(C) -- (D)
(B) -- (D)
(A) -- (C);
\end{tikzpicture}
\end{center}

If $(i,j)$ is on the main diagonal or the bottom edge is labeled $c$:

\begin{itemize}
\item[(7)] If there is no box of $\mu/\lambda$ in column $j+1$ of $\mu$, then $\gamma$ is obtained 
from $\mu$ by adding a box to the end of column $j+1$. The top edge is labeled $c$.

\item[(8)] If there is a box of $\mu/\lambda$ in column $j+1$ of $\mu$, then $\gamma=\mu$. The top edge is labeled with $c$ and $j+1$.
\end{itemize}

 \begin{center}
\begin{tikzpicture}[scale=1]
\node (A) at (-1,-1) {1};
\node (B) at (1,-1) {2};
\node (C) at (-1,1) {2};
\node (D) at (1,1) {3};
\node (E) at (0,1.25){c};
\node (F) at (0,-.75){c};
\draw (A) -- (B)
(C) -- (D)
(B) -- (D)
(A) -- (C);
\end{tikzpicture}\hspace{1in}
\begin{tikzpicture}[scale=1]
\node (A) at (-1,-1) {2};
\node (B) at (1,-1) {21};
\node (C) at (-1,1) {31};
\node (D) at (1,1) {31};
\node (E) at (0,1.25){3c};
\node (F) at (0,-.75){c};
\draw (A) -- (B)
(C) -- (D)
(B) -- (D)
(A) -- (C);
\end{tikzpicture}
\end{center}

If $\lambda=\nu$ and the bottom edge is labeled with $i$ but not with $c$:

\begin{itemize}
 \item[(9)] If there is no box of $\mu/\lambda$ immediately to the right of or immediately below the box at the end of row $i$ of $\nu$, 
then $\gamma=\mu$ and the top edge is labeled by $i$.

\item[(10)] If there is a box of $\mu/\lambda$ directly below the box at the end of row $i$ of $\nu$, then $\gamma=\mu$ and the edge between
them is labeled by $i+1$.

\item[(11)] If there is a box of $\mu/\lambda$ immediately to the right of the box at the end of row $i$ of $\nu$, say in position $(i,j+1)$, no box of $\mu/\lambda$ in 
row $i+1$, and the outer corner of row $i+1$ is not in column $j$, then $\gamma/\mu$ is one box in row $i+1$.

\item[(12)] If there is a box of $\mu/\lambda$ immediately to the right of the box at the end of row $i$ of $\nu$, say in position $(i,j+1)$, no box of $\mu/\lambda$ in 
row $i+1$, and the outer corner of row $i+1$ is in column $j+1$, 
then the outer corner of row $i+1$ is on the main diagonal. In this case, $\gamma=\mu$, and the top edge is labeled with $j+1$ and $c$. 
\end{itemize}

 \begin{center}
\begin{tikzpicture}[scale=1]
\node (A) at (-1,-1) {3};
\node (B) at (1,-1) {3};
\node (C) at (-1,1) {31};
\node (D) at (1,1) {31};
\node (E) at (0,1.25){1};
\node (F) at (0,-.75){1};
\draw (A) -- (B)
(C) -- (D)
(B) -- (D)
(A) -- (C);
\end{tikzpicture}\hspace{.15in}
\begin{tikzpicture}[scale=1]
\node (A) at (-1,-1) {2};
\node (B) at (1,-1) {2};
\node (C) at (-1,1) {21};
\node (D) at (1,1) {21};
\node (E) at (0,1.25){2};
\node (F) at (0,-.75){1};
\draw (A) -- (B)
(C) -- (D)
(B) -- (D)
(A) -- (C);
\end{tikzpicture}\hspace{.15in}
\begin{tikzpicture}[scale=1]
\node (A) at (-1,-1) {2};
\node (B) at (1,-1) {2};
\node (C) at (-1,1) {3};
\node (D) at (1,1) {31};
\node (F) at (0,-.75){1};
\draw (A) -- (B)
(C) -- (D)
(B) -- (D)
(A) -- (C);
\end{tikzpicture}\hspace{.15in}
\begin{tikzpicture}[scale=1]
\node (A) at (-1,-1) {1};
\node (B) at (1,-1) {1};
\node (C) at (-1,1) {2};
\node (D) at (1,1) {2};
\node (E) at (0,1.25){2c};
\node (F) at (0,-.75){1};
\draw (A) -- (B)
(C) -- (D)
(B) -- (D)
(A) -- (C);
\end{tikzpicture}
\end{center}

If the $\lambda=\nu$ and the bottom edge is labeled with $j$ and with $c$:

\begin{itemize}
 \item[(13)] If there is no box of $\mu/\lambda$ immediately to the right of or immediately below the box at the end of column $j$ of $\nu$, 
then $\gamma=\mu$ and the edge between them is labeled by $j$ and $c$.

\item[(14)] If there is a box of $\mu/\lambda$ immediately to the right of the box at the end of column $j$ of $\nu$, then $\gamma=\mu$ and the top edge is labeled by $j+1$ and $c$.

\item[(15)] If there is a box of $\mu/\lambda$ directly below the box at the end of column $j$ of $\nu$ but no box of $\mu/\lambda$ in 
column $j+1$, then $\gamma/\mu$ is one box in column $j+1$. The top edge is labeled $c$.

\end{itemize}

 \begin{center}
\begin{tikzpicture}[scale=1]
\node (A) at (-1,-1) {41};
\node (B) at (1,-1) {41};
\node (C) at (-1,1) {42};
\node (D) at (1,1) {42};
\node (E) at (0,1.25){4c};
\node (F) at (0,-.75){4c};
\draw (A) -- (B)
(C) -- (D)
(B) -- (D)
(A) -- (C);
\end{tikzpicture}\hspace{.15in}
\begin{tikzpicture}[scale=1]
\node (A) at (-1,-1) {31};
\node (B) at (1,-1) {31};
\node (C) at (-1,1) {41};
\node (D) at (1,1) {41};
\node (E) at (0,1.25){4c};
\node (F) at (0,-.75){3c};
\draw (A) -- (B)
(C) -- (D)
(B) -- (D)
(A) -- (C);
\end{tikzpicture}\hspace{.15in}
\begin{tikzpicture}[scale=1]
\node (A) at (-1,-1) {31};
\node (B) at (1,-1) {31};
\node (C) at (-1,1) {32};
\node (D) at (1,1) {42};
\node (F) at (0,-.75){3c};
\draw (A) -- (B)
(C) -- (D)
(B) -- (D)
(A) -- (C);
\end{tikzpicture}
\end{center}

We call the resulting array the \textit{shifted Hecke growth diagram} of $w$. In our previous example with
$w=4211232$, we would have the diagram below.

\begin{center}
\begin{tikzpicture}[scale=1]
\node (00) at (0,0) {$\varnothing$};
\node (10) at (1,0) {$\varnothing$};
\node (20) at (2,0) {$\varnothing$};
\node (30) at (3,0) {$\varnothing$};
\node (40) at (4,0) {$\varnothing$};
\node (50) at (5,0) {$\varnothing$};
\node (01) at (0,1) {$\varnothing$};
\node (.53.5) at (.5,3.5) {$X$};
\node (11) at (1,1) {$\varnothing$};
\node (21) at (2,1) {$\varnothing$};
\node (31) at (3,1) {$1$};
\node (41) at (4,1) {$1$};
\node (51) at (5,1) {$1$};
\node (02) at (0,2) {$\varnothing$};
\node (12) at (1,2) {$\varnothing$};
\node (22) at (2,2) {$1$};
\node (32) at (3,2) {$2$};
\node (1.51.5) at (1.5,1.5) {$X$};
\node (42) at (4,2) {$2$};
\node (52) at (5,2) {$2$};
\node (03) at (0,3) {$\varnothing$};
\node (13) at (1,3) {$\varnothing$};
\node (23) at (2,3) {$1$};
\node (33) at (3,3) {$2$};
\node (43) at (4,3) {$2$};
\node (53) at (5,3) {$2$};
\node (2.5.5) at (2.5,.5) {$X$};
\node (04) at (0,4) {$\varnothing$};
\node (14) at (1,4) {$1$};
\node (24) at (2,4) {$2$};
\node (34) at (3,4) {$3$};
\node (44) at (4,4) {$3$};
\node (54) at (5,4) {$31$};
\node (3.5.5) at (3.5,.5) {$X$};
\node (4.51.5) at (4.5,1.5) {$X$};
\node (5,52.5) at (5.5,2.5) {$X$};
\node (60) at (6,0) {$\varnothing$};
\node (61) at (6,1) {$1$};
\node (62) at (6,2) {$2$};
\node (63) at (6,3) {$3$};
\node (64) at (6,4) {$31$};
\node (70) at (7,0) {$\varnothing$};
\node (71) at (7,1) {$1$};
\node (72) at (7,2) {$2$};
\node (73) at (7,3) {$31$};
\node (74) at (7,4) {$32$};
\node (6.51.5) at (6.5,1.5) {$X$};
\node (3.51.15) at (3.5,1.15) {\tiny 1};
\node (3.52.15) at (3.5,2.15) {\tiny 2c};
\node (3.53.15) at (3.5,3.15) {\tiny 2c};
\node (3.54.15) at (3.5,4.15) {\tiny 3c};
\node (2.52.15) at (2.5,2.15) {\tiny c};
\node (2.53.15) at (2.5,3.15) {\tiny c};
\node (2.54.15) at (2.5,4.15) {\tiny c};
\node (1.54.15) at (1.5,4.15) {\tiny c};
\node (4.52.15) at (4.5,2.15) {\tiny 1};
\node (4.53.15) at (4.5,3.15) {\tiny 1};
\node (5.54.15) at (5.5,4.15) {\tiny 2};
\node (6.52.15) at (6.5,2.15) {\tiny 1};
\node (6.54.15) at (6.5,4.15) {\tiny c};

\draw (00)--(10)--(20)--(30)--(40)--(50)--(60)--(70)
(01)--(11)--(21)--(31)--(41)--(51)--(61)--(71)
(02)--(12)--(22)--(32)--(42)--(52)--(62)--(72)
(03)--(13)--(23)--(33)--(43)--(53)--(63)--(73)
(04)--(14)--(24)--(34)--(44)--(54)--(64)--(74)
(00)--(01)--(02)--(03)--(04)
(10)--(11)--(12)--(13)--(14)
(20)--(21)--(22)--(23)--(24)
(30)--(31)--(32)--(33)--(34)
(40)--(41)--(42)--(43)--(44)
(50)--(51)--(52)--(53)--(54)
(60)--(61)--(62)--(63)--(64)
(70)--(71)--(72)--(73)--(74);
\end{tikzpicture}
\end{center}

Let $\mu_0 = \varnothing \subseteq \mu_1 \subseteq \ldots \mu_k$ be the seqence of shifted partitions across the top of the 
growth diagram, and let $\nu_0=\varnothing \subseteq \nu_1 \subseteq \ldots \nu_n$ be the sequence of shifted partitions 
on the right side of the growth diagram. These sequences correspond to increasing shifted tableaux
$Q(w)$ and $P(w)$, respectively. 
If the edge between $\mu_i$ and $\mu_{i+1}$ is labeled $j$, then $\mu_i=\mu_{i+1}$, and the label $i+1$ of $Q(w_1\cdots w_{i+1})$ 
is placed in the box at the end of row $j$ of $Q(w_1\cdots w_i)$. If the edge between $\mu_i$ and $\mu_{i+1}$ is labeled $jc$, then $\mu_i=\mu_{i+1}$, and the label $i+1'$ of $Q(w_1\cdots w_{i+1})$ 
is placed in the box at the end of column $j$ of $Q(w_1\cdots w_i)$. For the shifted Hecke growth diagram above, we have:

\begin{center}
$P(w)=$
\begin{ytableau}
 1 & 2 & 3 \\
\none & 3 & 4
\end{ytableau}\hspace{.1in}
\hspace{1in}
 $Q(w)=$
 \begin{ytableau}
 1 & 2' & 3'4' \\
\none & 56 & 7'
 \end{ytableau}\hspace{.1in}.
\end{center}
We see that this $P$ and $Q$ agree with $P_S(w)$ and $Q_S(w)$, which is not a coincidence.

\begin{theorem}
 For any word $w$, the increasing shifted tableaux $P$ and $Q$ obtained from the sequence of shifted partitions across the right side of the shifted Hecke growth diagram for $w$
and across the top of the shifted Hecke growth diagram for $w$, respectively, are $P_S(w)$ and $Q_S(w)$, the shifted Hecke insertion and recording tableau for $w$.
\end{theorem}

\begin{proof}
Suppose the square we are examining is in row $s$ and column $t$. We argue by induction. We will use the same two lemmas from the Hecke growth diagram proof, which have
analogous proofs.

\begin{lemma} Oriented as in the square above, $|\nu/\lambda|\leq 1$.\end{lemma}

\begin{lemma} Oriented as in the square above, $\mu/\lambda$ is a rook strip, that is, no two boxes 
in $\mu/\lambda$ are in the same row or column.\end{lemma}

(1-2) Note that if the square contains an $X$, then $\lambda=\nu$ because there are no other $X$'s in column $t$. If $\lambda_1=\mu_1$, then there are no $s$'s in the first row 
of insertion tableau $\mu$. Thus the inserted $s$ will be added to the end of the first row, creating $\gamma$. If $\lambda_1=\mu_1$, then there is already an $s$ at the end ofthe first row 
of insertion tableau $\mu$. The inserted $s$ cannot be added to the first row, so the $\mu=\gamma$. In this case, the insertion terminated at the end of the first row. We know this is an inner corner 
because $s$ is the largest number inserted so far, so there can not be anything directly below the box labeled $s$ in the first row.

(3) If $\mu=\lambda$, there is no $X$ to the left of square $(s,t)$ in row 
$s$. Thus nothing changes between $\nu$ and $\gamma$ as we consider 
occurences of entry $s$. If $\lambda=\nu$, there is no $X$ below square 
$(s,t)$ in column $t$ and so no new insertion.Thus nothing changes between $\mu$ and $\gamma$.

(4) Suppose $\nu/\lambda$ is one box in row $n$ and $\mu/\lambda$ contains boxes 
in rows exactly $\{j_1,j_2,\ldots,j_k\}$. Suppose the box in $\nu/\lambda$ corresponds to 
inserting $r<s$. Then $\nu \nsubseteq \mu$ implies 
that $n\notin \{j_1,j_2,\ldots,j_k\}$. Since $\nu/\lambda$ is one box,
the last action in the insertion sequence of $r$ into the insertion tableau of 
shape $\lambda$ is a box being inserted in row $n$. Since there is no $s$ in row $n$,
the bumping sequence when inserting $r$ into the insertion tableau of shape $\mu$ 
will not disturb the $s$'s. The edge between $\lambda$ and $\nu$ is either labeled $c$ or not, corresponding 
to whether or not inserting the number in column $t$ involves column insertions. This remains unchanged.

(5-6) Since $\nu/\lambda$ is one box in position $(i,j)$, there is some $X$ in 
position $(r,t)$ for $r<s$. It follows that inserting $r$ into the insertion 
tableau $\mu$ will result in bumping the $s$ in position $(i,j)$
of $\mu$. If there is no box in row $i+1$ of $\mu/\lambda$, there is not an $s$ in row $i+1$ of insertion 
tableau $\mu$. Thus the $s$ in position $(i,j)$ is bumped from row $i$ by $r$ and added to the end of row 
$i+1$. If there is a box in row $i+1$ of $\mu/\lambda$, there is already an $s$ in row $i+1$ of insertion tableau 
$\mu$. It follows that the $s$ bumped from row $i$ cannot be added to row $i+1$, and so $\gamma=\mu$. The insertion 
terminates at the end of row $i+1$ since there cannot be any boxes directly below the $s$ in row $i+1$.

(7-8) If $(i,j)$ is on the main diagonal or the edge between $\lambda$ and $\nu$ is labeled $c$, any bumped entries 
will be column inserted into the column $j$. The rules are the transpose of rules 5 and 6, and the new edge is labeled $c$ since 
all subsequent bumping will be column inserted after the first column insertion.

(9) Since there is no box directly to the right of $(i,j)$, nothing is bumped when inserting $r<s$ into 
insertion tableau $\mu$. Since there is no box directly below $(i,j)$, the insertion terminates at inner 
corner $(i,j)$.

(10) In the situation described above, if there is a box directly below $(i,j)$, the insertion now terminates 
at the inner corner in row $i+1$.

(11) When inserting $r$ into $\mu$, the $s$ in position $(i,j+1)$ is bumped but not replaced and inserted 
into row $i+1$. Since there is no $s$ in row $i+1$, this $s$ can be added to the row as long as it is 
not directly below the $s$ in position $(i,j)$. 

(12)  When inserting $r$ into $\mu$, the $s$ in position $(i,j+1)$ is bumped but not replaced and inserted 
into row $i+1$. Since there is no $s$ in row $i+1$, we attempt to add this $s$ to the end of row $i+1$. 
However, this $s$ would end directly below the $s$ in row $i$, and so cannot be added. If the outer 
corner of row $i+1$ is not on the main diagonal, the insertion terminates at the end of row $i+1$. If it is
on the diagonal, the insertion terminates at the end of row $i$, and we know the last box in row 
$i$ must be in column $j$ since there are no entries larger than $s$. This edge is labeled with $c$ 
since this action may bump an entry from the main diagonal in column $j$ in future steps.

(13-15) Rules 13-15 are the tranpose of rules 10-12, and may be explained in an analogous way.
\end{proof}

We can also formulate the rules for the reverse insertion, as follows. The proof is omitted for brevity. 

{\bf {Reverse Rules}}
\begin{itemize}
\item[(R1)] If $\gamma/\mu$ is one box in the first row, then the square has a $X$ and $\lambda=\nu$.

\item[(R2)] If $\gamma=\mu$ and the edge between them is labeled 1, then the square has an $X$ and $\lambda=\nu$.

\item[(R3)] If $\gamma=\mu$ with no edge label, then $\lambda=\mu$. If $\gamma=\nu$, then $\lambda=\mu$ and the edge label between $\lambda$ and $\nu$ is the same as the label 
between $\gamma$ and $\mu$.

\item[(R4)] If $\nu\not\subseteq\mu$, then $\lambda=\mu \cap \nu$.

\item[(R5)] If $\gamma/\mu$ is one box in row $i+1$ for some $i>0$, the edge between $\gamma$ and $\mu$ has no label, 
and $\mu/\nu$ has no box in row $i$, then $\nu/\lambda$ is one box in row $i$.

\item[(R11)] If $\gamma/\mu$ is one box in row $i+1$ for some $i>0$, the edge between $\gamma$ and $\mu$ has no label, 
and $\mu/\nu$ has a box in row $i$, then $\lambda=\nu$ and the edge between them is labeled $i$.

\item[(R7)] If $\gamma/\mu$ is one box in column $j+1$ for some $j>0$, the edge between $\gamma$ and $\mu$ is labeled $c$, and $\mu/\nu$ has no box in column $j$, 
then $\nu/\lambda$ is one box in column $j$ and the edge between $\nu$ and $\lambda$ is labeled $c$.

\item[(R15)] If $\gamma/\mu$ is one box in column $j+1$ for some $j>0$, the edge between $\gamma$ and $\mu$ is labeled $c$, and $\mu/\nu$ has a box in column $j$, 
then $\nu=\lambda$ and the edge between them is labeled $j$ and $c$.

\item[(R6)] If $\gamma=\mu$, the edge between them is labeled $i+1$ for some $i>0$, $\mu/\nu$ has a box in row $i+1$, and $\mu_i\neq \mu_{i+1}+1$, then 
$\nu/\lambda$ is one box in row $i$.

\item[(R9)] If $\gamma=\mu$, the edge between them is labeled $i+1$ for some $i>0$, $\mu/\nu$ has no box in row $i+1$, then $\lambda=\nu$ and the edge 
between them is labeled $i+1$.

\item[(R10)] If $\gamma=\mu$, the edge between them is labeled $i+1$ for some $i>0$, $\mu/\nu$ has a box in row $i+1$, and $\mu_i=\mu_{i+1}+1$, then 
$\lambda=\nu$ and the edge between them is labeled $i$.

\item[(R8)] If $\gamma=\mu$, the edge between them is labeled $j+1$ and $c$ for some $j>0$, $\mu/\nu$ has a box in column $j+1$, 
and the $j^{th}$ and $j+1^{th}$ columns of $\mu$ are not the same size, then $\nu/\lambda$ is one box in column $j$.

\item[(R12)] If $\gamma=\mu$, the edge between them is labeled $j+1$ and $c$ for some $j>0$, and the bottom box of column $j+1$ is in the last row of $\mu$,
then $\lambda=\nu$ and the edge between them is labeled with the bottom row of $\nu$.

\item[(R13)] If $\gamma=\mu$, the edge between them is labeled $j+1$ and $c$ for some $j>0$, $\mu/\mu$ has no box in column $j+1$, then $\lambda=\nu$ and the 
edge between them is labeled $j+1$ and $c$.

\item[(R14)] If $\gamma=\mu$, the edge between them is labeled $j+1$ and $c$ for some $j>0$, $\mu/\mu$ has a box in column $j+1$, 
and and the $j^{th}$ and $j+1^{th}$ columns of $\mu$ are the same size, then $\lambda=\nu$ and the 
edge between them is labeled $j$ and $c$.
\end{itemize}

\section{Major examples: Young-Fibonacci lattice} \label{sec:YF}

\subsection{M\"obius deformation of the Young-Fibonacci lattice}
In order to describe the M\"obius deformation of the Young-Fibonacci lattice, we first need to determine its M\"obius function.

\begin{theorem} Let $X>Y$ be elements in the Young-Fibonacci lattice and suppose $X$ has $n$ edges below it. Then 
 $$\mu(Y,X)=\left\{ \begin{array}{ll}
                    n-1 &  \text{if }2Y=X  \\
		    -1 & \text{if $X$ covers $Y$} \\
		    0 &  \text{otherwise}
                   \end{array}
\right. 
$$
\end{theorem}

\begin{proof}
Suppose $X>Y$. If $\rho(X)=\rho(Y)+1$, then clearly, $\mu(Y,X)=-1$. 

If $\rho(X)=\rho(Y)+2$ and $X=2Y$, then $\mu(Y,X)=n-1$ since $2Y$ 
covers exactly the elements that cover $Y$. If $X\neq 2Y$, then it can be checked that $X$ covers exactly one element that covers $Y$. Therefore $\mu(Y,X)=0$. 

Assume now $\rho(X) \geq \rho(Y)+2$. We will argue that $\mu(Y,X)=0$ by induction on $\rho(X)$. The base case $\rho(X)=\rho(Y)+2$ was just proved. 
Let us argue the step of induction. There are two cases.
\begin{itemize}
 \item We have $2Y < X$. In this case, consider $\displaystyle\sum_{Y\leq W < X} \mu(Y,W)$, we argue it is $0$. Indeed, $\displaystyle\sum_{Y\leq W \leq 2Y} \mu(Y,W) =0$.
For the $2Y \not \leq W < X$ in the interval, we use induction assumption to conclude $\mu(Y,W)=0$.
 \item We have $2Y \not < X$. We claim there is only one $Z$ which covers $Y$ such that $Z < X$. Indeed, since Young-Fibonacci lattice is a lattice, and since 
join of any two distinct such $Z$'s is $2Y$, we would have a contradiction. Then $\mu(Y,Y)+\mu(Y,Z)=1 + (-1)=0$. For the rest of $Y < W < X$ we have $\mu(Y,W)=0$
by induction assumption.
\end{itemize}
\end{proof}

By the previous result, for each element $X$ in the Young-Fibonacci lattice, the M\"obius deformation is formed by adding an upward-oriented loop for each of the $n$
elements $X$ covers and adding $n-1$ downward-oriented edges from $X$ to $Y$ if $X=2Y$. For example, there are 2 edges from 221 to 21. The first few ranks are shown below.

\begin{minipage}{.5\textwidth}
\begin{center}
\begin{tikzpicture}[scale=1]
\ytableausetup{boxsize=.15cm}
\node (empty) at (0,0) {$\varnothing$};
\node (1) at (0,1) {\begin{ytableau} $ $ \end{ytableau}};
\node (11) at (1,2) {\begin{ytableau} $ $ & $ $ \end{ytableau}};
\node (2) at (-1,2) {\begin{ytableau} $ $ \\ $ $ \end{ytableau}};
\node (12) at (-2,3) {\begin{ytableau} \none & $ $ \\ $ $ & \end{ytableau}};
\node (21) at (0,3) {\begin{ytableau} $ $ & \none \\ $ $ & $ $ \end{ytableau}};
\node (111) at (2,3) {\begin{ytableau} $ $ & & \end{ytableau}};
\node (112) at (-3,4) {\begin{ytableau} \none & \none & $ $ \\ $ $ & & \end{ytableau}};
\node (22) at (-1,4) {\begin{ytableau} $ $ & \\ $ $ & \end{ytableau}};
\node (121) at (0,4) {\begin{ytableau} \none & $ $ & \none \\ $ $ & & \end{ytableau}};
\node (211) at (1,4) {\begin{ytableau} $ $ & \none & \none \\ $ $ & & \end{ytableau}};
\node (1111) at (3,4) {\begin{ytableau} $ $ & & & \end{ytableau}};
\draw (empty) -- (1) --(11)--(111)--(1111)
(1)--(2)--(12)--(112)
(11)--(21)--(22)
(1) to [out=60, in=100, distance=.5cm] (1)
(2) to [out=60, in=100, distance=.5cm] (2)
(11) to [out=60, in=100, distance=.5cm] (11)
(12) to [out=60, in=100, distance=.5cm] (12)
(21) to [out=-60, in=-100, distance=.5cm] (21)
(21) to [out=-60, in=-100, distance=.7cm] (21)
(111) to [out=60, in=100, distance=.5cm] (111)
(112) to [out=60, in=100, distance=.5cm] (112)
(22) to [out=60, in=100, distance=.5cm] (22)
(22) to [out=60, in=100, distance=.7cm] (22)
(121) to [out=60, in=100, distance=.5cm] (121)
(211) to [out=60, in=100, distance=.5cm] (211)
(211) to [out=60, in=100, distance=.7cm] (211)
(1111) to [out=60, in=100, distance=.5cm] (1111)
(2)--(21)--(211)
(111)--(211)
(21)--(121)
(12)--(22);
\end{tikzpicture}
\end{center}
\end{minipage}
\begin{minipage}{.5\textwidth}
\begin{center}
\begin{tikzpicture}[scale=1]
\ytableausetup{boxsize=.15cm}
\node (empty) at (0,0) {$\varnothing$};
\node (1) at (0,1) {\begin{ytableau} $ $ \end{ytableau}};
\node (11) at (1,2) {\begin{ytableau} $ $ & $ $ \end{ytableau}};
\node (2) at (-1,2) {\begin{ytableau} $ $ \\ $ $ \end{ytableau}};
\node (12) at (-2,3) {\begin{ytableau} \none & $ $ \\ $ $ & \end{ytableau}};
\node (21) at (0,3) {\begin{ytableau} $ $ & \none \\ $ $ & $ $ \end{ytableau}};
\node (111) at (2,3) {\begin{ytableau} $ $ & & \end{ytableau}};
\node (112) at (-3,4) {\begin{ytableau} \none & \none & $ $ \\ $ $ & & \end{ytableau}};
\node (22) at (-1,4) {\begin{ytableau} $ $ & \\ $ $ & \end{ytableau}};
\node (121) at (0,4) {\begin{ytableau} \none & $ $ & \none \\ $ $ & & \end{ytableau}};
\node (211) at (1,4) {\begin{ytableau} $ $ & \none & \none \\ $ $ & & \end{ytableau}};
\node (1111) at (3,4) {\begin{ytableau} $ $ & & & \end{ytableau}};
\draw (empty) -- (1) --(11)--(111)--(1111)
(1)--(2)--(12)--(112)
(11)--(21)--(22)
(2)--(21)--(211)
(111)--(211)
(21)--(121)
(12)--(22)
(21)--(1)
(22)--(2)
(211)--(11);
\end{tikzpicture}
\end{center}
\end{minipage}
\ytableausetup{boxsize=.5cm}

\begin{theorem} 
The M\"obius deformation of the Young-Fibonacci lattice forms a dual graded graph with $$(DU-UD)=D+I.$$ 
\end{theorem}

\begin{proof}
 Suppose that shape $X$ covers shape $Y$, and first assume that $\rho(Y)=\rho(X)-1$. The up-down paths from $X$ to $Y$ that consist of a loop at $X$ 
followed by following the edge from $X$ to $Y$ can be counted by the number of loops at $X$, i.e. the number of edges below $X$ in the Young-Fibonacci 
lattice. In the second type of up-down path, we start at $X$, move up to a distinct shape $W$, and then move down to $Y$. The only $W$ with 
$\rho(W)=\rho(Y)+2$ that cover $Y$ is the $W$ defined by $W=2Y$. This $W$ covers $X$ by definition of the Young-Fibonacci lattice. To count the number 
of paths through $W$, we need to count the number of edges from $W$ to $Y$. This number is the same as one fewer than the number edges below $W$ in $\mathbb{YF}$, 
which is the same as one fewer than the number of edges above $Y$ in $\mathbb{YF}$.

To count down-up paths, we first count the paths that consist of the edge from $X$ to $Y$ followed by a loop at $Y$. These can be counted by the number of loops at 
$Y$ or, in other words, the one fewer than the number of edges above $Y$ in $\mathbb{YF}$. The second type of down-up paths involve traveling from $X$ 
down to some $Z$ with $\rho(Z)+2=\rho(X)$ and then going up to $Y$. There is a unique such $Z$ defined by $2Z=X$, and $Y$ covers this $Z$. We thus 
need to count the number of ways to get from $X$ to $Z$, which is the same as one fewer than the number of edges below $X$ in $\mathbb{YF}$.

Putting the previous two arguments together, $[Y](DU-UD)(X)=1$ in this case, as desired.

Now suppose $\rho(Y)=\rho(X)-2$. The only up-down paths from $X$ to $Y$ consist of a loop at $X$ followed by an edge from $X$ to $Y$. There are 
$$(\#\{\text{edges below $X$ in $\mathbb{YF}$}\})(\#\{\text{edges below $X$ in $\mathbb{YF}$}\}-1)$$ such paths. 

To count down-up paths, we must count the number of ways to choose an edge from $X$ to $Y$ and then choose a loop of $Y$. There are 
$$(\#\{\text{edges below $X$ in $\mathbb{YF}$}\}-1)(\#\{\text{edges above $Y$ in $\mathbb{YF}$}\}-1)$$ such paths. Since 
$$\#\{\text{edges down from $X$ in $\mathbb{YF}$}\}=\#\{\text{edges above $Y$ in $\mathbb{YF}$}\},$$ 
we have that $$[Y](DU-UD)(X)=\#\{\text{edges below $X$ in $\mathbb{YF}$}\}-1=[Y]D(X).$$
\end{proof}

\subsection{K-Young-Fibonacci insertion} \label{sec:KYFinsertion}
As with the previous examples of M\"obius construction, we describe an insertion procedure that corresponds to the M\"obius deformation of the Young-Fibonacci lattice, which is an analogue of the Young-Fibonacci insertion of Fomin. We recommend \cite{F} for details about Young-Fibonacci insertion. In this analogue, any walk upward from $\varnothing$ will correspond to a recording tableau and any walk downward to $\varnothing$ will correspond to an insertion tableau. We begin by describing the
insertion tableaux in this setting.

\begin{definition}\label{def:KYFtableau}
A \textit{K-Young-Fibonacci (KYF) tableau} is a filling of a snakeshape with positive integers such that

\begin{itemize}
 \item[(i)] for any pair \begin{ytableau} B \\ A \end{ytableau} the inequality $A\leq B$ holds;
 \item[(ii)] to the right of any \begin{ytableau} B \\ A \end{ytableau}, there are no numbers from the interval $[A,B]$ 
 in either the upper or lower rows;
 \item[(iii)] if the position above \begin{ytableau} A \end{ytableau} is not occupied yet, then to the right of 
 \begin{ytableau} A \end{ytableau} there are no numbers greater than or equal to $A$ in either the lower or upper rows;
\item[(iv)] any pair \begin{ytableau} A \\ A \end{ytableau} must come before any single box \begin{ytableau} B  \end{ytableau}
and must not be in rightmost column of the snakeshape. In addition, $A$ may not be the smallest entry in the snakeshape.
\end{itemize}
\end{definition}

\begin{example}The three snakeshapes below are valid KYF tableaux.
\begin{center}
\begin{ytableau}
 \none & 3 & 4 & \none & 7 \\
 9 & 2 & 1 & 8 & 6 & 5
\end{ytableau}\hspace{.5in}
\begin{ytableau}
2 & 3 & \none & \none & 5 \\
 2 & 1 & 7 & 6 & 4 
\end{ytableau}\hspace{.5in}
 \begin{ytableau}
4 & 5 & \none & \none & 2 \\
 4 & 5 & 6 & 3 & 1 
\end{ytableau}
\end{center}

The snakeshapes below are not valid KYF tableaux. The first violates condition (iii), the second violates condition (ii),
and the third. fourth, and fifth violate condition (iv).
\begin{center}
\begin{ytableau}
 \none & 3 & 4 \\
 4 & 2 & 1
\end{ytableau}\hspace{.5in}
\begin{ytableau}
 3 & 4 & \none \\
2 & 1 & 4
\end{ytableau}\hspace{.5in}
 \begin{ytableau}
4 & 6 & \none & \none & 2 \\
 4 & 5 & 7 & 3 & 2 
\end{ytableau}\hspace{.5in}
 \begin{ytableau}
3 & 2 \\
1 & 2
\end{ytableau}\hspace{.5in}
\begin{ytableau}
 1 & 3 \\
 1 & 2
\end{ytableau}
\end{center}
\end{example}

Let $\tau$ be a valid KYF tableau. We show how to insert positive integer $x$ into $\tau$ to obtain a new 
valid KYF tableau that may or may not be different than $\tau$. Notice that this insertion procedure is different than Hecke insertion and shifted Hecke insertion 
in that the algorithm does not proceed row-by-row or column-by-column. As before, we designate a specific box of the resulting tableau, $\tau'$, as being the box where the insertion terminated. 
We will later use this notion to define recording tableaux. 

The rules for inserting any positive integer $x$ into $\tau$ are as follows:

\begin{enumerate}
 \item[(YF1)] If $x$ is equal to the smallest entry in $\tau$, then $\tau=\tau'$. In this case, we say the insertion terminated in the 
 top cell of the first column of $\tau=\tau'$. If this is not the case, continue to step 2.  
 
 \begin{example}
  Inserting $1$ into the tableau on the left or $2$ into the tableau on the right will not change the tableaux.
  \begin{center}
  \begin{ytableau}
   \none & 2 & \none \\
   4 & 1 & 3
  \end{ytableau}\hspace{1in}
\begin{ytableau}
3 & \none & \none \\
2 & 5 & 4
  \end{ytableau} 
 \end{center} 
 \end{example}

 \item[(YF2)] Attach a new box \begin{ytableau} x \end{ytableau} just to the left of $\tau$ in the lower row. 
 
 \item[(YF3)] Find all the entries of $\tau$ that are greater than or equal to $x$ and sort them:
 $$x\leq a_1 \leq a_2 \leq \ldots \leq a_k.$$ If in $\tau$ we have \begin{ytableau} A \\ A \end{ytableau}, then we 
 consider the $A$ in the upper row to be the larger of the two.
 
 \item[(YF4)] If $a_i=a_{i+1}$, then replace $a_{i+1}$ with $*$.
 
 \item[(YF5)] Now put a box \begin{ytableau} \bullet \end{ytableau} just above \begin{ytableau} x \end{ytableau} and move the 
 $a_i$ chainwise according to the following rule: $a_1$ moves into the new box, $a_2$ moves to $a_1$'s original position, 
 $a_3$ moves to $a_2$'s original position, etc. The box that was occupied by $a_k$ disappears. If it was located in the 
 lower row, then the left and right parts of the snake are concatenated. 
 
 \item[(YF6)] If there is a box with $*$ and no box directly above it, delete this box and concatenate the left and right parts of the snake. If this happens, we say the insertion terminated at the box in the 
 upper row of the column directly to the right of this concatenation. Otherwise, a box was added in the insertion process, and we say the insertion terminated at this new box.
 
 \item[(YF7)] Replace any pair \begin{ytableau} A \\ $*$ \end{ytableau} with \begin{ytableau} A \\ A \end{ytableau}.
 
\end{enumerate}

\begin{example}
 Let's insert 3 into $\tau=$\begin{ytableau} 3 & 4 & \none \\ 2 & 4 & 1 \end{ytableau}.
 We first attach \begin{ytableau} 3 \end{ytableau} as shown below.
 \begin{center} \begin{ytableau} \none & 3 & 4 & \none \\ 3 & 2 & 4 & 1 \end{ytableau} \end{center}
 Next, we locate and order $$3\leq (a_1=3) \leq (a_2=4) \leq (a_3 =4), $$ and we replace $a_3$ with $*$. After shifting the boxes 
 as in YF5, we have 
 \begin{center}
  \begin{ytableau}
   3 & 4 & \none & \none \\
   3 & 2 & $*$ & 1
  \end{ytableau}\hspace{.1in}.
  \end{center}
  We then delete the box with $*$ to obtain the final KYF tableau below.
   \begin{center}
  \begin{ytableau}
   3 & 4 & \none \\
   3 & 2 & 1
  \end{ytableau}\hspace{.1in}.
  \end{center}
  This insertion terminates at the box at the top of the third column.
  
  Notice also that inserting 1 into $\tau$ does not change the tableau, and this insertion terminates at the box at the top of the first column.
\end{example}

\begin{example}
 We insert 2 into KYF tableau \begin{ytableau} \none & \none & 2 & \none \\ 4 & 3 & 2 & 1 \end{ytableau}. We first 
 attach a box containing 2 to the left of the original tableau. We then locate and order 
 $$ 2 \leq (a_1=2) \leq (a_2=2)\leq (a_3=3) \leq (a_4=4).$$ We replace $a_2$ with $*$ and shift the boxes to obtain 
 \begin{center}
  \begin{ytableau}
   2 & \none & 3 & \none \\
   2 & 4 & $*$ & 1
  \end{ytableau}\hspace{.1in}.
 \end{center}
After performing the last step of the insertion procedure, we end with the tableau shown below.
\begin{center}
 \begin{ytableau}
  2 & \none & 3 & \none \\
  2 & 4 & 3 & 1
 \end{ytableau}
\end{center}
This insertion terminated at the box at the top of the first column of the resulting tableau.
\end{example}

As usual, for a word $w=w_1w_2\ldots w_n$, we define $P_{YF}(w)$ by setting $$P_{YF}(w_1\ldots w_k)=P_{YF}(w_1\ldots w_{k-1}) \leftarrow w_k.$$

\begin{example}\label{ex:1334241}
The sequence of KYF tableaux below shows the intermediate tableaux obtained in computing $P_{YF}(1334241)$, which is shown on the right.

\begin{center}
\begin{ytableau} \none \\ 1 \end{ytableau}\hspace{.3in}\begin{ytableau} \none & \none \\ 3 & 1 \end{ytableau}\hspace{.3in}\begin{ytableau} 3 & \none \\ 3 & 1 \end{ytableau}\hspace{.3in}
\begin{ytableau} \none & 3 & \none \\ 4 & 3 & 1 \end{ytableau}\hspace{.3in}
 \begin{ytableau} 3 & 4 & \none \\ 2 & 4 & 1 \end{ytableau}\hspace{.3in}
 \begin{ytableau} 4 & 3 & \none \\ 4 & 2 & 1 \end{ytableau}\hspace{.3in}
 \begin{ytableau} 4 & 3 & \none \\ 4 & 2 & 1 \end{ytableau}\hspace{.3in}
 \end{center}
\end{example}

We need the next definition to define the recording tableaux for KYF insertion.

\begin{definition}\label{def:KYFsetvalued}
A \textit{standard set-valued KYF tableau} is a snakeshape whose boxes are filled with finite, nonempty subsets of positive 
integers that satisfy the following conditions, where $\bar{A},\bar{B},\bar{C}$ are subsets, $\bar{A}< \bar{B}$
when max$(\bar{A})<$min$(\bar{B})$, and the letters $[n]$ are used exactly once for some $n$:

\begin{itemize}
 \item[(i)] for any pair \begin{ytableau} \bar{B} \\ \bar{A} \end{ytableau} the inequality $\bar{A}< \bar{B}$ holds;
 \item[(ii)] to the right of any \begin{ytableau} \bar{B} \\ \bar{A} \end{ytableau}, there are no numbers from the interval $[\text{max}(\bar{A}),\text{min}(\bar{B})]$ 
 in either the upper or lower rows;
 \item[(iii)] if the position above \begin{ytableau} \bar{A} \end{ytableau} is not occupied yet, then to the right of 
 \begin{ytableau} \bar{A} \end{ytableau} there are no entries greater than min($\bar{A}$) in either the lower or upper rows;
\end{itemize}
\end{definition}

A recording tableau for a word $w=w_1w_2\ldots w_n$ is a standard set-valued KYF tableau and is obtained as follows. Begin with $Q_{YF}(\varnothing)=\varnothing$. If the insertion of $w_k$ into $P_{YF}(w_1\cdots w_{k-1})$ 
resulted in adding a new 
box to $P_{YF}(w_1\cdots w_{k-1})$, add this same box with label $k$ to $Q_{YF}(w_1\ldots w_{k-1})$ to obtain $Q_{YF}(w_1\ldots w_k)$. 

If the insertion of $w_k$ into $P_{YF}(w_1\cdots w_{k-1})$ did not change the shape 
of $P_{YF}(w_1\cdots w_{k-1})$, obtain $Q_{YF}(w_1\ldots w_k)$ from $Q_{YF}(w_1\ldots w_{k-1})$ by adding the label $k$ to the box where the insertion terminated.

\begin{example} 
In Example \ref{ex:1334241}, we computed $P_{YF}(1334241)$. We repeat this computation on the top row and show the corresponding steps of building $Q_{YF}(1334241)$ on the bottom row.

\begin{center}
 \begin{ytableau} \none \\ 1 \end{ytableau}\hspace{.3in}\begin{ytableau} \none & \none \\ 3 & 1 \end{ytableau}\hspace{.3in}\begin{ytableau} 3 & \none \\ 3 & 1 \end{ytableau}\hspace{.3in}
\begin{ytableau} \none & 3 & \none \\ 4 & 3 & 1 \end{ytableau}\hspace{.3in}
 \begin{ytableau} 3 & 4 & \none \\ 2 & 4 & 1 \end{ytableau}\hspace{.3in}
 \begin{ytableau} 4 & 3 & \none \\ 4 & 2 & 1 \end{ytableau}\hspace{.3in}
  \begin{ytableau} 4 & 3 & \none \\ 4 & 2 & 1 \end{ytableau}
  \end{center}
  
  \begin{center}
  \begin{ytableau} \none \\ 1 \end{ytableau}\hspace{.3in}\begin{ytableau} \none & \none \\ 2 & 1 \end{ytableau}\hspace{.3in}\begin{ytableau} 3 & \none \\ 2 & 1 \end{ytableau}\hspace{.3in}
\begin{ytableau} \none & 3 & \none \\ 4 & 2 & 1 \end{ytableau}\hspace{.3in}
 \begin{ytableau} 5 & 3 & \none \\ 4 & 2 & 1 \end{ytableau}\hspace{.3in}
 \begin{ytableau} 5 & 3 & \none \\ 4 & 2 & 16 \end{ytableau}\hspace{.3in}
  \begin{ytableau} 57 & 3 & \none \\ 4 & 2 & 16 \end{ytableau}\hspace{.3in}
 \end{center}
\end{example}

We next define a reverse insertion procedure so that given a pair $(P_{YF}(w),Q_{YF}(w))$, we can recover $w=w_1\ldots w_n$.

First locate the box containing the largest label of $Q_{YF}(w)$, call the label $n$ and its column $c$,
and find the corresponding box in $P_{YF}(w)$. Let $x$ denote the label in the leftmost box of the bottom row of $P_{YF}(w)$.

\begin{itemize}
 \item[(rYF1)] If the label $n$ of $Q$ was not the only label in its box and was located in the upper row of the first column, 
 then $P_{YF}(w)=P_{YF}(w_1\ldots w_{n-1}) \leftarrow s$, where $s$ is the smallest entry in $P_{YF}(w_1\ldots w_{n-1})$.
Finally, $Q_{YF}(w_1\ldots w_{n-1})$ is obtained from $Q_{YF}(w)$ by removing the 
 label $n$. 
 \end{itemize}
 
 In all remaining cases, $P_{YF}(w)=P_{YF}(w_1\ldots w_{n-1}) \leftarrow x$, and we describe how to constuct $P_{YF}(w_1\ldots w_{n-1})$. In each case,
 $Q_{YF}(w_1\ldots w_{n-1})$ is obtained from $Q_{YF}(w)$ by removing the 
 label $n$. If $n$ is the only label in its box, the box is removed. 

 \begin{enumerate}
 \item[(rYF2)] If the label $n$ of $Q$ was the only label in its box:
 \begin{enumerate}
 \item Delete the leftmost square in the bottom row.
 \item Let $k$ denote the largest entry in $P_{YF}(w)$, and sort the entries of $P_{YF}(w)$ as shown below:
 $$x \leq b_1 \leq b_2\leq\ldots\leq b_t=k.$$
 If we have \begin{ytableau} A \\ A \end{ytableau}, then we 
 consider the $A$ in the upper row to be the larger of the two.
 \item If $b_i=b_{i+1}$, replace $b_{i}$ with $*$.
 \item Move the $b_i$ chainwise according the following rule: $b_1$ moves into $b_2$'s position, $b_2$ moves into $b_3$'s position, etc until $b_{t-1}$ moves into $b_t$'s position.
 The box that was occupied by $b_1$ disappears. 
 \item Place $b_t$ in the position determined by the shape of $Q_{S,n-1}$. 
 \item Replace any pair \begin{ytableau} $*$ \\ A \end{ytableau} with \begin{ytableau} A \\ A \end{ytableau}.
 \end{enumerate}
 
 \item[(rYF3)] If the label $n$ of $Q$ was not the only label in its box and the largest entry in $P_{YF}(w)$ is $k$: 
 \begin{enumerate}
 \item Add the column \begin{ytableau} \bullet \\ k \end{ytableau} 
 directly to the left of column $c$ to $P_{YF}(w)$.
 \item Sort the entries of $P_{YF}(w)$ as shown below:
 $$k=b_1\geq b_2 \geq b_2 \ldots \geq x=b_t.$$
 If we have \begin{ytableau} A \\ A \end{ytableau}, then we 
 consider the $A$ in the upper row to be the larger of the two.
 \item If $b_i=b_{i+1}$, replace $b_{i+1}$ with $*$.
 \item Move the 
 $b_i$ chainwise according to the following rule: $b_1$ moves into the new box \begin{ytableau} \bullet \end{ytableau}, $b_2$ moves to $b_1$'s original position, 
 $b_3$ moves to $b_2$'s original position, etc. The box that was occupied by $b_t=x$ disappears. 
 \item Replace any pair \begin{ytableau} $*$ \\ A \end{ytableau} with \begin{ytableau} A \\ A \end{ytableau}.
\end{enumerate}
\end{enumerate}

The steps above clearly reverse the KYF insertion steps, giving us the result below.

\begin{theorem}
 KYF insertion and reverse KYF insertion define mutually inverse bijections between the set of words on $\mathbb{N}$
 and the set of pairs $(P_{YF},Q_{YF})$ consisting of a KYF tableau and a set-valued KYF tableau of the same shape.
\end{theorem}
\begin{proof}
 Omitted for brevity.
\end{proof}

\subsection{KYF growth and decay}\label{sec:KYFgrowth}
As before, given any word $w=w_1w_2\cdots w_k$, containing
$n \leq k$ distinct numbers, we can create an $n \times k$ array with an $X$ in the $w_i^{th}$ square from the bottom of column $i$. Note that there can be multiple $X$'s in the same row but is at most
one $X$ per column. 

We will label the corners of each square with a snakeshape, label some of the horizontal edges of the squares with
a specific inner corner by writing the column of the inner corner, and label some vertical edges corresponding to where a 2 was added to the bottom shape to obtain the top shape. For example, 
if the corner at the bottom of the veritcal edge is labeled with 221, the corner at the top is labeled 2221, and the vertical edge has label 3, this means that the third column of 2's in 
2221 is the column that was added when going from 221 to 2221. We begin by labeling all corners
in the bottom row and left side of the diagram with the empty shape, $\varnothing$.

To complete the labeling of the corners, suppose the corners $\mu$, $\lambda$, and $\nu$ are labeled, 
where $\mu$, $\lambda$, and $\nu$ are as in the picture below. We label $\gamma$ according to the following rules.

 \begin{center}
\begin{tikzpicture}[scale=1]
\node (A) at (-1,-1) {$\lambda$};
\node (B) at (1,-1) {$\nu$};
\node (C) at (-1,1) {$\mu$};
\node (D) at (1,1) {$\gamma$};
\draw (A) -- (B)
(C) -- (D)
(B) -- (D)
(A) -- (C);
\end{tikzpicture}
\end{center}

\textbf{If the square contains an X:}

\begin{itemize}
 \item[(1)] If $\lambda=\nu=\varnothing$, then $\gamma=\mu=1$ and the top edge is labeled 1.
 
\item[(2)] If $\lambda=\mu=\nu$, then $\gamma=1\lambda$.

\item[(3)] If $\mu=2\lambda$ and the left edge is labeled $i$, then $\gamma=2\lambda$, the top edge is labeled $i+1$, and the right edge is labeled 1.

\item[(4)] If $\mu\neq\lambda$, $\mu\neq 2\lambda$, and $\mu\neq1$, then $\gamma=2\lambda$ and the right edge is labeled 1. 
\end{itemize}

 \begin{center}
\begin{tikzpicture}[scale=1]
\node (A) at (-1,-1) {$\varnothing$};
\node (B) at (1,-1) {$\varnothing$};
\node (C) at (-1,1) {1};
\node (D) at (1,1) {1};
\node (E) at (0,0) {$X$};
\node (F) at (0,1.25){1};
\draw (A) -- (B)
(C) -- (D)
(B) -- (D)
(A) -- (C);
\end{tikzpicture}\hspace{.75in}
\begin{tikzpicture}[scale=1]
\node (A) at (-1,-1) {$\lambda$};
\node (B) at (1,-1) {$\lambda$};
\node (C) at (-1,1) {$\lambda$};
\node (D) at (1,1) {1$\lambda$};
\node (E) at (0,0) {$X$};
\draw (A) -- (B)
(C) -- (D)
(B) -- (D)
(A) -- (C);
\end{tikzpicture}\hspace{.75in}
\begin{tikzpicture}[scale=1]
\node (A) at (-1,-1) {$\lambda$};
\node (B) at (1,-1) {$\lambda$};
\node (C) at (-1,1) {$2\lambda$};
\node (D) at (1,1) {2$\lambda$};
\node (E) at (0,0) {$X$};
\node (F) at (0,1.25) {$i+1$};
\node (E) at (-1.25,0) {$i$};
\node (G) at (1.25,0) {$1$};
\draw (A) -- (B)
(C) -- (D)
(B) -- (D)
(A) -- (C);
\end{tikzpicture}\hspace{.75in}
\begin{tikzpicture}[scale=1]
\node (A) at (-1,-1) {$\lambda$};
\node (B) at (1,-1) {$\lambda$};
\node (C) at (-1,1) {$\mu$};
\node (D) at (1,1) {2$\lambda$};
\node (E) at (0,0) {$X$};
\node (G) at (1.25,0) {$1$};
\draw (A) -- (B)
(C) -- (D)
(B) -- (D)
(A) -- (C);
\end{tikzpicture}
\end{center}

\textbf{If the square does not contain an X and $\mu=\lambda$ or $\nu=\lambda$ with no edge label between $\lambda$ and $\nu$:}

\begin{itemize}
 \item[(5)] If $\mu=\lambda$, then set $\gamma = \nu$ and label the top edge with the same label as the bottom edge if one exists. If $\nu=\lambda$, then $\gamma = \mu$ with the right edge
 labeled with the same label as the left edge if such a label exists. 
\end{itemize}

 \begin{center}
\begin{tikzpicture}[scale=1]
\node (A) at (-1,-1) {$\lambda$};
\node (B) at (1,-1) {$\nu$};
\node (C) at (-1,1) {$\lambda$};
\node (D) at (1,1) {$\nu$};
\node (E) at (0,1.25) {i};
\node (F) at (0,-.75){i};
\draw (A) -- (B)
(C) -- (D)
(B) -- (D)
(A) -- (C);
\end{tikzpicture}\hspace{1in}
\begin{tikzpicture}[scale=1]
\node (A) at (-1,-1) {$\lambda$};
\node (B) at (1,-1) {$\lambda$};
\node (C) at (-1,1) {$\mu$};
\node (D) at (1,1) {$\mu$};
\node (E) at (-1.25,0) {$i$};
\node (G) at (1.25,0) {$i$};
\draw (A) -- (B)
(C) -- (D)
(B) -- (D)
(A) -- (C);
\end{tikzpicture}
\end{center}

\textbf{If $\nu=\lambda$ and the bottom edge is labeled 1:}

\begin{itemize}
 \item[(6)] If $\nu=\lambda$ and the bottom edge is labeled 1, then $\gamma=\mu$, the top edge is labeled 1, and the 
 label on the left edge (if one exists) is the same as the label on the right edge.
\end{itemize}

\begin{center}
\begin{tikzpicture}[scale=1]
\node (A) at (-1,-1) {$\lambda$};
\node (B) at (1,-1) {$\lambda$};
\node (C) at (-1,1) {$\mu$};
\node (D) at (1,1) {$\mu$};
\node (E) at (0,-.75) {$1$};
\node (G) at (0,1.25) {$1$};
\node (H) at (-1.25,0) {$j$};
\node (I) at (1.25,0) {$j$};
\draw (A) -- (B)
(C) -- (D)
(B) -- (D)
(A) -- (C);
\end{tikzpicture}
\end{center}

\textbf{If no previous cases apply, then set $\gamma=2\lambda$ and follow the rule below:}

\begin{itemize}
 \item[(7)] If the bottom edge is labled $i$, then label the right edge $i$. If the left edge is labled $j$, then label the top edge $j+1$.
\end{itemize}

\begin{center}
\begin{tikzpicture}[scale=1]
\node (A) at (-1,-1) {$\lambda$};
\node (B) at (1,-1) {$\nu$};
\node (C) at (-1,1) {$\mu$};
\node (D) at (1,1) {$2\lambda$};
\node (E) at (0,-.75) {$i$};
\node (G) at (1.25,0) {$i$};
\draw (A) -- (B)
(C) -- (D)
(B) -- (D)
(A) -- (C);
\end{tikzpicture}\hspace{1in}
\begin{tikzpicture}[scale=1]
\node (A) at (-1,-1) {$\lambda$};
\node (B) at (1,-1) {$\nu$};
\node (C) at (-1,1) {$\mu$};
\node (D) at (1,1) {$2\lambda$};
\node (E) at (-1.25,0) {$j$};
\node (G) at (0,1.25) {$j+1$};
\draw (A) -- (B)
(C) -- (D)
(B) -- (D)
(A) -- (C);
\end{tikzpicture}\hspace{1in}
\begin{tikzpicture}[scale=1]
\node (A) at (-1,-1) {$\lambda$};
\node (B) at (1,-1) {$\nu$};
\node (C) at (-1,1) {$\mu$};
\node (D) at (1,1) {$2\lambda$};
\node (F) at (0,-.75) {$i$};
\node (H) at (1.25,0) {$i$};
\node (E) at (-1.25,0) {$j$};
\node (G) at (0,1.25) {$j+1$};
\draw (A) -- (B)
(C) -- (D)
(B) -- (D)
(A) -- (C);
\end{tikzpicture}
\end{center}

We call the resulting array the KYF growth diagram of $w$. For example, continuing with the word from Example \ref{ex:1334241},
$w=1334241$, we would have the diagram below.

\begin{example} Below is the KYF growth diagram for the word 1334241.
\begin{center}
\begin{tikzpicture}[scale=1]
\node (00) at (0,0) {$\varnothing$};
\node (10) at (1,0) {$\varnothing$};
\node (20) at (2,0) {$\varnothing$};
\node (30) at (3,0) {$\varnothing$};
\node (40) at (4,0) {$\varnothing$};
\node (50) at (5,0) {$\varnothing$};
\node (01) at (0,1) {$\varnothing$};
\node (.5.5) at (.5,.5) {$X$};
\node (11) at (1,1) {$1$};
\node (21) at (2,1) {$1$};
\node (31) at (3,1) {$1$};
\node (41) at (4,1) {$1$};
\node (51) at (5,1) {$1$};
\node (02) at (0,2) {$\varnothing$};
\node (12) at (1,2) {$1$};
\node (22) at (2,2) {$1$};
\node (32) at (3,2) {$1$};
\node (1.52.5) at (1.5,2.5) {$X$};
\node (42) at (4,2) {$1$};
\node (52) at (5,2) {$11$};
\node (03) at (0,3) {$\varnothing$};
\node (13) at (1,3) {$1$};
\node (23) at (2,3) {$11$};
\node (33) at (3,3) {$21$};
\node (43) at (4,3) {$21$};
\node (53) at (5,3) {$21$};
\node (2.52.5) at (2.5,2.5) {$X$};
\node (04) at (0,4) {$\varnothing$};
\node (14) at (1,4) {$1$};
\node (24) at (2,4) {$11$};
\node (34) at (3,4) {$21$};
\node (44) at (4,4) {$121$};
\node (54) at (5,4) {$221$};
\node (3.53.5) at (3.5,3.5) {$X$};
\node (4.51.5) at (4.5,1.5) {$X$};
\node (5,53.5) at (5.5,3.5) {$X$};
\node (60) at (6,0) {$\varnothing$};
\node (61) at (6,1) {$1$};
\node (62) at (6,2) {$11$};
\node (63) at (6,3) {$21$};
\node (64) at (6,4) {$221$};
\node (70) at (7,0) {$\varnothing$};
\node (71) at (7,1) {$1$};
\node (72) at (7,2) {$11$};
\node (73) at (7,3) {$21$};
\node (74) at (7,4) {$221$};
\node (6.5.5) at (6.5,.5) {$X$};
\node (6.51.15) at (6.5,1.15) {\tiny 1};
\node (6.52.15) at (6.5,2.15) {\tiny 1};
\node (6.53.15) at (6.5,3.15) {\tiny 1};
\node (6.54.15) at (6.5,4.15) {\tiny 1};
\node (6.93.5) at (6.9,3.5) {\tiny 1};
\node (2.92.5) at (2.9,2.5) {\tiny 1};
\node (3.92.5) at (3.9,2.5) {\tiny 1};
\node (5.93.5) at (5.9,3.5) {\tiny 1};
\node (4.93.5) at (4.9,3.5) {\tiny 2};
\node (4.53.1) at (4.5,3.15) {\tiny 2};
\node (5.54.1) at (5.5,4.15) {\tiny 3};

\draw (00)--(10)--(20)--(30)--(40)--(50)--(60)--(70)
(01)--(11)--(21)--(31)--(41)--(51)--(61)--(71)
(02)--(12)--(22)--(32)--(42)--(52)--(62)--(72)
(03)--(13)--(23)--(33)--(43)--(53)--(63)--(73)
(04)--(14)--(24)--(34)--(44)--(54)--(64)--(74)
(00)--(01)--(02)--(03)--(04)
(10)--(11)--(12)--(13)--(14)
(20)--(21)--(22)--(23)--(24)
(30)--(31)--(32)--(33)--(34)
(40)--(41)--(42)--(43)--(44)
(50)--(51)--(52)--(53)--(54)
(60)--(61)--(62)--(63)--(64)
(70)--(71)--(72)--(73)--(74);
\end{tikzpicture}
\end{center}
\end{example}

Let $\mu_0 = \varnothing \subseteq \mu_1 \subseteq \ldots \mu_k$ be the seqence of snakeshapes across the top of the 
growth diagram, and let $\nu_0=\varnothing \subseteq \nu_1 \subseteq \ldots \nu_n$ be the sequence of snakeshapes 
on the right side of the KYF growth diagram. These sequences correspond to KYF tableaux
$Q(w)$ and $P(w)$, respectively. 

If the edge between $\nu_i$ and $\nu_{i+1}$ is labeled $j$, then $2\nu_i=\nu_{i+1}$, and a column of $i$'s is added in the $j^{th}$ column of $\nu_i$ to obtain $\nu_{i+1}$. 
If the edge between $\mu_i$ and $\mu_{i+1}$ is labeled $j$, then $\mu_i=\mu_{i+1}$ and the label $i+1$ of $Q(w_1\cdots w_{i+1})$ is placed in the box at the top of the 
the $j^{th}$ column of $Q(w_1\cdots w_i)$.

In the example above, we have 

\begin{center}
 $P=$
 \begin{ytableau}
 4 & 3 & \none \\
 4 & 2 & 1
 \end{ytableau}\hspace{1in}
 $Q=$
 \begin{ytableau}
  57 & 3 & \none \\
  4 & 2 & 16
 \end{ytableau}\hspace{.1in},
\end{center}
which agrees with the insertion and recording tableau from Example \ref{ex:1334241}. As for Hecke insertion and shifted Hecke insertion, this is not a coincidence. 

\begin{theorem}
For any word $w$, the KYF tableau and set-valued KYF tableau $P$ and $Q$ obtained from the sequence of snakeshapes across the right side of the KYF growth diagram for $w$
and across the top of the KYF growth diagram for $w$, respectively, are $P_{YF}(w)$ and $Q_{YF}(w)$, the KYF insertion and recording tableau for $w$.
\end{theorem}

\begin{proof}
Suppose the square under consideration is in row $t$ and column $s$. Assume also the $X$ in column $s$ is in row $r$. We may or may not have $r=t$. 
We will argue the result by induction on $t+s$. The base case is trivial. As in the previous growth diagram proofs, we have the following lemmas.

\begin{lemma} Oriented as in the square above, $|\nu/\lambda|\leq 1$.\end{lemma}

\begin{lemma} Oriented as in the square above, to obtain $\mu$ from $\lambda$, we add either one box of a column of 2 boxes.\end{lemma}

Rule 1 follows from the fact that inserting $t$ into the tableau \begin{ytableau} t \end{ytableau} does not change the tableau. Rules 2 applies when inserting positive integer $t$ that is strictly 
larger than all other entries of the KYF tableau. Thus a box with $t$ will be added directly to the left of the tableau. In Rule 3, we are adding integer $t$ to a tableau that already has 2 
boxes labeled $t$ in the $i^{th}$ column, and $t$ is the largest integer in the tableau. Following the insertion rules, we end with a column of $t$'s in the first column, hence the right edge is labeled 1. 
The special corner of this insertion is in the $(i+1)^{th}$ column, the column where the $t$ on the top row of the original KYF tableau was deleted in the insertion procedure. 

In Rule 4, there is exactly one $t$ in insertion tableau $\mu$. Thus, inserting a $t$ will result in a tableau where the first column has 
two boxes filled with a $t$.

If $\lambda=\nu$ as in Rule 5, then $r>t$. Thus no integer is being inserted as we move from $\mu$ to $\gamma$, so $\mu=\gamma$ and the edge labeles are unchanged. 
Similarly, if if $\lambda=\mu$, then there is no $X$ to the left of the square we are considering, and so no boxes labeled $t$ in the KYF tableau $\gamma$. Thus there is no change in the insertion and recording tableaux.

Rule 6 follows directly from insertion rule YF1.

In considering Rule 7, note that $r<t$ because if not, 
we would be in the situation described 
in Rule 5. In Rule 7, $\gamma=2\lambda$ because no matter what $r$ is inserted, since $\mu\neq\lambda$, there is at least one entry of $\mu$ that is larger than $r$. This means that when we insert $r$ into insertion
tableau $\mu$, a box with $r$ will be added to the first column, and after the shifting of YF5, there will be a box directly above this new box with $r$. An entry 
$t$ will shift into the position of any box that was ``lost'' when shifting during the insertion of $r$ into $\lambda$, and if there were two boxes labeled $t$ in $\mu$, the second will be deleted 
in YF6. It follows that $\gamma=2\lambda$.

If the bottom edge is labeled $i$, then insertion of $r$ into insertion tableau $\lambda$ included the situation described by YF6 in column $i$ of $\lambda$. When we insert 
$r$ into insertion tableau $\mu$, column $i$ now becomes \begin{ytableau} t \\ $*$ \end{ytableau} and thus \begin{ytableau} t \\ t \end{ytableau} in $\gamma$. This means that the edge between 
$\nu$ and $\gamma$ should be labeled $i$, indicating where the $t$'s are in insertion tableau $\gamma$.

If the left edge is labeled $j$, then the difference between insertion tableaux $\lambda$ and $\mu$ is a column \begin{ytableau} t \\ t \end{ytableau} in column $j$ of $\mu$. 
The process of inserting $r$ into insertion tableau $\mu$ will match that of inserting $r$ into insertion tableau $\lambda$ until we get to the point where we are shifting the entries $t$ and $*$. Since $t$ is 
the largest entry in $\mu$, the insertion of $r$ into $\mu$ will end with deleting the unmatched $*$ from column $j+1$ of the insertion tableau $\gamma$, and therefore the top edge is labeled $j+1$.
\end{proof}

We can also formulate the rules for the reverse insertion as follows. The proof is omitted for brevity.

\textbf{Reverse Rules}

\begin{enumerate}
 \item[(R1)] If $\mu=\gamma=1$, $\nu=\varnothing$, at the top edge is labeled 1, then the square contains an $X$ and $\lambda=\varnothing$.
 \item[(R2)] If $\mu=\nu$ and $\gamma=1\mu=1\nu$, then the square contains an $X$ and $\lambda=\mu=\nu$.
 \item[(R3)] $\mu=\gamma=2\nu$, the top edge is labeled $i$ and the right edge is labeled 1, then there is an $X$ in the square, $\lambda=\nu$ and the left edge 
 is labeled $i-1$.
 \item[(R4)] If $\gamma=2\nu$ and the right edge is labeled 1, then the square contains an $X$ and $\lambda=\nu$.
 \item[(R5)] If $\gamma=\nu$, then set $\lambda=\mu$ and label the bottom edge with the same label as the top edge if one exists. If $\gamma=\mu$, 
 then $\lambda=\nu$ and the left edge has the same label as the right edge if such a label exists.
 \item[(R6)] If $\gamma=\mu$ and the top edge is labeled 1, then $\lambda=\nu$, the bottom edge is labeled 1, and the left edge has the same label 
 as the right edge.
 \item[(R7)] If no previous cases apply, then $\lambda$ is obtained from $\gamma$ by removing the first 2. If the right edge is labeled $i$, then the 
 bottom edge is labeled $i$. If the top edge is labeled $j$, then the left edge is labeled $j-1$.
\end{enumerate}

\section{Other examples} \label{sec:other}

\subsection{Binary tree deformations} \label{sec:bin}

\begin{example}\label{ex:BinWord}
 The lifted binary tree is shown on the left, where vertices can be naturally labeled by bit strings: 0, 1, 10, 11, 100, 101,$\ldots$. The graph \textit{BinWord} with 
 the same set of vertices and rank function is shown on the right. In \textit{BinWord}, an element $x$ covers $y$ if $y$ can be obtained from $x$ by deleting a single
 digit from $x$, and in addition, 1 covers 0. The lifted binary tree and \textit{BinWord} form a dual graded graph, see \cite[Example 2.4.1]{F}.
 
 \begin{minipage}{.5\textwidth}
\begin{center}
\begin{tikzpicture}[scale=1]
\node (0) at (0,0) {\small{0}};
\node (1) at (0,1) {\small{1}};
\node (11) at (2,2) {\small{11}};
\node (10) at (-2,2) {\small{10}};
\node (100) at (-3,3) {\small{100}};
\node (101) at (-1,3) {\small{101}};
\node (110) at (1,3) {\small{110}};
\node (111) at (3,3) {\small{111}};
\node (1000) at (-3.5,4) {\small{1000}};
\node (1001) at (-2.5,4) {\small{1001}};
\node (1010) at (-1.5,4) {\small{1010}};
\node (1011) at (-.5,4) {\small{1011}};
\node (1100) at (.5,4) {\small{1100}};
\node (1101) at (1.5,4) {\small{1101}};
\node (1110) at (2.5,4) {\small{1110}};
\node (1111) at (3.5,4) {\small{1111}};
\draw (0)--(1)--(11)--(111)--(1111) (1)--(10)--(100)--(1000) (100)--(1001) (10)--(101)--(1010) (101)--(1011)
(11)--(110)--(1100) (110)--(1101) (111)--(1110);
\end{tikzpicture}
\end{center}
\end{minipage}
  \begin{minipage}{.5\textwidth}
\begin{center}
\begin{tikzpicture}[scale=1]
\node (0) at (0,0) {\small{0}};
\node (1) at (0,1) {\small{1}};
\node (11) at (2,2) {\small{11}};
\node (10) at (-2,2) {\small{10}};
\node (100) at (-3,3) {\small{100}};
\node (101) at (-1,3) {\small{101}};
\node (110) at (1,3) {\small{110}};
\node (111) at (3,3) {\small{111}};
\node (1000) at (-3.5,4) {\small{1000}};
\node (1001) at (-2.5,4) {\small{1001}};
\node (1010) at (-1.5,4) {\small{1010}};
\node (1011) at (-.5,4) {\small{1011}};
\node (1100) at (.5,4) {\small{1100}};
\node (1101) at (1.5,4) {\small{1101}};
\node (1110) at (2.5,4) {\small{1110}};
\node (1111) at (3.5,4) {\small{1111}};
\draw (0)--(1)--(11)--(111)--(1111) (1)--(10)--(100)--(1000) (100)--(1001) (10)--(101)--(1010) (101)--(1011)
(11)--(110)--(1100) (110)--(1101) (111)--(1110) (101)--(11) (110)--(10) (100)--(1010) (100)--(1100) (101)--(1001) (101)--(1101) (110)--(1110) (110)--(1010) (111)--(1101) (111)--(1011);
\end{tikzpicture}
\end{center}
\end{minipage}
 
We form a Pieri deformation by interpreting the graph in the context of the ring of quasisymmetric functions, $\QSym$ (see \cite{EC2}). If we interpret ``1'' as $L_1$, 
``10'' as $L_2$, ``100'' as $L_3$, 11001 as $L_{131}$, and so on, we see that $x$ covers $y$ in the graph on the right exactly when $x$ appears in the product
$L_y\cdot L_1$, where $L_\alpha$ is the fundamental quasisymmetric function with the usual product. Let $D$ be the operator that subtracts 1 from the rightmost number in the 
subscript or deletes the rightmost number if it is 1. For example, $D(L_{14})=L_{13}$ and $D(L_{1421})=L_{142}$. Then $x$ covers $y$ in the 
graph on the left exactly when $y=D(x)$. To form the Pieri deformation, we can let $f=L_1+L_{11}+L_{111}+\ldots$. In other words, the multiplicity of the 
edge from $x$ down to $y$ is the coefficient of $x$ in $y\cdot (L_1+L_{11}+\ldots)$. 

 \begin{minipage}{.5\textwidth}
\begin{center}
\begin{tikzpicture}[scale=1]
\node (0) at (0,0) {\small{0}};
\node (1) at (0,1) {\small{1}};
\node (11) at (2,2) {\small{11}};
\node (10) at (-2,2) {\small{10}};
\node (100) at (-3,3) {\small{100}};
\node (101) at (-1,3) {\small{101}};
\node (110) at (1,3) {\small{110}};
\node (111) at (3,3) {\small{111}};

\draw (0)--(1)--(11)--(111) (1)--(10)--(100)  (10)--(101) (101)
(11)--(110);
\end{tikzpicture}
\end{center}
\end{minipage}
  \begin{minipage}{.5\textwidth}
\begin{center}
\begin{tikzpicture}[scale=1]
\node (0) at (0,0) {\small{0}};
\node (1) at (0,1) {\small{1}};
\node (11) at (2,2) {\small{11}};
\node (10) at (-2,2) {\small{10}};
\node (100) at (-3,3) {\small{100}};
\node (101) at (-1,3) {\small{101}};
\node (110) at (1,3) {\small{110}};
\node (111) at (3,3) {\small{111}};

\draw (0)--(1)--(11)--(111) (1)--(10)--(100)  (10)--(101)
(11)--(110)  (101)--(11) (110)--(10) 
(101)--(1) (110)--(1) (111) to [bend right=10] (1) (0)--(11) (0) to [bend right=10] (111);
\end{tikzpicture}
\end{center}
\end{minipage}

 Using $A=\QSym$ and $D$ and $f$ as defined above, it is easy to see that $D(f)=L_\varnothing+f=id+f$. Using the Hopf algebra structure of $\QSym$, it is clear that
 $D$ is a derivation by Lemma \ref{lem:derivation}. Thus the resulting 
 graph is a dual filtered graph by Theorem \ref{thm:pieri}.

\end{example}

\begin{example}\label{ex:L}
 Consider the Hopf algebra of multi-quasisymmetric functions, $\mQSym$, defined in \cite{LP}. This Hopf algebra has a 
 basis of multi-fundamental quasisymmetric functions indexed by compositions, $\{\tilde L_\alpha\}$. 
 Notice that a graph with vertices indexed by $\{\tilde L_\alpha\}$
 has the same vertices as in Example \ref{ex:BinWord}. To define the product, we first define the concept of a multiword. 
 Let $u=u_1u_2\cdots u_k$ be a word. We call $w=w_1w_2\cdots w_m$ a \textit{multiword} of $u$ if
 there is a surjective and non-decreasing map $t:[m]\to[k]$ so that $w_j=u_{t(j)}$. For example, 11335662 and 133335562 are 
 multiwords of 13562. Let $u=u_1u_2\cdots u_k$ and $v=v_1v_2\cdots v_l$ be two words,
 and assume that all letters $v_i$ and $u_j$ are distinct. A word $w=w_1w_2\cdots w_m$ is a \textit{multishuffle} of $u$ and $v$ if 
 \begin{itemize}
  \item[(1)] for any $i\in [1,m-1]$, $w_i\neq w_{i+1}$, and
  \item[(2)] when restricted to alphabets $v_i$ and $u_j$, $w$ becomes a multiword of $v$ and $u$, respectively.
 \end{itemize}
For example, 1818373567627 is a multishuffle of 13562 and 87. If $u$ or $v$ contain repeated letters, obtain the 
multishuffle by first replacing $u$ and $v$ with words with distinct letters, taking the multishuffle
product, and then replacing the letters in the result with the original letters. For example, if multishuffling 
121 with 1, we can instead multishuffle 132 with 4 to get $1324+13242+41342+\ldots$, and then 
substitute the original letters to get $1214+12111+11211+\ldots$. Note that multiplicities may occur in this case.
 
To multiply $\tilde L_\alpha$ and $\tilde L_\beta$, first choose words $w_\alpha$ and $w_\beta$ such that the composition of the 
descent set of $u$ is $\alpha$ and of $v$ is $\beta$. For example, if $\alpha=(2,1)$ and $\beta$ is $(1)$, we can choose $w_\alpha=231$ and $w_\beta=1$. Say that $w_\alpha$ is on $[n]$ for some $n\in\mathbb{N}$, and let $w_\beta[n]$ be the word obtained 
by adding $n$ to each letter of $w_\beta$. In the previous example, $w_\beta[3]=4$. Now 
$$\tilde L_\alpha \tilde L_\beta=\displaystyle\sum _w \tilde L_{\mathcal{C}(w)},$$ where we sum over multishuffles of $w_\alpha$ and 
$w_\beta[n]$ and $\mathcal{C}(w)$ is the composition associated to the descent set of word $w$. 
For example, 
$\tilde L_{(2,1)}\tilde L_{(1)}=\tilde L_{(2,2)}+\tilde L_{(2,1,1)}+\tilde L_{(3,1)}+\tilde L_{(12121)}+\ldots$, 
where the terms shown correspond to multishuffles 2314, 2431, 2341, 4243141.

Define the coproduct using the \textit{cuut coproduct} of a word $w$: $\blacktriangle(w_1w_2\cdots w_k)=\varnothing\otimes w_1\cdots w_k+
w_1\otimes w_1\cdots w_k+w_1\otimes w_2\cdots w_k+w_1w_2\otimes w_2\cdots w_k+\ldots+w_1\cdots w_k\otimes w_k+w_1\cdots w_k\otimes \varnothing.$ Then let 
$$\Delta(\tilde L_\alpha)=\displaystyle\sum_{u\otimes u' \text{ in } cuut(w_\alpha)}\tilde L_{\mathcal{C}(u)}\otimes \tilde L_{\mathcal{C}(u')}.$$ For example, $$\Delta(\tilde L_{(2,1)})=\varnothing\otimes \tilde L_{(2,1)}+\tilde L_{(1)}\otimes \tilde L_{(2,1)}+\tilde L_{(1)}\otimes \tilde L_{(1,1)}+\tilde L_{(2)}\otimes \tilde L_{(1,1)}+\tilde L_{(2)}\otimes\tilde L_{(1)}+\tilde L_{(2,1)}\otimes \tilde L_{(1)}+\tilde L_{(2,1)}\otimes \varnothing$$ since $$\blacktriangle(231)=
\varnothing\otimes 231+2\otimes 231+2\otimes 31+23\otimes 31 +23\otimes 1+231\otimes 1+231\otimes\varnothing.$$

 Taking $A$ to be $\mQSym$ with the basis $\{\tilde L_\alpha \}$ indexed by
 compositions, we create a dual filtered graph by taking $f=\tilde L_{(1)}$ and $D=\xi\circ\Delta$ with $g=\tilde L_{(1)}$.
 The first five ranks are shown below.
 
 \begin{minipage}{.5\textwidth}
\begin{center}
\begin{tikzpicture}[scale=.9]
\node (0) at (0,0) {\small{0}};
\node (1) at (0,1) {\small{1}};
\node (11) at (2,2) {\small{11}};
\node (10) at (-2,2) {\small{2}};
\node (100) at (-3,3) {\small{3}};
\node (101) at (-1,3) {\small{21}};
\node (110) at (1,3) {\small{12}};
\node (111) at (3,3) {\small{111}};
\node (1000) at (-3.5,4) {\small{4}};
\node (1001) at (-2.5,4) {\small{31}};
\node (1010) at (-1.5,4) {\small{22}};
\node (1011) at (-.5,4) {\small{211}};
\node (1100) at (.5,4) {\small{13}};
\node (1101) at (1.5,4) {\small{121}};
\node (1110) at (2.5,4) {\small{112}};
\node (1111) at (3.5,4) {\small{1111}};
\draw (1) to [out=60,in=100,distance=.5cm] (1)
(1) to [out=60,in=100,distance=.5cm] (1)
(11) to [out=60,in=100,distance=.5cm] (11)
(10) to [out=60,in=100,distance=.5cm] (10)
(100) to [out=-60,in=-100,distance=.5cm] (100)
(101) to [out=-60,in=-100,distance=.5cm] (101)
(110) to [out=-60,in=-100,distance=.5cm] (110)
(111) to [out=-60,in=-100,distance=.5cm] (111)
(1000) to [out=60,in=100,distance=.5cm] (1000)
(1001) to [out=60,in=100,distance=.5cm] (1001)
(1010) to [out=60,in=100,distance=.5cm] (1010)
(1011) to [out=60,in=100,distance=.5cm] (1011)
(1100) to [out=60,in=100,distance=.5cm] (1100)
(1101) to [out=60,in=100,distance=.5cm] (1101)
(1110) to [out=60,in=100,distance=.5cm] (1110)
(1111) to [out=60,in=100,distance=.5cm] (1111)
(0)--(1)--(11)--(111)--(1111) (1)--(10)--(100)--(1000) (100)--(1001) (10)--(101)--(1010) (101)--(1011)
(11)--(110)--(1100) (110)--(1101) (111)--(1110);
\end{tikzpicture}
\end{center}
\end{minipage}
  \begin{minipage}{.5\textwidth}
\begin{center}
\begin{tikzpicture}[scale=.9]
\node (0) at (0,0) {\small{0}};
\node (1) at (0,1) {\small{1}};
\node (11) at (2,2) {\small{11}};
\node (10) at (-2,2) {\small{2}};
\node (100) at (-3,3) {\small{3}};
\node (101) at (-1,3) {\small{21}};
\node (110) at (1,3) {\small{12}};
\node (111) at (3,3) {\small{111}};
\node (1000) at (-3.5,4) {\small{4}};
\node (1001) at (-2.5,4) {\small{31}};
\node (1010) at (-1.5,4) {\small{22}};
\node (1011) at (-.5,4) {\small{211}};
\node (1100) at (.5,4) {\small{13}};
\node (1101) at (1.5,4) {\small{121}};
\node (1110) at (2.5,4) {\small{112}};
\node (1111) at (3.5,4) {\small{1111}};
\draw (0)--(1)--(11)--(111)--(1111) (1)--(10)--(100)--(1000) (100)--(1001) (10)--(101)--(1010) (101)--(1011)
(11)--(110)--(1100) (110)--(1101) (111)--(1110) (101)--(11) (110)--(10) (100)--(1010) (100)--(1100) (101)--(1001) (101)--(1101) (110)--(1110) (110)--(1010) (111)--(1101) (111)--(1011)
(101)--(1) (110)--(1) (1001)--(10) (1010) to [bend right=5] (10) (1010) to [bend left=5] (10) (1010) to [bend left =15] (1) (1010)--(11) (1011) to [bend right =5] (11) (1100) to [bend left = 5](10)
(1101)--(10) (1101) to [bend right =15] (1) (1101) to [bend left=5] (11) (1101) to [bend right =5] (11) (1110)--(11);
\end{tikzpicture}
\end{center}
\end{minipage}

This is a dual filtered graph by Theorem \ref{thm:pieri}.

The following proposition shows that we have another instance of M\"obius via Pieri phenomenon, just like in the case of Young's lattice.

\begin{proposition}
 The Pieri construction above is a M\"obius deformation of the dual graded graph in Example \ref{ex:BinWord}.
\end{proposition}
\begin{proof}
We can transform the poset {\it{BinWord}} from Example \ref{ex:BinWord} into the poset of subwords shown below by ignoring the first 1 in every word. In doing this, the element $\varnothing$ in \textit{BinWord} disappears, $1$ becomes $\varnothing$ in the subword poset, $10$ becomes $0$, $11001$ becomes $1001$ and so on. With this change, $x<y$ in \textit{BinWord} exactly when $x$ can be found as a subword of $y$, ignoring the inital 1's in each.

\begin{center}
\begin{tikzpicture}[scale=1]
\node (1) at (0,1) {\small{$\varnothing$}};
\node (11) at (2,2) {\small{1}};
\node (10) at (-2,2) {\small{0}};
\node (100) at (-3,3) {\small{00}};
\node (101) at (-1,3) {\small{01}};
\node (110) at (1,3) {\small{10}};
\node (111) at (3,3) {\small{11}};
\node (1000) at (-3.5,4) {\small{000}};
\node (1001) at (-2.5,4) {\small{001}};
\node (1010) at (-1.5,4) {\small{010}};
\node (1011) at (-.5,4) {\small{011}};
\node (1100) at (.5,4) {\small{100}};
\node (1101) at (1.5,4) {\small{101}};
\node (1110) at (2.5,4) {\small{110}};
\node (1111) at (3.5,4) {\small{111}};
\draw (1)--(11)--(111)--(1111) (1)--(10)--(100)--(1000) (100)--(1001) (10)--(101)--(1010) (101)--(1011)
(11)--(110)--(1100) (110)--(1101) (111)--(1110) (101)--(11) (110)--(10) (100)--(1010) (100)--(1100) (101)--(1001) (101)--(1101) (110)--(1110) (110)--(1010) (111)--(1101) (111)--(1011);
\end{tikzpicture}
\end{center}

Given a word $y=y_1y_2\cdots y_n$, define its \textit{repetition set} $\mathcal{R}(y)=\{i:y_{i-1}=y_i\}.$ An embedding of $x$ in $y$ is a sequence $1\leq i_1<i_2<\ldots <i_k\leq n$ such that $x=y_{i_1}y_{i_2}\cdots y_{i_k}$. It is a \textit{normal embedding} if $\mathcal{R}(y)\subseteq \{i_1,i_2,\ldots, i_k\}$. For two words, $x$ and $y$, let $\binom{y}{x}_n$ denote the number of normal embeddings of $x$ in $y$. For example, $\binom{10110}{10}_n=1$ and $\binom{1010}{10}_n=3$. Now, from Theorem 1 in \cite{Bj}, $$\mu(x,y)=(-1)^{|y|-|x|}\binom{y}{x}_n.$$

We show that the coefficient of $\tilde L_\beta$ in $\tilde L_\alpha \tilde L_{(1)}$ is the number of normal embeddings of $\alpha$ into $\beta$ after viewing $\alpha$ and $\beta$ as sequences of 0's and 1's and deleting the first 1 in each by showing the unique way to insert strings of alternating 1's and 0's into $\alpha$ at a specific location to obtain $\beta$ in the multishuffle. So, for example, the coefficient of $\tilde L_{(1111)}$ in $\tilde L_{(11)}\tilde L_{(1)}$ is zero since $\binom{111}{1}_n=0$ and the coefficient of $\tilde L_{(122)}$ in $\tilde L_{(12)} \tilde L_{(1)}$ is $3$ since $\binom{1010}{10}_n=3$.

Choose word $w_\alpha$ on $[n]$ and multishuffle it with $n+1$.
We note that adding two 0's to $\alpha$ that are not separated by a 1 in $\alpha$ in the multishuffle means increasing the length of an increasing segment of $w_\alpha$ by two, which is impossible as adding $n+1$ to the end of an increasing segment is the only way to increase its length, and the letters of $w_\alpha$ must stay in the same relative order in the multishuffle. Similarly for adding two 1's to $w_\alpha$ that are not separated by a 0 in $\alpha$. This agrees with the M\"obius function since there are no normal embeddings in this case, e.g. $\binom{1010001}{101}_n=0$.

Suppose we want to insert a segment of alternating 1's and 0's in $\alpha$ in a place with a 0 (if anything) on the left and a 1 (if anything) on the right. Let $w_\alpha$ be $w_1w_2\cdots w_{k-1}w_k w_{k+1}\cdots w_n$, where the $w_k$ and $w_{k+1}$ correspond to the 1 and 0 of $\alpha$ mentioned above. The only way to add an alternating sequence beginning with 0 is to insert the string $(n+1)w_k (n+1) w_k (n+1)\ldots$ between $w_k$ and $w_{k+1}$. The resulting word, $$w_1\cdots w_{k-1}w_k(n+1)w_k(n+1)w_k\cdots w_{k+1}
\cdots w_n$$ is clearly a multishuffle of $w_\alpha$ and $n+1$. Similarly, to add an alternating sequence beginning with 1, ``delete" the $w_k$ from $w$ and replace it with $(n+1)w_k(n+1)w_k(n+1)\cdots$. The resulting word, $$w_1\cdots w_{k-1}(n+1)w_k(n+1)w_k\cdots w_{k+1}\cdots w_n$$ is a multishuffle of $w_\alpha$ and $n+1$.

If instead we insert the alternating string in a place of $\alpha$ with a 1 on the left and a 0 on the right, we only need to consider strings beginning with 0 as the 1 at the beginning of the string can be added as in the previous case. To insert an alternating string beginning with 0, we must insert the segment $(n+1)w_k(n+1)w_k\cdots$ between $w_k$ and $w_{k+1}$ of $w_\alpha$. 

If we insert the alternating string in place of $\alpha$ with 0's to the left and right, we only need to consider strings that begin with 1. To do this, delete the $w_k$ from $w$ and insert the segment $(n+1)w_k(n+1)w_k\cdots$ between $w_{k-1}$ and $w_{k+1}$

If we insert the alternating string in place of $\alpha$ with 1's to the left and right, we only need to consider strings that begin with 0. To do this, insert the segment $(n+1)w_k(n+1)w_k\cdots$ between $w_k$ and $w_{k+1}$.
\end{proof}
\end{example}

\subsection{Poirer-Reutenauer deformations} \label{sec:PR}

\begin{example}\label{ex:SYTtree}
The \textit{SYT-tree} has as vertices standard Young tableau, and standard Young tableau $T_1$ covers standard Young tableau $T_2$ if $T_2$ is obtained from 
$T_1$ by deleting the box with largest entry. It is shown below on the left. It is dual to the \textit{Schensted graph}, which is shown on the right. 

The \textit{Schensted graph} is constructed using RSK insertion. A standard Young tableau $T_1$ covers $T_2$ if $T_1$ appears in the 
product of $T_2$ with the single-box standard tableau, where multiplication of standard Young tableaux is as follows. Suppose $T_1$ and $T_2$ are standard Young tableaux
with row words $\r(T_1)$ on $[n]$ and $\r(T_2)$ on $[m]$, respectively. Then let $T_1\cdot T_2$ be the sum of tableaux obtained by 
first lifting $\r(T_1)$ to $[n+m]$ in all ways that preserve relative order and then RSK inserting the corresponding lifting of $\r(T_2)$.
For example, in multiplying \begin{ytableau} 1 & 2 \end{ytableau} by itself, we would get the sum of $(P(12)\leftarrow 3) \leftarrow 4$, 
$(P(13)\leftarrow 2 )\leftarrow 4$, $P((14)\leftarrow 2)\leftarrow 3$, $(P(23)\leftarrow 1)\leftarrow 4$, $(P(24)\leftarrow 1)\leftarrow 3$, 
and $(P(34)\leftarrow 1)\leftarrow 2$.

 \begin{minipage}{.5\textwidth}
\begin{center}
\begin{tikzpicture}[scale=1]
\node (empty) at (0,0) {$\varnothing$};
\node (1) at (0,2) {\small{\begin{ytableau}1\end{ytableau}}};
\node (21) at (2,4) {\small{\begin{ytableau} 1 \\ 2\end{ytableau}}};
\node (12) at (-2,4) {\small{\begin{ytableau}1 & 2\end{ytableau}}};
\node (123) at (-3,6) {\small{\begin{ytableau}1 & 2 & 3\end{ytableau}}};
\node (312) at (-1,6) {\small{\begin{ytableau}1 & 2 \\ 3\end{ytableau}}};
\node (213) at (1,6) {\small{\begin{ytableau}1 & 3 \\2\end{ytableau}}};
\node (321) at (3,6) {\small{\begin{ytableau}1 \\ 2 \\3\end{ytableau}}};
\draw (empty)--(1)--(12)--(123) (12)--(312) (1)--(21)--(321) (21)--(213);
\end{tikzpicture}
\end{center}
\end{minipage}
 \begin{minipage}{.5\textwidth}
\begin{center}
\begin{tikzpicture}[scale=1]
\node (empty) at (0,0) {$\varnothing$};
\node (1) at (0,2) {\small{\begin{ytableau}1\end{ytableau}}};
\node (21) at (2,4) {\small{\begin{ytableau} 1 \\ 2\end{ytableau}}};
\node (12) at (-2,4) {\small{\begin{ytableau}1 & 2\end{ytableau}}};
\node (123) at (-3,6) {\small{\begin{ytableau}1 & 2 & 3\end{ytableau}}};
\node (312) at (-1,6) {\small{\begin{ytableau}1 & 2 \\ 3\end{ytableau}}};
\node (213) at (1,6) {\small{\begin{ytableau}1 & 3 \\2\end{ytableau}}};
\node (321) at (3,6) {\small{\begin{ytableau}1 \\ 2 \\3\end{ytableau}}};
\draw (empty)--(1)--(12)--(123) (12)--(312) (1)--(21)--(321) (21)--(213) (312)--(21) (213)--(12);
\end{tikzpicture}
\end{center}
\end{minipage}

Let $f=T_1+T_2+T_3+\ldots$, where $T_i$ is the one-row tableau 
with row reading word $123\ldots i$, and create a Pieri deformation by adding edges to the \textit{Schensted graph} so that $a_2(T_i,T_j)$ 
is the coefficient of $T_j$ in $T_i\cdot f$. 

 \begin{minipage}{.5\textwidth}
\begin{center}
\begin{tikzpicture}[scale=1]
\node (empty) at (0,0) {$\varnothing$};
\node (1) at (0,2) {\small{\begin{ytableau}1\end{ytableau}}};
\node (21) at (2,4) {\small{\begin{ytableau} 1 \\ 2\end{ytableau}}};
\node (12) at (-2,4) {\small{\begin{ytableau}1 & 2\end{ytableau}}};
\node (123) at (-3,6) {\small{\begin{ytableau}1 & 2 & 3\end{ytableau}}};
\node (312) at (-1,6) {\small{\begin{ytableau}1 & 2 \\ 3\end{ytableau}}};
\node (213) at (1,6) {\small{\begin{ytableau}1 & 3 \\2\end{ytableau}}};
\node (321) at (3,6) {\small{\begin{ytableau}1 \\ 2 \\3\end{ytableau}}};
\draw (empty)--(1)--(12)--(123) (12)--(312) (1)--(21)--(321) (21)--(213);
\end{tikzpicture}
\end{center}
\end{minipage}
 \begin{minipage}{.5\textwidth}
\begin{center}
\begin{tikzpicture}[scale=1]
\node (empty) at (0,0) {$\varnothing$};
\node (1) at (0,2) {\small{\begin{ytableau}1\end{ytableau}}};
\node (21) at (2,4) {\small{\begin{ytableau} 1 \\ 2\end{ytableau}}};
\node (12) at (-2,4) {\small{\begin{ytableau}1 & 2\end{ytableau}}};
\node (123) at (-3,6) {\small{\begin{ytableau}1 & 2 & 3\end{ytableau}}};
\node (312) at (-1,6) {\small{\begin{ytableau}1 & 2 \\ 3\end{ytableau}}};
\node (213) at (1,6) {\small{\begin{ytableau}1 & 3 \\2\end{ytableau}}};
\node (321) at (3,6) {\small{\begin{ytableau}1 \\ 2 \\3\end{ytableau}}};
\draw (empty)--(1)--(12)--(123) (12)--(312) (1)--(21)--(321) (21)--(213) (312)--(21) (213)--(12) (12)--(empty) (123) to [bend right =25] (empty)
(123) to [bend left=10] (1) (312)--(1) (213)--(1);
\end{tikzpicture}
\end{center}
\end{minipage}

To see that this gives a dual filtered graph, take $A$ to be the algebra of standard Young tableaux with multiplication as defined above. Note that 
$A$ is dual to the Poirier-Reutenauer Hopf algebra \cite{PR}. Take $D$ to be the operator that deletes the box with the largest entry, and notice that $T_1$ covers $T_2$ in \textit{SYT-tree} exactly when $D(T_1)=T_2$. 
The operator $D$ is a derivation by Lemma \ref{lem:derivation} using the fact that $\Delta(T)=\sum (u\otimes v)$, where we sum over words $u$ and $v$ 
such that $\r(T)$ is a shuffle of $u$ and $v$. Lastly, it is easy to see that $D(f)=\varnothing + f$. Thus by Theorem \ref{thm:pieri}, we have 
a dual filtered graph.
 
\end{example}

\begin{example}\label{ex:KPR}
Take $A$ to be the $K$-theoretic Poierier-Reuntenauer bialgebra (KPR) defined in \cite{PP} with a basis of $K$-Knuth classes of initial words (i.e. words containing only $1,2,\ldots,n$ for some $n$) on $\mathbb{N}$. The $K$-Knuth relations are as follows:

\begin{center}
$pp\equiv p$ \hspace{.9in} for all $p$\\
$pqp\equiv qpq$ \hspace{.8in} for all $q,p$\\
$pqs\equiv qps$ and $spq\equiv sqp$ \hspace{.1in} for $p<s<q$
\end{center}
See \cite{BuchSam} for details on this relation.

Multiplication of two classes $[w_1]$ and $[w_2]$ is done by shuffling each pair of elements, one from $[w_1]$ and one from $[w_2[n]]$, where $w_1$ is a word on $n$ and $w_2[n]$ is obtained by adding $n$ to each letter in $w_2$, and writing the result
as a sum of classes. For example, $$[[12]]\cdot[[312]]=[[53124]]+[[51234]]+[[35124]]+[[351234]]+[[53412]]+[[5351234]].$$ The coproduct is similarly defined by taking the sum of the coproduct, $\Delta(w_1\cdots w_k)=\varnothing\otimes w_1\cdots w_k+std(w_1)\otimes std(w_2\cdots w_k)+\ldots
+std(w_1\cdots w_{k-1}w_k)+ std(w_1\cdots w_k)\otimes \varnothing$, of each element in the class and writing the result as a sum of classes. Here, $std(w)$ sends $w$ to the unique word with the same relative order using all letters $\{1,2,\ldots,n\}$ for some $n$. For example, $std(13375)=12243$. Using this, we compute, $$\Delta([[12]])=[[\varnothing]]\otimes [[12]]+[[1]]\otimes[[1]]+[[1]]\otimes[[12]]+[[12]]\otimes[[1]]+[[12]]\otimes[[\varnothing]].$$ Letting $g=[1]$ and 
$f=[1]$, we create a Pieri construction using Theorem \ref{thm:pieri}. The partial dual filtered graph is shown below. The first five ranks are shown completely, and 
there are a few additional elements shown in the graph on the right to illustrate all elements that cover $[212]$. To see that there is an upward edge from 
$[3412]$ to $[3124]$, we first notice that $34124 \in [3124]$ and that $3412\otimes 1$ appears in the coproduct of 34124.

In the figure below, $K$-Knuth equivalence classes of words are represented by an increasing tableau. Note that there may be more than one increasing tableau in any given class.

\begin{center}
\begin{tikzpicture}[scale=.9]
 \node (empty) at (0,1) {$\varnothing$};
 \node (1) at (0,2) {$\tiny\young(1)$};
 \node (21) at (-2,4) {$\tiny\young(1,2)$};
 \node (12) at (2,4) {$\tiny\young(12)$};
 \node (321) at (-4,6) {$\tiny\young(1,2,3)$};
 \node (213) at (-2,6) {$\tiny\young(13,2)$};
 \node (212) at (0,6) {$\tiny\young(12,2)$};
 \node (312) at (2,6) {$\tiny\young(12,3)$};
 \node (123) at (4,6) {$\tiny\young(123)$};
\node (4321) at (-9,9) {$\tiny\young(1,2,3,4)$};
\node (3214) at (-8,9) {$\tiny\young(14,2,3)$};
 \node (4213) at (-7,9){$\tiny\young(13,2,4)$};
\node (4312) at (-6,9) {$\tiny\young(12,3,4)$};
\node (3213) at (-5,9) {$\tiny\young(13,2,3)$};
\node (3212) at (-4,9) {$\tiny\young(12,2,3)$};
 \node (3412) at (-3,9){$\tiny\young(12,34)$};
 \node (2413) at (-2,9) {$\tiny\young(13,24)$};
 \node (2312) at (-.5,9) {$\tiny\young(12,23)$};
 \node (3123) at (1,9) {$\tiny\young(123,3)$};
 \node (2134) at (2.5,9) {$\tiny\young(134,2)$};
 \node (3124) at (4,9) {$\tiny\young(124,3)$};
 \node (4123) at (5.5,9) {$\tiny\young(123,4)$};
 \node (2123) at (7,9) {$\tiny\young(123,2)$};
 \node (1234) at (8.5,9) {$\tiny\young(1234)$};
 \draw (empty)--(1)--(21)--(321)--(4321) (21)--(213) (21)--(212) (21) to [bend right =10] (312) (12)--(212) (12)--(312) (12)--(123) (12) to [bend left =10] (213) (1)--(12)--(123)--(1234)
 (321)--(3214) (321)--(4213) (321)--(4312) (321)--(3213) (321)--(3212) (213)--(3214) (213)--(3412) (213)--(2413) (213)--(2312) (213)--(2134) (312)--(4213) (312)--(4312) (312)--(3212) (312)--(3123)
 (312)--(3124) (312)--(4123) (123)--(3123) (123)--(2134) (123)--(3124) (123)--(4123) (123)--(2123) (123)--(1234) (212)--(3213) (212)--(2123)
 (1) to [out=60,in=100,distance=.5cm] (1)
(12) to [out=-60,in=-100,distance=.5cm] (12)
(21) to [out=-60,in=-100,distance=.5cm] (21)
(123) to [out=-60,in=-100,distance=.5cm] (123)
(321) to [out=-60,in=-100,distance=.5cm] (321)
 (213) to [out=60,in=100,distance=.5cm] (213)
  (213) to [out=60,in=100,distance=.6cm] (213)
  (212) to [out=-60,in=-100,distance=.5cm] (212)
  (212) to [out=-60,in=-100,distance=.6cm] (212)
  (312) to [out=60,in=100,distance=.5cm] (312)
   (3214) to [out=60,in=100,distance=.5cm] (3214)
  (3214) to [out=60,in=100,distance=.6cm] (3214)
   (4213) to [out=60,in=100,distance=.5cm] (4213)
  (4213) to [out=60,in=100,distance=.6cm] (4213)
  (3213) to [out=60,in=100,distance=.5cm] (3213)
  (3213) to [out=60,in=100,distance=.6cm] (3213)
  (3212) to [out=60,in=100,distance=.5cm] (3212)
  (3212) to [out=60,in=100,distance=.6cm] (3212)
  (3123) to [out=60,in=100,distance=.5cm] (3123)
  (3123) to [out=60,in=100,distance=.6cm] (3123)
  (2134) to [out=60,in=100,distance=.5cm] (2134)
  (2134) to [out=60,in=100,distance=.6cm] (2134)
  (3124) to [out=60,in=100,distance=.5cm] (3124)
  (3124) to [out=60,in=100,distance=.6cm] (3124)
  (2123) to [out=60,in=100,distance=.5cm] (2123)
  (2123) to [out=60,in=100,distance=.6cm] (2123)
  (4321) to [out=60,in=100,distance=.5cm] (4321)
  (4312) to [out=60,in=100,distance=.5cm] (4312)
   (3412) to [out=60,in=100,distance=.5cm] (3412)
  (2413) to [out=60,in=100,distance=.5cm] (2413)
   (2312) to [out=60,in=100,distance=.5cm] (2312)
  (4123) to [out=60,in=100,distance=.5cm] (4123)
   (1234) to [out=60,in=100,distance=.5cm] (1234);
 \path[->](212) edge node[right]{$ $} (312);
\path[->] (3213) edge [bend right=70] (4213);
\path[->] (3412) edge [bend left =40] (3124);
\end{tikzpicture}
\end{center}
\begin{minipage}{.5\textwidth}
\begin{center}
\begin{tikzpicture}[scale=.9]
 \node (empty) at (0,1) {$\varnothing$};
 \node (1) at (0,2) {$\tiny\young(1)$};
 \node (21) at (-2,3) {$\tiny\young(1,2)$};
 \node (12) at (2,3) {$\tiny\young(12)$};
 \node (321) at (-4,5) {$\tiny\young(1,2,3)$};
 \node (213) at (-2,5) {$\tiny\young(13,2)$};
 \node (212) at (0,5) {$\tiny\young(12,2)$};
 \node (312) at (2,5) {$\tiny\young(12,3)$};
 \node (123) at (4,5) {$\tiny\young(123)$};
\node (4321) at (-9,7) {$\tiny\young(1,2,3,4)$};
\node (3214) at (-8,7) {$\tiny\young(14,2,3)$};
 \node (4213) at (-7,7){$\tiny\young(13,2,4)$};
\node (4312) at (-6,7) {$\tiny\young(12,3,4)$};
\node (3213) at (-5,7) {$\tiny\young(13,2,3)$};
\node (3212) at (-4,7) {$\tiny\young(12,2,3)$};
 \node (3412) at (-3,7){$\tiny\young(12,34)$};
 \node (2413) at (-2,7) {$\tiny\young(13,24)$};
 \node (2312) at (-.5,7) {$\tiny\young(12,23)$};
 \node (3123) at (1,7) {$\tiny\young(123,3)$};
 \node (2134) at (2.5,7) {$\tiny\young(134,2)$};
 \node (3124) at (4,7) {$\tiny\young(124,3)$};
 \node (4123) at (5.5,7) {$\tiny\young(123,4)$};
 \node (2123) at (7,7) {$\tiny\young(123,2)$};
 \node (1234) at (8.5,7) {$\tiny\young(1234)$};
 \node (32312) at (-2,8) {$\tiny\young(12,23,3)$};
 \node (23123) at (-.5,8) {$\tiny\young(123,23)$};
 \node (32123) at (2,8) {$\tiny\young(123,2,3)$};
 \node (323123) at (0,9) {$\tiny\young(123,23,3)$};
 \draw (empty)--(1)--(21)--(321)--(4321) (321)--(3214) (21) to [bend left=40] (3213) (21)--(213) (1)--(212)--(3212) (212)--(2123) (212)--(2312) (12)--(312) 
 (12) to [bend right=40] (3123) 
 (12)--(123) (1)--(12)
 (213)--(2134) (213)--(2413) (213)--(4213) (312)--(3124) (312)--(3412) (312)--(4312) (123)--(1234) (123)--(4123)
 (212)--(32312) (212) to [bend left=40] (23123) (212) to [bend right =10] (32123) (212) to [bend right =15] (323123);
\end{tikzpicture}
\end{center}
\end{minipage}
\end{example}

\subsection{Malvenuto-Reutenauer deformations} \label{sec:MR}

\begin{example}\label{ex:PermTrees}
 Below we have dual graphs where vertices are permutations. On the left, a permutation $\sigma$ covers $\pi$ if $\pi$ is obtained from $\sigma$ by deleting
 the the rightmost number (in terms of position in the permutation). On the right, a permutation $\sigma$ covers $\pi$ if $\pi$ is obtained from $\sigma$ 
 by deleting the largest number in the permutation. There are other similar constructions, which can be found in \cite[Example 2.6.8]{F}.
 
\begin{minipage}{.5\textwidth}
\begin{center}
\begin{tikzpicture}[scale=1]
\node (empty) at (0,0) {$\varnothing$};
\node (1) at (0,1) {1};
\node (21) at (-2,3) {21};
\node (12) at (2,3) {12};
\node (321) at (-3.5,5) {321};
\node (213) at (-2,5) {213};
\node (312) at (-.5,5) {312};
\node (123) at (3.5,5) {123};
\node (132) at (2,5) {132};
\node (231) at (.5,5) {231};
\draw (empty)--(1)--(21)--(321) (312)--(21)--(213) (1)--(12)--(123) (231)--(12)--(132);
\end{tikzpicture}
\end{center}
\end{minipage}
\begin{minipage}{.5\textwidth}
\begin{center}
\begin{tikzpicture}[scale=1]
\node (empty) at (0,0) {$\varnothing$};
\node (1) at (0,1) {1};
\node (21) at (-2,3) {21};
\node (12) at (2,3) {12};
\node (321) at (-3.5,5) {321};
\node (213) at (-2,5) {213};
\node (312) at (-.5,5) {312};
\node (123) at (3.5,5) {123};
\node (132) at (2,5) {132};
\node (231) at (.5,5) {231};
\draw (empty)--(1)--(21)--(321) (231)--(21)--(213) (1)--(12)--(123) (312)--(12)--(132);
\end{tikzpicture}
\end{center}
\end{minipage}

Consider the Hopf algebra of permutations with the shuffle product and coproduct defined by $\Delta(w)=\sum std(u)\otimes std(v)$, where the concatenation of $u$ and $v$ is $w$ and $std(w)$ sends $w$ to the unique permutation with the same relative order. For example, $std(1375)=1243$.
Take $D$ to be the operator that deletes the rightmost letter of the permutation.
Then we see that $a_1(v,w)$ is the coefficient of $v$ in $D(w)$. We create a Pieri deformation by letting $f=1+12+123+\ldots$ and 
$a_2(v,w)$ be the coefficient of $w$ in $f\cdot v$. Clearly $D(f)=\varnothing+f$, and we use Lemma \ref{lem:derivation} to see tha $D$ is a derivation.
It then follows from Theorem \ref{thm:pieri} that the resulting graph is a dual filtered graph.

\begin{minipage}{.5\textwidth}
\begin{center}
\begin{tikzpicture}[scale=1]
\node (empty) at (0,0) {$\varnothing$};
\node (1) at (0,1) {1};
\node (21) at (-2,3) {21};
\node (12) at (2,3) {12};
\node (321) at (-3.5,5) {321};
\node (213) at (-2,5) {213};
\node (312) at (-.5,5) {312};
\node (123) at (3.5,5) {123};
\node (132) at (2,5) {132};
\node (231) at (.5,5) {231};
\draw (empty)--(1)--(21)--(321) (312)--(21)--(213) (1)--(12)--(123) (231)--(12)--(132);
\end{tikzpicture}
\end{center}
\end{minipage}
\begin{minipage}{.5\textwidth}
\begin{center}
\begin{tikzpicture}[scale=1]
\node (empty) at (0,0) {$\varnothing$};
\node (1) at (0,1) {1};
\node (21) at (-2,3) {21};
\node (12) at (2,3) {12};
\node (321) at (-3.5,5) {321};
\node (213) at (-2,5) {213};
\node (312) at (-.5,5) {312};
\node (123) at (3.5,5) {123};
\node (132) at (2,5) {132};
\node (231) at (.5,5) {231};
\draw (empty)--(1)--(21)--(321) (231)--(21)--(213) (1)--(12)--(123) (312)--(12)--(132) (12)--(empty) (213)--(1) (231)--(1) (123) to [bend right=5] (1)
(123) to [bend left=5] (empty);
\end{tikzpicture}
\end{center}
\end{minipage}
\end{example}

\begin{example}\label{ex:mMR}
We next describe a $K$-theoretic analogue of the Malvenuto-Reutenauer Hopf algebra. For details, see \cite{LP}.
A small multi-permutation or $\mathfrak{m}$-permutation of $[n]$ is a word $w$ in the alphabet $1,2,\ldots,n$ such that no two consecutive letters in $w$ are equal.
Now let the small multi-Malvenuto-Reutenauer Hopf algebra, $\mathfrak{m}$MR be the free $\mathbb{Z}$-module of arbitrary $\mathbb{Z}$-linear combinations of multi-permutations. Recall the definition of the multishuffle product from Example \ref{ex:L}. 
Given two $\mathfrak{m}$-permutations $w=w_1\cdots w_k$ and $u=u_1\cdots u_l$, define their product to be the multishuffle product of $w$ with $u[n]$, where $w$ contains exactly the numbers $1,2,\ldots,n$ and 
$u[n]=(u_1+n)(u_2+n)\ldots(u_l+n)$. 

To define the coproduct, we must define the \textit{cuut coproduct}, $\blacktriangle(w)$, for any word $w=w_1\cdots w_k$. We have $\blacktriangle(w)=\varnothing\otimes w_1\cdots w_k+
w_1\otimes w_1\cdots w_k+w_1\otimes w_2\cdots w_k+w_1w_2\otimes w_2\cdots w_k+\ldots+w_1\cdots w_k\otimes w_k+w_1\cdots w_k\otimes \varnothing$. Let $st(w)$ send a word $w$ to the unique
$\mathfrak{m}$-permutation $u$ of the same length such that $w_i\leq w_j$ if and only if $u_i\leq u_j$ for each $1\leq i,j \leq l(w)$. Finally, define the coproduct in $\mathfrak{m}$MR by
$\Delta(w)=st(\blacktriangle(w))$.

Using $g=1$, $D=\xi\circ\Delta$, and $f=1$, we form the dual filtered graph partially shown below. Notice that $D(w_1\ldots w_k)=w_1\cdots w_k + w_1\cdots w_{k-1}$, so $D(f)=\varnothing+f$.

\begin{minipage}{.5\textwidth}
\begin{center}
\begin{tikzpicture}[scale=1]
\node (empty) at (0,0) {$\varnothing$};
\node (1) at (0,1) {1};
\node (12) at (-2,2) {12};
\node (21) at (2,2) {21};
\node (121) at (-3.5,4) {121};
\node (123) at (-2.5,4) {123};
\node (132) at (-1.5,4) {132};
\node (312) at (-.5,4) {312};
\node (213) at (.5,4) {213};
\node (231) at (1.5,4) {231};
\node (321) at (2.5,4) {321};
\node (212) at (3.5,4) {212};
\draw (empty)--(1)--(12)--(123) (12)--(132) (12)--(231) (12)--(121) (21)--(212) (21)--(312) (21)--(213) (1)--(21)--(321) (12)--(121) (21)--(212)
(1) to [out=60,in=100,distance=.5cm] (1)
(21) to [out=-60,in=-100,distance=.5cm] (21)
(12) to [out=-60,in=-100,distance=.5cm] (12)
(121) to [out=60,in=100,distance=.5cm] (121)
(212) to [out=60,in=100,distance=.5cm] (212)
(321) to [out=60,in=100,distance=.5cm] (321)
(132) to [out=60,in=100,distance=.5cm] (132)
(123) to [out=60,in=100,distance=.5cm] (123)
(312) to [out=60,in=100,distance=.5cm] (312)
(213) to [out=60,in=100,distance=.5cm] (213)
(231) to [out=60,in=100,distance=.5cm] (231);
\end{tikzpicture}
\end{center}
\end{minipage}
\begin{minipage}{.5\textwidth}
\begin{center}
\begin{tikzpicture}[scale=1]
\node (empty) at (0,0) {$\varnothing$};
\node (1) at (0,1) {1};
\node (12) at (-2,2) {12};
\node (21) at (2,2) {21};
\node (121) at (-3.5,4) {121};
\node (123) at (-2.5,4) {123};
\node (132) at (-1.5,4) {132};
\node (312) at (-.5,4) {312};
\node (213) at (.5,4) {213};
\node (231) at (1.5,4) {231};
\node (321) at (2.5,4) {321};
\node (212) at (3.5,4) {212};
\draw (empty)--(1)--(12)--(123) (12)--(132) (12)--(312) (1)--(21)--(321) (21)--(231) (21)--(213)
(121) to [bend right =40] (1)
(212) to [bend left=40] (1);
\end{tikzpicture}
\end{center}
\end{minipage}
\end{example}

\subsection{Stand-alone examples}

\begin{example}\label{ex:FibonacciGraph}
 There is another graph with the same set of vertices as the Young-Fibonacci lattice called the Fibonacci graph. The vertices of the Fibonacci graph are words on the alphabet
 $\{1,2\}$, or snakeshapes, and a word $w$ is covered by $w'$ if $w'$ is obtained from $w$ by adding a 1 at the end or by changing any 1 into a 2. In the graph that is dual to the Fibonacci graph, $w$ is covered 
 by $w'$ if $w$ is obtained from $w'$ by deleting a 1, and the multiplicity of the edge between $w$ and $w'$ is the number of ways to delete a 1 from $w'$ to get $w$. 
 The pair form a dual graded graph shown below, see \cite[Example 2.3.7]{F}, where the numbers next to the edges denote multiplicity.
 
 \begin{minipage}{.5\textwidth}
\begin{center}
\begin{tikzpicture}[scale=1]
\node (empty) at (0,-1) {$\varnothing$};
\node (1) at (0,1) {1};
\node (11) at (0,3) {11};
\node (2) at (1.5,3) {2};
\node (111) at (0,5) {111};
\node (21) at (1.5,5) {21};
\node (12) at (3,5) {12};
\node (1111) at (0,7) {1111};
\node (211) at (1.5,7) {211};
\node (121) at (3,7) {121};
\node (112) at (4.5,7) {112};
\node (22) at (6,7) {22};
\draw (empty)--(1)--(11)--(111)--(1111)
(1)--(2) (11)--(12)--(22) (11)--(21)--(22) (111)--(211) (111)--(121) (111)--(112)
(2)--(21)--(211) (12)--(121);
\end{tikzpicture}
\end{center}
\end{minipage}
\begin{minipage}{.5\textwidth}
\begin{center}
\begin{tikzpicture}[scale=1]
\node (empty) at (0,-1) {$\varnothing$};
\node (1) at (0,1) {1};
\node (11) at (0,3) {11};
\node (2) at (1.5,3) {2};
\node (111) at (0,5) {111};
\node (21) at (1.5,5) {21};
\node (12) at (3,5) {12};
\node (1111) at (0,7) {1111};
\node (211) at (1.5,7) {211};
\node (121) at (3,7) {121};
\node (112) at (4.5,7) {112};
\node (22) at (6,7) {22};
\node (label2) at (-.2,2) {\tiny{2}};
\node (label3) at (-.2,4) {\tiny{3}};
\node at (-.2,6) {\tiny{4}};
\node at (1.3,6) {\tiny{2}};
\node at (3.95,6) {\tiny{2}};
\draw (empty)--(1)--(11)--(111)--(1111) (2)--(21)--(211) (2)--(12)--(112) (21)--(121)--(12);
\end{tikzpicture}
\end{center}
\end{minipage}
 
 We construct a Pieri deformation of this dual graded graph by adding downward-oriented edges so that in the resulting set of downward-oriented edges,
 there is an edge from $w'$ to $w$ for every way $w$ can be obtained from $w'$ by deleting at least one 1. For example, there are six edges from 1111 to 11 since there are six ways 
 to delete two 1's from 1111 to obtain 11.
 
  \begin{minipage}{.5\textwidth}
\begin{center}
\begin{tikzpicture}[scale=1]
\node (empty) at (0,-1) {$\varnothing$};
\node (1) at (0,1) {1};
\node (11) at (0,3) {11};
\node (2) at (1.5,3) {2};
\node (111) at (0,5) {111};
\node (21) at (1.5,5) {21};
\node (12) at (3,5) {12};
\node (1111) at (0,7) {1111};
\node (211) at (1.5,7) {211};
\node (121) at (3,7) {121};
\node (112) at (4.5,7) {112};
\node (22) at (6,7) {22};
\draw (empty)--(1)--(11)--(111)--(1111)
(1)--(2) (11)--(12)--(22) (11)--(21)--(22) (111)--(211) (111)--(121) (111)--(112)
(2)--(21)--(211) (12)--(121);
\end{tikzpicture}
\end{center}
\end{minipage}
\begin{minipage}{.5\textwidth}
\begin{center}
\begin{tikzpicture}[scale=1]
\node (empty) at (0,-1) {$\varnothing$};
\node (1) at (0,1) {1};
\node (11) at (0,3) {11};
\node (2) at (1.5,3) {2};
\node (111) at (0,5) {111};
\node (21) at (1.5,5) {21};
\node (12) at (3,5) {12};
\node (1111) at (0,7) {1111};
\node (211) at (1.5,7) {211};
\node (121) at (3,7) {121};
\node (112) at (4.5,7) {112};
\node (22) at (6,7) {22};
\node (label2) at (-.2,2) {\tiny{2}};
\node (label3) at (-.2,4) {\tiny{3}};
\node at (-.2,6) {\tiny{4}};
\node at (1.3,6) {\tiny{2}};
\node at (3.5,6) {\tiny{2}};
\node at (-.75,5) {\tiny{6}};
\node at (-.75,3) {\tiny{3}};
\node at (.8,4.5) {\tiny{4}};
\draw (empty)--(1)--(11)--(111)--(1111) (2)--(21)--(211) (2)--(12)--(112) (21)--(121)--(12) (121)--(2)
(11) to [bend right=30] (empty)
(111) to [bend left=40] (empty)
(111) to [bend right=30] (1)
(1111) to [bend right=40] (empty)
(1111) to [bend left=30] (1)
(1111) to [bend right=40] (11)
(211) to [bend left=20] (2)
(112) to [bend left=20] (2);
\end{tikzpicture}
\end{center}
\end{minipage}
 
\begin{theorem}
 The resulting graph is a dual filtered graph. 
\end{theorem}
\begin{proof}
 Consider an algebra structure on words in alphabet $\{1,2\}$ with shuffle product, as in \cite{F}. Consider $D$ to be the operator that changes any $2$ 
into a $1$ or deletes the last digit if it is $1$. According to \cite[Lemma 2.3.9]{F}, $D$ is a derivation such that the coefficient of $v$ in $D(w)$ is the multiplicity 
$a_1(v,w)$ for the Fibonacci graph. Take $f = 1 + 11 + 111 + \ldots$. It is easy to see that $D(f) = id + f$ and that the coefficient of $w$ in $fv$ is the multiplicity 
$a_2(v,w)$. Thus, Theorem \ref{thm:pieri} applies, and the statement follows. 
\end{proof}
\end{example}

\begin{example}\label{ex:polys} The dual filtered graph on the polynomial ring shown below is that mentioned in Section \ref{sec:differential}. We take $U$ to be multiplication by $x$ and $D=e^{\frac{\partial}{\partial x}}-1$. It is easy to check that $DU-UD=D+I$, so the result is a dual filtered graph.

\begin{center}
\begin{minipage}{.05\textwidth}
\begin{center}
\begin{tikzpicture}
\node (1) at (0,0) {1};
\node (x) at (0,2) {$x$};
\node (x2) at (0,4) {$x^2$};
\node (x3) at (0,6) {$x^3$};
\node (x4) at (0,8) {$x^4$};
\node (x5) at (0,10) {$x^5$};
\draw (1)--(x)--(x2)--(x3)--(x4)--(x5);
\end{tikzpicture}
\end{center}
\end{minipage}\hspace{2in}
\begin{minipage}{.05\textwidth}
\begin{center}
\begin{tikzpicture}
\node (1) at (0,0) {1};
\node (x) at (0,2) {$x$};
\node (x2) at (0,4) {$x^2$};
\node (x3) at (0,6) {$x^3$};
\node (x4) at (0,8) {$x^4$};
\node (x5) at (0,10) {$x^5$};
\node at (-.1,9) {\tiny{5}};
\node at (-.1,7) {\tiny{4}};
\node at (-.1,5) {\tiny{3}};
\node at (-.1,3) {\tiny{2}};
\node at (.6,4) {\tiny{3}};
\node at (-.55,6) {\tiny{6}};
\node at (-1.1,5) {\tiny{4}};
\node at (.57,7) {\tiny{10}};
\node at (1.05,6) {\tiny{10}};
\node at (1.7,5) {\tiny{5}};
\draw (1)--(x)--(x2)--(x3)--(x4)--(x5)
(x2) to [bend right=20] (1) (x3) to [bend left=30] (1)
(x3) to [bend left=20] (x)
(x4) to [bend right =20] (x2) (x4) to [bend right =30] (x) (x4) to [bend right =40] (1)
(x5) to [bend left =20] (x3)
(x5) to [bend left =30] (x2)
(x5) to [bend left =40] (x)
(x5) to [bend left =50] (1);
\end{tikzpicture}
\end{center}
\end{minipage}
\end{center}
\end{example}

\section{Enumerative theorems via up-down calculus} \label{sec:enum}

Let $\varnothing$ be the minimal element of a dual filtered graph satisfying $DU-UD=1+D$. Let $T(n,k)$ be 
the number of ways $n$ labeled objects can be distributed into $k$ nonempty parcels. We have 
$$T(n,k) = k! \cdot S(n,k),$$ where $S(n,k)$ is the Stirling number of the second kind. 

\begin{theorem}
For any dual filtered graph, the coefficient of $\varnothing$ in $D^k U^n (\varnothing)$ is $T(n,k)$.  
\end{theorem}

\begin{proof}
 We think of replacing the fragments $DU$ inside the word by either $UD$, $D$, or $1$. This way the initial word gets rewritten until there is no $DU$ in any of the terms:
$$D^k U^n = D^{k-1} U D U^{n-1} + D^k U^{n-1} + D^{k-1} U^{n-1} = \ldots.$$ Only the terms of the form $U^t$ that appear at the end can contribute to the coefficient 
we are looking for, since $D (\varnothing) = 0$. Among those, only the terms $U^0=1$ can contribute. Thus, we are looking for all the terms where $D$'s and $U$'s eliminated each other. 

It is easy to see that each $D$ eliminates at least one $U$, and each $U$ must be eliminated by some $D$. The number of ways to match $D$'s with $U$'s in this way is exactly 
$T(n,k)$.
\end{proof}

Let $f^{\lambda}$ be the number of increasing tabeleaux of shape $\lambda$. Let $g^{\lambda}_n$ be the number of set-valued tableaux of shape $\lambda$ and 
content $1, \ldots, n$. Let $F(n)$ denote the {\it {Fubini number}}, or {\it {ordered Bell number}} - the number of ordered set partitions of $[n]$.

\begin{corollary}
 We have 
$$\sum_{|\lambda| \leq n} f^{\lambda} g^{\lambda}_n = F(n).$$
\end{corollary}

\begin{proof}
 The left side is clearly the coefficient of $\varnothing$ in $(D + D^2 + \dotsc + D^n) U^n (\varnothing)$. It remains to note that $\sum_{k=1}^n T(n,k) = F(n)$.
\end{proof}

This is the analogue of the famous Frobenius-Young identity. Of course, this also follows from bijectivity of Hecke insertion. The advantage of our proof is that a similar 
result exists for any dual filtered graph. 

The following result is analogous to counting oscillating tableaux.

\begin{theorem}
For any dual filtered graph the coefficient of $\varnothing$ in $(D+U)^n (\varnothing)$ is equal to the number of set partitions of $[n]$ with parts of size at least $2$.  
\end{theorem}

\begin{proof}
 As before, the desired coefficient is equal to the number of ways for all $D$'s to eliminate all $U$'s via the commutation relation. Each factor in the product 
$$(D+U)(D+U)\dotsc(D+U)$$ thus either eliminates one of the factors to the right of it, or is eliminated by a factor to the left. Grouping together such factors, we get a set 
partition with parts of size at least $2$. On the other hand, any such set partition corresponds to a choice of $D$ in the first factor of each part and of $U$'s in the rest.
Thus, it corresponds to term $1$ after such $D$'s eliminate such $U$'s. The statement follows.
\end{proof}

\end{document}